\documentclass[12pt]{amsart}

\usepackage{amsmath}
\usepackage{amssymb}
\usepackage{amsthm}
\usepackage{bm}
\usepackage{mathrsfs}
\usepackage{hyperref}
\usepackage{graphicx}
\usepackage{xcolor}
\usepackage{enumitem}
\setlist[enumerate]{parsep=0pt plus 4pt,topsep=0pt plus 4pt}
\usepackage{url}
\usepackage{mathtools}
\allowdisplaybreaks
\usepackage{epsfig}
\usepackage{xspace}

\newcommand\excise[1]{}

\newtheorem{theorem}{Theorem}
\newtheorem{thm}{Theorem}[section]
\newtheorem{lemma}[thm]{Lemma}
\newtheorem{prop}[thm]{Proposition}

\newtheorem{conj}[thm]{Conjecture}
\newtheorem{cor}[thm]{Corollary}

\newtheorem{claim}[thm]{Claim}
\newtheorem*{claim*}{Claim}
\theoremstyle{definition}
\newtheorem{defn}[thm]{Definition}
\newtheorem{remark}[thm]{Remark}
\newtheorem{example}[thm]{Example}

\newtheorem{notation}[thm]{Notation}

\numberwithin{equation}{section}

\newcounter{separated-sec}
\newcounter{collapse-sec}

\newcommand{\Ring}[1]{\ensuremath{\mathbb{#1}}}

\renewcommand\>{\rangle}
\newcommand\<{\langle}

\newcommand\BB{\mathfrak{B}}
\newcommand\DD{\mathbb{D}}
\newcommand\EE{\mathbb{E}}
\newcommand\HH{\mathcal{H}}

\newcommand\LL{\mathcal{L}}
\newcommand\MM{\mathcal{M}}
\newcommand\NN{\Ring{N}}

\newcommand\RR{\Ring{R}}

\newcommand\XX{\mathcal{X}}
\newcommand\bb{\mathfrak{b}}
\newcommand\dd{\mathbf{d}}

\newcommand\xx{\mathbf{x}}
\newcommand\bP{\mathbf{P}}

\newcommand\cC{\mathcal{C}}
\newcommand\cF{\mathcal{F}}

\newcommand\cI{\mathcal{I}}
\newcommand\cN{\mathcal{N}}
\newcommand\cP{\mathcal{P}}
\newcommand\cR{\mathcal{R}}
\newcommand\cS{\mathcal{S}}
\newcommand\fC{\mathfrak{C}}
\newcommand\oB{{\hspace{.2ex}\overline{\hspace{-.2ex}B}\hspace{.2ex}}}
\newcommand\oC{{\hspace{.2ex}\overline{\hspace{-.2ex}C}\hspace{.2ex}}}

\newcommand\oX{{\hspace{.2ex}\overline{\hspace{-.2ex}X}\hspace{.2ex}}}
\newcommand\tF{\widetilde{F}}

\newcommand\ta{\widetilde{a}}

\newcommand\ud{d}
\newcommand\td{\widetilde{d}}
\newcommand\ve{\varepsilon}

\newcommand\bmu{{\bar\mu}}
\newcommand\bnu{{\bar\nu}}
\newcommand\hmu{{\widehat\mu}}

\newcommand\Fd{F_\delta}
\newcommand\Gn{G_n}
\newcommand\IZ{\cI(Z)}

\newcommand\Ti{T_\infty}
\newcommand\ev{\mathscr{E}}

\newcommand\Cmu{C_{\hspace{-.2ex}\mu}}
\newcommand\Emu{E_\mu}
\newcommand\Fdn{F_{\delta_n}}
\newcommand\Fmu{F}
\newcommand\Gmu{\Gamma_{\hspace{-.1ex}\mu}} 
\newcommand\Rmu{\ol{\RR_+\mu}}
\newcommand\Smu{S_\bmu\MM}

\newcommand\Tmu{T_\bmu\MM}

\newcommand\Ups{\Upsilon} 
\newcommand\Upt{\noheight{$\widetilde\Upsilon$}} 
\newcommand\bPi{\bP_\infty}
\newcommand\bgn{\ol g_n}

\newcommand\evn{\ev_n}

\newcommand\muL{\mu^\LL}

\newcommand\Chmu{C_{\hspace{-.2ex}\hmu}}
\newcommand\Cmue{C_{\hspace{-.2ex}\mu}^{\hspace{.05ex}e}}

\newcommand\Emue{E_\mu^{\hspace{.1ex}e}}

\newcommand\Fdnp{F_{\delta^{\hspace{.1ex}\prime}_n}}
\newcommand\Lmui{\LL^{-1}_{|\mu|}}
\newcommand\NLmu{N_{\muL}}
\newcommand\SEmu{S\hspace{-.2ex}\Emu} 

\newcommand\Tmue{T^{\hspace{.1ex}e}_\bmu\MM}
\newcommand\Tmus{T^{\hspace{.1ex}s}_\bmu\MM}
\newcommand\Tmux{T_\bmu\XX}

\newcommand\into{\hookrightarrow}

\newcommand\oCmu{\oC_{\hspace{-.45ex}\mu}}
\newcommand\oXnL{\oX{\hspace{-.2ex}}_n^\LL}
\newcommand\onto{\twoheadrightarrow}

\newcommand\ttdn{\noheight{$\theta_{\td_n}$}}
\newcommand\tudn{\noheight{$\theta_{\ud_n}$}}
\newcommand\wXnL{X_n^\LL}

\newcommand\oCmue{\oC{}_{\hspace{-.2ex}\mu}^{\hspace{.05ex}e}}
\newcommand\oSCmu{S\oCmu} 
\newcommand\pDelta{\Delta'}
\newcommand\pdelta{\delta^{\hspace{.2ex}\prime}}

\newcommand\nablaq{\nabla_{\hspace{-.25ex}q\hspace{.25ex}}}
\newcommand\nablamu{\nabla_{\hspace{-.25ex}\mu}}
\newcommand\nablatF{\nabla_{\hspace{-.4ex}\tF}}
\newcommand\ppDelta{\Delta''}
\newcommand\nablabmu{\nabla_{\hspace{-.25ex}\bmu}}

\renewcommand\implies{\Rightarrow}

\newcommand\ol[1]{{\overline{#1}}}

\newcommand\wt[1]{{\widetilde{#1}}}

\newcommand\interior[1]{{\kern0pt#1}^{\mathrm{o}}}

\DeclareMathOperator*\argmax{argmax}
\DeclareMathOperator*\argmin{argmin}
\DeclareMathOperator\CAT{CAT}
\DeclareMathOperator\Cov{Cov} 

\DeclareMathOperator\Disc{Disc}

\DeclareMathOperator\hull{hull}
\DeclareMathOperator\supp{supp}

\DeclareMathOperator\tangCOV{\Sigma}
\DeclarePairedDelimiter\abs{\lvert}{\rvert}

\newcommand\noheight[1]{\raisebox{0pt}[0pt][0pt]{#1}}

\definecolor{teal}{HTML}{029386}
\begin{document}

\mbox{}\vspace{-3ex}
\title[Central limit theorems for Fr\'echet means on stratified spaces]%
{Central limit theorems for Fr\'echet means\\on stratified spaces}%
\author{Jonathan C. Mattingly}
\address{\rm Departments of Mathematics and of Statistical Sciences,
Duke University, Durham, NC 27708}
\urladdr{\url{https://scholars.duke.edu/person/jonathan.mattingly}\vspace{-1ex}}
\author{Ezra Miller}
\address{\rm Departments of Mathematics and of Statistical Sciences,
Duke University, Durham, NC 27708}
\urladdr{\url{https://scholars.duke.edu/person/ezra.miller}\vspace{-1ex}}
\author[Do Tran]{Do Tran}%
\address{\rm Georg-August Universit\"at at G\"ottingen, Germany\vspace{-1ex}}

\makeatletter
  \@namedef{subjclassname@2020}{\textup{2020} Mathematics Subject Classification}
\makeatother

\date{11 November 2023}

\begin{abstract}
Fr\'echet means of samples from a probability measure~$\mu$ on any
smoothly stratified metric space~$\MM$ with curvature bounded above
are shown to satisfy a central limit theorem (CLT).  The methods and
results proceed by introducing and proving analytic properties of the
\emph{escape vector} of any finitely supported measure~$\delta$
in~$\MM$, which records infinitesimal variation of the Fr\'echet
mean~$\bmu$ of~$\mu$ in response to perturbation of~$\mu$ by adding
the mass $t\delta$ for $t \to 0$.  The CLT limiting distribution~$\cN$
on the tangent cone~$\Tmu$ at the Fr\'echet mean~$\bmu$ is
characterized in four ways.  The first uses tangential collapse~$\LL$
to compare $\Tmu$ with a linear space and then applies a distortion
map to the usual linear CLT to transfer back to $\Tmu$.  Distortion
is defined by applying escape after taking preimages under~$\LL$.  The
second characterization constructs singular analogues of Gaussian
measures on smoothly stratified spaces and expresses~$\cN$ as the
escape vector of any such \emph{Gaussian mass}.  The third
characterization expresses~$\cN$ as the directional derivative, in the
space of measures on~$\MM$, of the barycenter map at~$\mu$ in the
(random) direction given by any Gaussian mass.  The final
characterization expresses~$\cN$ as the directional derivative, in the
space~$\cC$ of continuous real-valued functions on~$\Tmu$, of a
minimizer map, with the derivative taken at the Fr\'echet function $F
\in \cC$ along the (random) direction given by the negative of the
Gaussian tangent field induced by~$\mu$.  Precise mild hypotheses on
the measure~$\mu$ guarantee these CLTs, whose convergence is proved
via the second characterization of~$\cN$ by formulating a duality
between Gaussian masses and Gaussian tangent fields.
\end{abstract}

\maketitle
\vspace{-1.5ex}
\tableofcontents
\vspace{-3ex}

\section*{Introduction}\label{s:intro}

\addtocontents{toc}{\protect\setcounter{tocdepth}{1}}
\subsection*{Overview}

The central limit theorem (CLT) in~$\RR^d$ records how the averages of
size~$n$ samples are asymptotically distributed as $n$ grows and the
differences between the population and empirical averages are rescaled
by a factor of~$\sqrt n$.  In nonlinear intrinsic geometric settings,
averages are replaced by Fr\'echet means, which are expected
square-distance minimizers \cite{frechet1948}.  Laws of large numbers
for Fr\'echet means have been known for nearly half a century in great
generality for samples from separable metric spaces
\cite{ziezold1977}.  However, geometric CLTs have been much less
forthcoming, with the first results in smooth settings arriving in the
early 2000s \cite{bhattacharya-patrangenaru2003,
bhattacharya-patrangenaru2005}.  The few CLTs available in singular
(that is, non-smooth) contexts occur on spaces built from flat pieces
glued together in combinatorial ways \cite{hotz-et-al.2013,
barden-le-owen2013, kale-2015, barden-le-owen2018, barden-le2018}.  Of
these, only \cite{barden-le2018}, whose CLT occurs on nonnegatively
curved spaces glued from right-angled orthants,
treats phenomena at singularities in codimension~more~than~$2$.

General stratified sample spaces are considered explicitly in prior
research; see \cite[Section~2.1]{eltzner-huckemann2019} and
\cite{harms-michor-pennec-sommer2020} for apt, comprehensive reviews
of the history and literature, including (in the former case) a
summary of hypotheses that various known CLTs assume.  However,
non-smooth variation of Fr\'echet means and more general
descriptors have until now been skirted, as pointed out in
\cite[Remark~2.8(ii)]{eltzner-huckemann2019}, by either assuming or
proving that the relevant descriptor lies in a smooth space---or at
least in a smooth part of the space of descriptors.  The goal here is
to face the singularities head on, treating truly singular variation
of empirical Fr\'echet means in a general geometric context.  Even
deciphering the form taken by a CLT in the singular setting is a major
theoretical stride.

Smooth CLTs proceed by approximating a manifold with its tangent space
at the Fr\'echet mean.  The known singular CLTs reduce to linear cases
by gluing flat pieces.  What should a CLT look like on general
non-smooth spaces, where tangent spaces are not linear and no
combinatorial gluing or flat pieces are available?  Our main results
answer this question by
\begin{itemize}
\item%
highlighting appropriate classes of singular spaces and measures on
them;
\item%
identifying the ``Gaussian'' random variables to which suitably
rescaled empirical barycenters converge;
\item%
understanding convergence to these distributions; and
\item%
encapsulating how geometry of the space is reflected in the limiting
distribution
\end{itemize}
to formulate and prove the first general CLT for Fr\'echet means in
singular spaces.  This paper is the final in a series
whose prior three installments develop relevant convex geometry in
tangent cones of spaces with curvature bounded above
\cite{shadow-geom}, introduce suitable singular spaces and show how to
collapse them onto smooth spaces \cite{tangential-collapse}, and prove
preliminary central limiting convergence along tangent rays at
singular Fr\'echet means in these spaces \cite{random-tangent-fields}.

\subsection*{Motivations}

Our initial impetus to consider CLTs in singular spaces comes from
statistics, where increasingly complex (``big'') data are sampled from
increasingly \mbox{complex} spaces of data objects; see
\cite{huckemann-eltzner2020, ooda-2021}, for example.  The spaces
parametrize data objects such as phylogenetic and more general trees
\cite{BHV01, holmes2003, feragen-lo-et-al2013,
lin-sturmfels-tang-yoshida2017, lueg-garba-nye-huckemann2021}, shapes
\cite{le2001, kendall-barden-carne-le99}, positive semi-definite
matrices \cite{groisser-jung-schwartzman2017, buet-pennec2023}, or
arbitrary mixtures of linear subspaces or more general smooth spaces.
Applications arise in disciplines ranging from computer vision
\cite{hartley-trumpf-dai-li2013} to medical image analysis
\cite{pennec-sommer-fletcher2020, pritchard-stephens-donnelly2000},
for instance.  Despite the wealth of examples, sophisticated
statistical techniques based on geometry currently lack mathematical
foundations for probability and analysis in singular settings.
A~non-smooth CLT serves as a crucial first step.  \mbox{Quantitative}
bounds can then lead, for example, to confidence regions (generalizing
\cite{willis-2019} for treespace) or singularity hypothesis testing
by detecting whether asymptotics in a given dataset behave in smooth
or singular ways.
In addition, the inverse Hessian (see Section~\ref{b:distortion}) is
potentially key to generalizing the Cram\'er--Rao bound on variance in
parametric statistics involving families of distributions indexed by
manifolds
to much more general settings, perhaps including essentially
nonparametric settings in which
distributions can be selected from more or less arbitrary spaces of
measures.  This approach could have impact throughout theoretical and
applied statistics, including applications to topological data
analysis via the space of
persistence~diagrams~\cite{mileyko-mukherjee-harer2011}.

The statistical perspective leads naturally to fundamental
probabilistic problems, such as characterizing Gaussian objects in
singular settings: do they have alternative interpretations in terms
of infinite divisibility, heat dissipation, or random walks on
singular spaces \cite{brin-kifer2001, bouziane2005, nye-white2014,
nye2020}?  Again, the spaces in all of these studies of Brownian
motion are glued in combinatorial ways from flat pieces, so general
foundations are called for to enable unfettered investigations into
geometric probability of singularities.  Preliminary discussions
indicate that singular Gaussians via tangential collapse are pivotal
for proving that relevant processes are martingales
\cite{arya-mukherjee-2023}.

On the metric geometry side, local features induce curvature that
deforms the asymptotics of large samples.  In the smooth CLT
\cite{bhattacharya-patrangenaru2005}, curvature enters as a correction
upon passing from a Riemannian manifold to its tangent space at the
Fr\'echet mean, where the asymptotics naturally occur.  Turning this
observation on its head, understanding asymptotics of large samples
furnishes insight into local curvature invariants.  The metric
geometry of shadows \cite{shadow-geom}, tangential collapse
\cite{tangential-collapse}, and distortion (Section~\ref{b:clt}), for
instance, generalize the concept of angle deficit to singularities of
stratified $\CAT(\kappa)$ spaces which, when integrated, could yield
global topological information.  Similarly, any advance on Brownian
motion in singular spaces could have topological consequences, in
analogy to the smooth setting, where Brownian motion leads to a proof
of the Atiyah--Singer index theorem \cite{hsu-1988}.
Hence another key motivation for this paper is to initiate a program
to probe the geometry of singularities via asymptotics of sampling.

This program in singularity theory is especially tantalizing when the
singularties are algebraic, as resolution of singularities
\cite{hironaka1964} permits comparison of smooth and singular
settings.  It suggests a functorial central limit theory that pulls
back or pushes forward local sheaf-theoretic information around
Fr\'echet means under dominant algebraic maps to correct for
distortion, just as inverse Hessians on Riemannian manifolds (or the
distortion map in our singular setting here) corrects the passage from
a space to its tangent space.  The first step would be to show that
CLTs via desingularization agree with CLTs calculated directly on
singular spaces, via the methods here, when the spaces satisfy the
relevant smoothly stratified hypotheses.  Desingularization then
extends the theory to all varieties.  Such a theory stands in analogy
with how multiplier ideals encode their analytic incarnations in terms
of locally summable holomorphic germs \cite{nadel-1990}
algebraically as sections of $\mathbb{Q}$-divisors
\cite[Remark~5.9]{demailly-2012}.  Thinking further along algebraic
lines, when samples are taken from a moduli space, so sample points
represent algebraic varieties or schemes, a CLT would summarize
variation of average behavior in families.

\vspace{-.2ex}

\subsection*{Methods}

The leaps required to formulate and derive the singular CLT involve
multiple changes in
\pagebreak[4]
perspective which, taken together, constitute a powerful toolkit for
analyzing random variables valued in singular spaces.
These include especially
\begin{enumerate}
\item\label{i:shadow}%
comparison of the tangent cone at the population mean with the tangent
cones infinitesimally near the population mean
\cite{tangential-collapse} and application of the resulting tangential
collapse to push the CLT from a singular to a linear setting;
\item\label{i:inverseHessian}%
direct geometric construction of the inverse Hessian of the Fr\'echet
function (the \emph{distortion map}), given that the Hessian itself
becomes singular in this setting;
\item\label{i:ray-by-ray}%
interpretation of the CLT as convergence of random tangent fields to a
Gaussian random field ray by ray at the Fr\'echet mean
\cite{random-tangent-fields} followed by its transformation, through a
variational problem, to full spatial variation about
the~mean;
\item\label{i:escape}%
characterization of the limiting distribution of a measure~$\mu$ as
the distribution of means after perturbing~$\mu$ by adding a
singular-Gaussian random measure.
\end{enumerate}
To orient the reader, it may help to discuss these developments in
more detail.

To begin, even the transition from sample space to tangent object at
the population mean requires novel input in the singular context,
namely the introduction \cite{tangential-collapse} of a new class of
singular spaces suitable for central limit theorems.  These metric
spaces are optimally general: they\vspace{-.5ex}
\begin{itemize}
\item%
have curvature bounded above in the usual Alexandrov sense
\cite{BBI01};
\item%
are stratified in the usual topological sense
\cite{goresky-macpherson1988}; and
\item%
admit logarithm and exponential maps that are locally inverses of each
other, to enable transfer from the original space to the tangent cone.
\end{itemize}\vspace{-.5ex}
See Section~\ref{b:CAT(k)}--\ref{b:collapse} for a summary of the
salient points.%
\enlargethispage*{1.5ex}

Broadly speaking, the driving theoretical advance here reduces
variation of Fr\'echet means~$\bmu_n$ of samples~$\mu_n$ from a
population measure~$\mu$ to variation of the Fr\'echet mean~$\bmu$
of~$\mu$ itself upon perturbing~$\mu$ by adding an infinitesimal
finitely supported mass~$t\delta$ with $t \to 0$.  Behavior of the
resulting \emph{escape vector}~(\S\ref{b:intro-escape}) as a function
of~$\delta$ is what distinguishes the singular setting from the smooth
setting (where escape can be phrased in terms of influence functions;
see Remark~\ref{r:influence}).  Indeed, in linear settings, adding a
point mass~$\delta$ to~$\mu$ simply causes the mean~$\bmu$ to move
along a vector pointing toward the support of~$\delta$.  In smooth
settings, local perturbation of~$\mu$ in the tangent space behaves
similarly, but transferring that motion back from the tangent space to
the sample space incurs a distortion quantified by the inverse Hessian
of the Fr\'echet function to capture the average spread rate of
geodesics through~$\bmu$.  In the singular setting of smoothly
stratified spaces, the geometry is set up so that local perturbation
of~$\mu$ in the tangent cone makes sense.  But the Hessian is
singular: geodesics can branch at~$\bmu$.  The distortion inherent in
transfer from the tangent cone to the sample space therefore requires
direct construction of the inverse Hessian~(\S\ref{b:distortion})
rather than proceeding by way of the Hessian itself.

The geometry of distortion on smoothly stratified metric spaces is
distinctly singular in nature.  As might be expected from branching
geodesics, it collapses sectors containing sizable open sets; this is
the geometry of shadows \cite{shadow-geom} and tangential collapse
\cite{tangential-collapse}, which in turn form foundations for the
construction of Gaussian objects
in singular contexts~(\S\ref{b:gaussians}).  These \emph{Gaussian
{}masses} are measures on the tangent cone defined essentially as
sections (also known in probabilistic language as selections) of
linear Gaussian random variables under tangential collapse.

But, crucially, the limiting distribution in a singular CLT is not
Gaussian: rather, it is a type of convex projection of a Gaussian.  In
the smooth setting, since the tangent space at the Fr\'echet mean is
already linear, distortion is a linear map.  The lesson from the
singular setting is to interpret that linear map as escape.  Thus the
limiting distribution of~$\mu$ is the escape vector of the Gaussian
{}mass~$\Gmu$ constructed from~$\mu$ via tangential collapse
(Theorem~\ref{t:intro-CLT} in~\S\ref{b:gaussians}).

Beyond producing limiting distributions, passage from a Gaussian
{}mass~$\Gmu$ to its escape vector is integral to the proof of
the singular CLT.  Indeed, thinking of this passage ray by ray amounts
to minimizing the Fr\'echet function of~$\mu$ after perturbing it by a
Gaussian random tangent field~$G$: a family of univariate Gaussian
random variables indexed by the tangent cone at~$\mu$
\cite{random-tangent-fields}.  (Random tangent fields are reviewed in
Section~\ref{b:random-tangent-fields}.)  This perspective is made
precise by a duality between $\Gmu$ and~$G$ akin to Riesz
representation~(\S\ref{b:random-tangent-fields}).  The power of this
duality stems from the CLT for random tangent fields
\cite{random-tangent-fields} (Theorem~\ref{t:tangent-field-CLT}),
which translates, via duality and the continuous mapping theorem, into
a perturbative version of the CLT
(Theorem~\ref{t:intro-perturbation-CLT} in
\S\ref{b:via-random-fields}) from which all other versions of the CLT
follow.

Expressing the CLT via distortion (Theorem~\ref{t:intro-CLT}
in~\S\ref{b:stratified}) or Gaussian escape
(Theorem~\ref{t:intro-perturbation-CLT}
in~\S\ref{b:via-random-fields}) recovers a limiting distribution on a
Euclidean space at the cost of pushing forward to the tangent vector
space of an infinitesimally nearby point selected arbitrarily from the
relevant stratum whose closure contains the population mean.  Thinking
more deeply about the perturbation proof techniques leads naturally to
more refined and intrinsic interpretations of the CLT as assertions
concerning directional derivatives in spaces of measures or spaces of
functions~(\S\ref{b:directional}), which lend fresh insight even in
the classical multivariate Euclidean setting.

\subsection*{Outline}

The exposition in Section~\ref{s:core} provides a detailed overview of
the core results, methods, and motivations.  It starts with a review
of the usual linear CLT and its generalization to smooth manifolds,
viewed through a lens that casts the singular CLT as a natural
outgrowth, culminating in a first precise statement of the singular
CLT~\mbox{(\S\ref{b:linear}--\ref{b:stratified})}.  Subsequent
subsections of Section~\ref{s:core} serve as guides through the
geometric, analytic, and probabilistic ingredients of the statement
and proof of the singular
CLT~(\S\ref{b:intro-escape}--\ref{b:via-random-fields}) as well as its
interpretations via perturbation~(\S\ref{b:via-random-fields}) and as
directional derivatives in measure or function
spaces~(\S\ref{b:directional}).

\addtocontents{toc}{\protect\setcounter{tocdepth}{2}}
\subsection*{Acknowledgements}

The authors are grateful to Stephan Huckemann for numerous discussions
on topics related to this work, both at Duke and at G\"ottingen, and
\enlargethispage*{1.7ex}%
for funding many visits to enable these discussions.  DT~was partly
funded by DFG~HU 1575/7.  JCM thanks NSF RTG grant DMS-2038056 for
general support as well as~Davar Khoshnevisan and Sandra Cerrai for
enlightening conversations.  Amy Willis observed connections to
influence functions (Remark~\ref{r:influence}) during a
presentation~of escape vectors at an Institute for Mathematical and
Statistical Innovation
workshop in Chicago,~July~2023.


\section{Core ideas}\label{s:core}

\subsection{Linear setting}\label{b:linear}
\mbox{}\medskip

\noindent
In a vector space such as $\RR^m$, the law of large numbers (LLN) for
a collection of independent random variables $X_1,X_2,\dots$, all with
a common distribution $\mu$, states~that
\begin{align}\label{eq:vectorLLN}
  \bmu_n
  :=
  \frac1n \sum_{i=1}^n X_i\,
  \xrightarrow{n\to\infty\;}{}
  \!\int x \,\mu(dx)\,
  =:
  \bmu\
\text{ almost surely.}
\end{align}
Thus the \emph{empirical mean} $\bmu_n$ fluctuates more tightly around
the \emph{population mean} $\bmu$ with increasing sample size~$n$.
The Central Limit Theorem (CLT) asserts that expanding the
fluctuations (or, equivalently, zooming in) by a factor of~$\sqrt n$
causes them to stabilize to a fixed limiting distribution:
\begin{equation}\label{eq:vectorCLT}
  \sqrt n\,(\bmu_n - \bmu)
  =
  \frac1{\sqrt n}\sum_{i=1}^n X_i
  \xrightarrow{n\to\infty\;}{}
  N_\Sigma
  \text{ in distribution,}
\end{equation}
where $N_\Sigma \sim \cN(0,\Sigma)$ is a random variable distributed
as a Gaussian measure with the same covariance~$\Sigma$ as~$\mu$.

\subsection{Smooth manifold setting}\label{b:smooth}
\mbox{}\medskip

\noindent
Generalizing the LLN and CLT to an ambient space that is a smooth
manifold instead of a vector space requires reinterpreting both
\eqref{eq:vectorLLN} and~\eqref{eq:vectorCLT} since addition,
subtraction, and rescaling are not defined on manifolds.

The first ingredient in these reinterpretations comes from viewing the
mean of a measure~$\mu$ in light of its central importance in
statistics, namely as the deterministic point which minimizes the
expected squared distance from a random point distributed according
to~$\mu$.  More explicitly, the formulas for $\bmu$ and $\bmu_n$
in~\eqref{eq:vectorLLN} ask what point~$p$ minimizes, respectively,
the \emph{Fr\'echet function}
\begin{equation}\label{eq:FrechetF}
  F_\mu(p) = \frac12\int_\MM d(p,x)^2\,\mu(dx)
\end{equation}
of the \emph{population measure}~$\mu$ and the Fr\'echet function
$F_{\mu_n}(p)$ of the \emph{empirical measure}
$$
  \mu_n(dx) = \frac 1n (\delta_{X_1} + \dots + \delta_{X_n}).
$$
As usual, $\delta_x$ here is the unit
singular measure supported at~$x$, defined by $\delta_x(A) = 1$ if $x
\in A$ and $0$ otherwise.

On a metrized manifold
the Fr\'echet function~\eqref{eq:FrechetF} of a measure~$\mu$ is well
defined.  Hence $\bmu$ and~$\bmu_n$ can be defined as points~$p$ that
respectively minimize~$F_\mu(p)$ and $F_{\mu_n}(p)$:
\begin{equation}\label{eq:mean}
  \bmu = \argmin_p F_\mu(p).
\end{equation}
The classical summation formulas for these minimizers are not defined
on an arbitrary manifold, but the empirical measure~$\bmu_n$ is,
because measures on any space can be added.  Hence on a manifold the
statement
\begin{equation}\label{eq:smoothLLN}
  \bmu_n \xrightarrow{n\to\infty\;}{} \bmu \ \text{ almost surely}
\end{equation}
of the LLN remains unchanged, as should be expected \cite{ziezold1977}.

Turning to the CLT, it is natural to recast~\eqref{eq:vectorCLT} as a
statement taking place in the tangent space $\Tmu$ of the
manifold~$\MM$ being sampled.  Indeed, since~\eqref{eq:vectorCLT}
concerns the variation of rescaled differences $\bmu_n - \bmu$ as the
(moving) empirical mean converges to the (fixed) population mean, the
limit is naturally a random tangent vector.  The limiting process can
be brought entirely into the tangent space using the \emph{logarithm
map} centered at the mean~$\bmu$, which takes a point in~$\MM$ to the
direction and distance that~$\bmu$ moves to get there:
\begin{align*}
\\[-5ex]
  \log_\bmu: \MM^* & \to \Tmu
\\*
                 v &\mapsto \dd(\bmu,v)V,
\end{align*}
where $\MM^*$ is the set of points in~$\MM$ with a unique shortest
path to~$\bmu$ (so~$\MM^*$ is the complement of the cut locus
of~$\bmu$)
and $V$ is the tangent vector at~$\bmu$ of the geodesic from~$\bmu$
to~$v$ traversed at unit speed in the given metric~$\dd$.

The pushforward of~$\mu$ to the tangent space~$\Tmu$ is a measure $\nu
= (\log_\bmu)_\sharp \mu$ defined~by
$$
  \nu(A) = \mu\bigl(\log_\bmu^{-1}(A)\bigr)
$$
for measurable sets $A \subseteq \Tmu$.  Note that $\nu$ has
mean~$0$ automatically.  The linear CLT
$$
  \sqrt n\,\bnu_n
  \xrightarrow{n\to\infty\;}{}
  \cN(0,\Sigma)
$$
for this pushforward from~$\MM$ to~$\Tmu$ has limiting
distribution~$\cN(0,\Sigma)$, a Gaussian random variable taking values
in~$\Tmu$ with the usual bilinear covariance
\begin{equation}\label{eq:Rcov}
\begin{split}
  \Sigma: \Tmu \times \Tmu
        & \to \RR
\\
  (U,V) & \mapsto \int_{\Tmu} \<U,y\>\<V,y\>\, \nu\: dy.
\end{split}
\end{equation}

However, pushing $\mu$ forward to the tangent space~$\Tmu$
unceremoniously flattens the geometry of~$\mu$ in a way that erases
the curvature of the sample space~$\MM$ at the Fr\'echet mean~$\bmu$.
The situation is remedied by pushing~$\cN$ forward through the
\emph{distortion map}
$$
  \HH: \Tmu \to \Tmu
$$
that coincides with the inverse of the Hessian at~$\bmu$ of the
Fr\'echet function~\eqref{eq:FrechetF}.  The distortion map encodes,
for example, a tug of the covariance toward the origin in directions
where geodesics spread rapidly as they depart the Fr\'echet mean.
What results is the smooth CLT, to the effect that asymptotically, the
rescaled sample Fr\'echet means are distributed as the pushforward of
the Gaussian measure~$\cN(0,\Sigma)$ through the distortion map~$\HH$
\cite[\S2]{bhattacharya-patrangenaru2005}:
\begin{equation}\label{eq:smoothCLT}
  \lim_{n\to\infty} \sqrt n \log_\bmu \bmu_n
  \sim
  \HH_\sharp\cN(0,\Sigma).
\end{equation}
Note that rescaling by $\sqrt n$ makes sense in the tangent
space~$\Tmu$, unlike the expression on the left-hand side
of~\eqref{eq:vectorCLT}, where the scaling, difference, and sum do not
make formal sense in the nonlinear setting.

Typically, proofs of the LLN \eqref{eq:smoothLLN} or the CLT
\eqref{eq:smoothCLT} proceed by expanding the pushed-foward empirical
mean $\log_\bmu \bmu_n$ in a Taylor series and studying either the
limit (for~LLN) or rescaled limit (for CLT) in the tangent
space~$\Tmu$.  Doing so uses that the Fr\'echet function is
differentiable and $\Tmu$ is a vector space, so the classical LLN or
CLT applies to the terms from the Taylor expansion.  These classical
LLN or CLT results are then distorted by~$\HH$ to obtain the limiting
distribution on~$\Tmu$~in~\eqref{eq:smoothCLT}.

\subsection{Stratified setting}\label{b:stratified}
\mbox{}\medskip

\noindent
Confinement of empirical Fr\'echet means to decreasing neighborhoods
of the population mean pushes the CLT naturally to the tangent space,
which in the smooth manifold setting of the previous subsection is a
vector space.  This linearity in the smooth setting recovers a CLT
with the same form~\eqref{eq:smoothCLT} as the ordinary CLT in a
vector space~\eqref{eq:vectorCLT}, the differences being
\begin{itemize}
\item%
the logarithm map on the left-hand side of~\eqref{eq:smoothCLT},
which fulfills the need to push forward to the linear setting, and
\item%
the distortion map~$\HH$ on the right-hand side
of~\eqref{eq:smoothCLT}, which accounts for the curvature lost in
transition from the manifold to the tangent space.
\end{itemize}
In the stratified setting, where the ambient metric space is a
possibly singular union of smooth manifolds
(Definition~\ref{d:stratified-space}), the same general framework
succeeds, but this time with an extra step to reach the linear
setting: if the Fr\'echet mean~$\bmu$ is a singular point, then the
singularity must be flattened by a map $\LL: \Tmu \to \RR^m$.  Indeed,
the defining property of a non-smooth point is the failure of its
tangent space to be linear.

This brings us to the first version of our primary result.
Theorem~\ref{t:intro-CLT} is stated with complete precision (see
Theorem~\ref{t:clt}), using terminology we hope is evocative enough to
suffice until it is explained in
Sections~\ref{b:intro-escape}--\ref{b:gaussians}.

\begin{theorem}\label{t:intro-CLT}
Fix a smoothly stratified metric space~$\MM$ with localized immured
amenable probability measure~$\mu$.  A tangential collapse $\LL: \Tmu
\to \RR^m$ of~$\mu$ exists, and if $\cN(0,\Sigma)$ is a Gaussian
random vector in~$\RR^m$ with the same covariance $\Sigma =
\Cov\bigl((\LL \circ \log_\bmu)_\sharp \mu\bigr)$ as the pushforward
of~$\mu$ under logarithm at the Fr\'echet mean~$\bmu$ followed by
collapse,~then
\begin{equation}\label{eq:stratifiedCLT}
  \lim_{n\to\infty} \sqrt n \log_\bmu \bmu_n
  \sim
  \HH_\sharp\cN(0,\Sigma),
\end{equation}
where $\HH: \RR^m \to \Tmu$ is the distortion map.
\end{theorem}

Theorem~\ref{t:intro-CLT} is proved by way of a perturbative version
of the CLT~(\S\ref{b:via-random-fields}), which is stated in terms of
escape vectors~(\S\ref{b:intro-escape}) of Gaussian
{}masses~(\S\ref{b:gaussians}).  Escape is also one of the two
ingredients in distortion~(\S\ref{b:distortion}), the other being
tangential collapse \cite{tangential-collapse}, which is also the
basis for defining Gaussian {}masses.  The perturbative CLT itself
leads directly not only to Theorem~\ref{t:intro-CLT} but also to more
sophisticated perspectives on what a singular CLT should mean, as
singular analogues of G\^ateaux derivatives in spaces of measures or
functions~(\S\ref{b:directional}).

It should be noted that the path to the CLT here reprises certain key
elements of \cite{barden-le2018}, notably with our limit logarithm
maps
and escape cones,
which play roles
analogous to those in the developments by Barden and Le
\cite{barden-le2018} on orthant spaces \cite{centroids}; see
\cite[Remark~3.3]{shadow-geom} and
\cite[Remark~2.13]{tangential-collapse} for comments directly
connecting the relevant definitions to corresponding definitions in
\cite{barden-le2018}.  Our abstract take on these notions yields
complete control over how the limiting distribution behaves on the
boundary of its support without relying on the piecewise Euclidean
nature of orthant spaces.

\subsection{Escape vectors}\label{b:intro-escape}
\mbox{}\medskip

\noindent
Perturbing a measure~$\mu$ on a space~$\MM$ by adding an infinitesimal
point mass~$t\delta_x$ at~$x \in\nolinebreak \MM$ for $t \to 0$ causes
the Fr\'echet mean~$\bmu$ to wiggle.  When~$\MM$ is~$\CAT(\kappa)$,
the direction and magnitude of the \emph{escape vector}
(Definition~\ref{d:escape-vector}) along which this wiggle occurs
depends only on the logarithm $X = \log_\bmu x$.  Therefore this
version of escape amounts to a function $\ev: \Tmu \to \Tmu$.

More generally, since it is needed for technical reasons (that lead to
foundational reimagining; see Remark~\ref{r:Gmu}), this discussion
works as well when $\mu$ is perturbed by an infinitesimal measure
$t\delta$ for $t \to 0$ that is finitely supported in~$\MM$.  In that
case the escape vector (Definition~\ref{d:delta-escape}) only depends
on the corresponding measure~$\Delta = \log_\bmu \delta$ that is
finitely supported on the tangent cone~$\Tmu$.  Formalizing and
proving these assertions occupies the entirety of
Section~\ref{s:escape}, as we explain here.

Escape highlights a telling difference between the smooth and singular
cases: escape is unfettered in the smooth case, whereas in the
singular case, the escape vector of a measure~$\Delta$ sampled
from~$\Tmu$ (Definition~\ref{d:sampled-from}) only occurs along a
certain convex set of vectors, which constitute the \emph{escape cone}
$\Emu$ (Definition~\ref{d:escape-cone}).  This paper is, in a sense, a
detailed analytic study of escape: its
\begin{itemize}
\item%
constructions, implicit and explicit, to show it is well defined
(Theorem~\ref{t:escape-vector-confinement});
\item%
confinement to the escape cone~$\Emu$
(Theorem~\ref{t:escape-vector-confinement});
\item%
continuity as a function of the input measure~$\Delta$
(Corollary~\ref{c:escape-continuity});
\item%
consequences for asymptotics, when applied to a Gaussian {}mass, which
yields the limiting distribution in the perturbative CLT
(Theorem~\ref{t:perturbative-CLT}); and
\item%
contribution to distortion (Definition~\ref{d:distortion-map}) in
concert with tangential collapse.
\end{itemize}
In addition, the need to apply a continuous mapping theorem requires
the development of a functional version of escape
(Section~\ref{s:confinement}; see \S\ref{b:via-random-fields} for
discussion).

For the purpose of asymptotics of sample Fr\'echet means, the
measure~$\delta$ is sampled not merely from~$\MM$ but from the
potentially smaller set that is the support of~$\mu$.  For such
measures the escape vector $\ev(\delta)$ lies in the \emph{fluctuating
cone}~$\Cmu$ (Definition~\ref{d:fluctuating-cone}): the intersection
of the escape cone~$\Emu$ with the convex cone in~$\Tmu$ generated by
the logarithm of the support of~$\mu$.  This further ``fluctuating
confinement'' (Corollary~\ref{c:confinement-to-Cmu}) occurs when the
population measure~$\mu$ is \emph{immured} (Section~\ref{b:immured}),
a mild hypothesis that requires Fr\'echet means of finite samples
from~$\mu$ to land, after taking logarithm at~$\bmu$, in the convex
cone generated by the support.

\subsection{Distortion and amenability}\label{b:distortion}
\mbox{}\medskip

\noindent
The difference between the singular and smooth CLT is not visible in
the equations---Eq.~\eqref{eq:smoothCLT} and~\eqref{eq:stratifiedCLT}
are identical on both sides---but lies rather in the \emph{tangential
collapse}
$$
  \LL: \Tmu \to \RR^m
$$
from \cite{tangential-collapse}, which is reviewed in
Section~\ref{b:collapse}.  Its presence complicates the distortion
map~$\HH$ in fundamentally singular (that is, non-smooth) ways that
highlight the geometry behind two of the basic analytic hypotheses on
the measure in Theorem~\ref{t:intro-CLT}.

Already in the smooth setting, the CLT~\eqref{eq:smoothCLT}
encapsulates a notion of curvature embedded in the distortion
map~$\HH$.  Specifically, because the Fr\'echet function~$F(x)$ is the
expectation of the squared distance function $x \mapsto \dd(w,x)^2$,
its Hessian at~$\bmu$ furnishes a notion of spreading rate for
geodesics through~$\bmu$, from which sectional curvature can be
inferred.  Being the inverse of the Hessian at~$\bmu$ of~$F$, the
distortion map $\HH$ provides a $\mu$-average rate at which geodesics
through~$\bmu$ spread, which is a notion of Ricci curvature at~$\bmu$.
As a result, Theorem~\ref{t:intro-CLT} requires the squared distance
function to have second derivatives, at least directionally, and for
those directional derivatives to have $\mu$-expectations.  That is the
meaning of the \emph{amenable} hypothesis
(Definition~\ref{d:amenable}): for all~$w$ outside of the cut locus of
the Fr\'echet mean~$\bmu$, the squared distance from~$w$ has
second-order directional derivative at~$\bmu$ dominated by a
$\mu$-integrable function on~$\MM$.  This standard analytic sort of
hypothesis on the measure~$\mu$ allows differential and Taylor
expansion techniques for optimization.

Tangential collapse is guaranteed to exist when the measure~$\mu$ is
\emph{localized} (Definition~\ref{d:localized}): it has unique
Fr\'echet mean~$\bmu$, its Fr\'echet function is locally convex in a
neighborhood of~$\bmu$, and the logarithm map $\log_\bmu: \MM \to
\Tmu$ is $\mu$-almost surely uniquely defined.  This localized
hypothesis is satisfied, for instance, if the support of~$\mu$ lies in
a metric ball of radius less than $\pi/\sqrt\kappa$ on
a~$\CAT(\kappa)$ space \cite[Example~2.2]{tangential-collapse}.  In
particular, all measures are localized when $\kappa = 0$.

With tangential collapse in hand, distortion is its inverse image
followed by escape (Definition~\ref{d:distortion-map}).  That is,
given a point~$V$ in the convex hull of the image of tangential
collapse~$\LL$, find a measure~$\Delta$ sampled from~$\Tmu$ whose
image under~$\LL$ is~$V$, and take the escape vector of~$\Delta$.
Thus $\HH(V) = \ev \circ \LL^{-1}(V)$.  The fact that $V$ need not lie
in the image of~$\LL$ but only the convex hull of the image of~$\LL$
is why escape vectors of points do not suffice: $V$ might not have a
preimage point under~$\LL$
(Remark~\ref{r:gaussian-section}).  For geometric intuition regarding
distortion, keep in mind the description of tangential collapse
in terms of shadows \cite[Remark~3.18]{shadow-geom}: some subcone
of~$\Tmu$ collapses to a single ray while the rest of~$\Tmu$ remains
intact locally isometrically.

\subsection{Gaussians on singular spaces}\label{b:gaussians}
\mbox{}\medskip

\noindent
Tangential collapse pushes the (singular) tangent cone~$\Tmu$ to a
(smooth) linear space~$\RR^m$.  In so doing, the population
measure~$\mu$ pushes forward to a measure~$\hmu$ on~$\Tmu$, and from
there to a measure~$\muL$ on~$\RR^m$.  Of course, $\muL$ satisfies its
own ordinary linear CLT, with Gaussian limiting
distribution~$\cN(0,\Sigma)$.

Recovering a Gaussian object on~$\Tmu$ from this setup proceeds as in
distortion:
\begin{itemize}
\item%
let $\NLmu = \sim \cN(0,\Sigma)$ be a linear Gaussian-distributed
vector in~$\RR^m$, and
\item%
define a \emph{Gaussian {}mass} $\Gmu$ to be a lift of $\NLmu$
to a random measure sampled from the support of~$\hmu$ in~$\Tmu$
(Definition~\ref{d:gaussian-section}).
\end{itemize}
Thus $\Gmu$ is a random finitely supported measure on the tangent
cone~$\Tmu$.

In the end, the interest lies in the law not of~$\Gmu$ but of its
escape vector $\ev(\Gmu)$, for that turns out to be the limiting
distribution in Theorem~\ref{t:intro-CLT}; see
Theorem~\ref{t:intro-perturbation-CLT} in~\S\ref{b:via-random-fields}.
That is, distortion of a linear Gaussian equals escape of a Gaussian
{}mass:
$$
  \HH(\NLmu) = \ev(\Gmu).
$$

This formula pinpoints fluctuating confinement, as discussed in
\S\ref{b:intro-escape}, as essential for the theory surrounding
distortion and Gaussian {}masses.  Beyond the fact that the
fluctuating cone~$\Cmu$ maps isometrically to its image under
tangential collapse~$\LL$ (by definition; see
Section~\ref{b:collapse}), fluctuating confinement renders distortion
well defined.  That is because the relevant preimages of~$\NLmu$
under~$\LL$ all have the same escape vector
(Proposition~\ref{p:reduced-escape}), which in turn is precisely
because the support of the Gaussian vector~$\NLmu$ is generated by
samples from~$\mu$.

\subsection{Convergence via random tangent fields}\label{b:via-random-fields}
\mbox{}\medskip

\noindent
On a $\CAT(\kappa)$ space~$\MM$, a \emph{random tangent field} at a
point~$\bmu$ is a collection of real-valued random variables indexed
by unit tangent vectors at~$\bmu$ (see
Section~\ref{b:random-tangent-fields},~which summarizes
\cite{random-tangent-fields}).  One way of producing such fields is by
taking the \emph{inner product} of a fixed tangent vector~$V$
at~$\bmu$ with a $\Tmu$-valued random variable $X_i =\nolinebreak
\log_\bmu x_i$, where $x_i \sim \mu$ is an $\MM$-valued random
variable.  The inner product $\<X_i, V\>$ is not linear in the usual
sense---there is no ambient linear structure on the tangent
cone~$\Tmu$---but is merely a measure of angle derived from the metric
on the unit sphere around~$\bmu$ (Definition~\ref{d:inner-product}).
Starting with $n$ independent random variables $x_1,\dots,x_n$
in~$\MM$ therefore yields an \emph{empirical tangent field}
$$
  \Gn = \frac 1{\sqrt n} \sum_{k=1}^n \<X_i, V\>
$$
that is \emph{representable} (Definition~\ref{d:representable}) via
inner products with any vector~$V$ in the escape cone~$\Emu$ (see
\S\ref{b:intro-escape}).

The CLT for random tangent fields \cite{random-tangent-fields}
(reviewed here as Theorem~\ref{t:tangent-field-CLT} and
Corollary~\ref{c:cont-realization}) captures the asymptotic variation
of large sample means ray by ray: $\Gn \to G$, where $G$ is a
\emph{Gaussian} tangent field
(Definition~\ref{d:gaussian-tangent-field}).  One consequence of the
random field CLT is that Gaussian tangent fields are also
representable, and in fact each is represented by a Gaussian {}mass,
at least along rays in the fluctuating cone
(Theorem~\ref{t:Riesz-representation}):
$$
  G(X) = \<\Gmu, V\>_\bmu
$$
for all fluctuating vectors~$V$.

Thinking of inner products $\<X_i, V\>$ and $\<\Gmu, V\>$ as random
continuous real-valued functions of $V \in \Tmu$ that happen to be
representable as inner products, the escape vectors~$\ev(X_i)$
and~$\ev(\Gmu)$ can be thought of as vectors in~$\Tmu$ that minimize
perturbations of the Fr\'echet function by adding random representable
functions~$R$.  This minimization is continuous as a function of~$R$
(that is the role of Section~\ref{s:confinement}; see
Theorem~\ref{t:convergence-of-Upsilon}), in analogy with continuity of
escape for measures (Corollary~\ref{c:escape-continuity}), and it is
similarly confined to the fluctuating cone when the measure is immured
(Theorem~\ref{t:convergence-of-Upsilon}), in analogy with escape for
measures (Corollary~\ref{c:confinement-to-Cmu}).  The continuous
mapping theorem therefore applies, assembling the ray-by-ray
convergence of the random field CLT into the following convergence of
spatial variation around the Fr\'echet~mean~$\bmu$.

\begin{theorem}[Perturbative~CLT]\label{t:intro-perturbation-CLT}
Fix a localized immured amenable measure~$\mu$ on a smoothly
stratified metric space~$\MM$.  The empirical Fr\'echet mean~$\bmu_n$
and escape vector~$\ev(\Gmu)$ of any Gaussian {}mass~$\Gmu$~satisfy
$$
  \lim_{n\to\infty} \sqrt n \log_\bmu \bmu_n
  \overset{d}=
  \ev(\Gmu).
$$
\end{theorem}

Theorem~\ref{t:intro-perturbation-CLT} is stated in more
detail---including explicit formulas for the polar coordinates of the
Gaussian tangent field~$G$---as Theorem~\ref{t:perturbative-CLT}.
That is before Theorem~\ref{t:intro-CLT}, which is proved as an easy
consequence in Theorem~\ref{t:clt}, once the definition of distortion
is made precise in Definition~\ref{d:distortion-map}.

\subsection{CLTs as directional derivatives}\label{b:directional}
\mbox{}\medskip

\noindent
Tangential collapses provide a concrete handle on what the limiting
distribution in the singular CLT looks like, but the construction may
feel noncanonical, even if the end result does not depend on the
choices of tangential resolution and Gaussian {}mass.  That
said, another major outcome of our methods is that they allow
statements of the CLT that are universal in the sense they visibly
involve only intrinsic information that requires no choices.

To begin, view the process of taking Fr\'echet means of measures as a
map to~$\MM$ from the space~$\cP_2\MM$ of $L^2$-measures on~$\MM$,
namely
\begin{align*}
  \bb : \cP_2\MM & \to \MM
\\
             \mu &\mapsto \bmu.
\end{align*}
The escape vector can be seen (Lemma\,\ref{l:bb}) as a directional
derivative (Definition~\ref{d:bb-directional-deriv}):
$$
  \nablamu\bb(X) = \ev(X) \text{ for all } X \in \Tmu.
$$
which allows reinterpretion of the limiting distribution of~$\bmu_n$
as a directional derivative.

\begin{theorem}\label{t:deriv-of-barycenter}
The limiting distribution in the CLT for~$\bmu_n$
(Theorems~\ref{t:intro-CLT} and~\ref{t:intro-perturbation-CLT}) is the
directional derivative, in the space $\cP_2(\MM)$ of measures, of the
barycenter map~$\bb$ at~$\mu$ in the direction of a random discrete
measure given by any Gaussian {}mass~$\Gmu$:
$$
  \lim_{n\to\infty} \sqrt n \log_\bmu \bmu_n
  \overset{d}=
  \nablamu\bb(\Gmu).
s$$
\end{theorem}

Theorem~\ref{t:deriv-of-barycenter} is restated and proved as
Theorem~\ref{t:barycenter-CLT}.

Alternatively, the limiting distribution in the CLT of~$\bmu_n$ can be
viewed as the directional derivative of the \emph{confined minimizer
map} on continuous functions from the tangent cone to~$\RR$, namely
\begin{align*}
\fC: \cC(\Tmu,\RR) & \to \Tmu
\\
                 f &\mapsto \fC(f) \in \argmin_{X \in \oCmue} f(X),
\end{align*}
which takes each function~$f$ to a choice of minimizer in the closed
fluctuating cone~$\oCmue$.  (``$\fC$''~here is for ``confined''.)
For this purpose, the \emph{directional derivative} of the confined
minimizer $\fC$ at $\tF = \Fmu \circ \exp_\bmu$ along a function $R :
\Tmu \to \RR$ is
$$
  \nablatF\fC(R)
  =
  \lim_{t \to 0} \frac{\fC(\tF + tR)}{t}.
$$

\begin{theorem}\label{t:deriv-of-minimizer-map}
The limiting distribution in the CLT for~$\bmu_n$
(Theorems~\ref{t:intro-CLT} and~\ref{t:intro-perturbation-CLT}) is the
directional derivative, in the space $\cC(\Tmu,\RR)$, of the confined
minimizer map~$\fC$ at the Fr\'echet function $\tF$ along the negative
of the Gaussian tangent field~$G$ induced by~$\mu$:
$$
  \lim_{n\to\infty} \sqrt n \log_\bmu \bmu_n
  \overset{d}=
  \nablatF \fC(-G).
$$
\end{theorem}

Theorem~\ref{t:deriv-of-minimizer-map} is restated and proved as
Theorem~\ref{t:minimizer-CLT}.  Ideally, the confined minimization
over the closed fluctuating cone would be strengthened to minimize
over the entire tangent cone.  This is explained and stated precisely
as Conjecture~\ref{conj:minimizer-CLT}, which says that
Theorem~\ref{t:deriv-of-minimizer-map} holds after the word
``confined'' is deleted and both occurrences of the symbol~$\fC$ are
replaced by the \emph{minimizer map}
\begin{align*}
\BB: \cC(\Tmu,\RR) &\to \Tmu
\\
                 f &\mapsto \BB(f) \in \argmin_{X \in \Tmu} f(X),
\end{align*}
which takes each function~$f$ to a choice of minimizer anywhere in the
tangent cone.  (``$\BB$'' here is for ``barycenter''.)

%

\section{Geometric and probabilistic prerequisites}\label{s:prereqs}

\noindent
Formulating central limit theorems on singular spaces first requires a
class of spaces with sufficient structure.  These spaces are recalled
from \cite{tangential-collapse} in Section~\ref{b:collapse}, along
with the comparison of their tangent cones to linear spaces by
tangential collapse.  On the probabilistic side, a central limit
theorem proof must eventually rely on a convergence theorem, which in
this case is recalled from \cite{random-tangent-fields} in
Section~\ref{b:random-tangent-fields}.  Minimal basic background on
geometry of~$\CAT(\kappa)$ spaces necessary to make precise statements
of these prerequisites and the main new results later on occupies
Section~\ref{b:CAT(k)}.  More details, proofs, and foundations can be
found in the prequels \cite{shadow-geom, tangential-collapse,
random-tangent-fields} and in the sources they cite, especially
\cite{BBI01} and~\cite{bridson2013metric}.  In particular, the
exposition here assumes knowledge of $\CAT(\kappa)$ spaces; a
bare-bones introduction to those, tailored to the new developments,
can be found in \cite[Section~1]{shadow-geom}.

\subsection{\texorpdfstring{$\CAT(\kappa)$}{CAT(k)} background}\label{b:CAT(k)}

\begin{prop}\label{p:cat-spaces}
Any point~$\bmu$ in any $\CAT(\kappa)$ space~$\MM$ has a \emph{tangent
cone}~$\Tmu$ that is $\CAT(0)$ and carries an action of the
nonnegative real numbers~$\RR_+$ by scaling.  The set of unit vectors
in~$\Tmu$ is the \emph{unit tangent sphere}~$\Smu$.
\end{prop}

The following assertion about angles and the unit sphere metric comes
from \cite[Lemma~9.1.39]{BBI01}, but it might be helpful to note that
it appears as \cite[Lemma~1.2]{random-tangent-fields} and
\cite[Proposition~1.7]{shadow-geom} also.

\begin{lemma}\label{l:angular-metric}
The unit sphere $\Smu$ in Proposition~\ref{p:cat-spaces} is a length
space (induced from~$\Tmu$) whose \emph{angular metric}~$\dd_s$ allows
angles to be defined by
$$
  \dd_s(V, W) = \angle(V, W)
  \text{ whenever } V, W \in \Smu
  \text{ with } \angle(V, W) < \pi
$$
and $\angle(V, W) = \pi$ if $\dd_s(V, W) \geq \pi$ (see
Remark~\ref{r:path-through-apex} for geometric explanation).
\end{lemma}

The following notions regarding inner products on $\CAT(\kappa)$
spaces are taken from \cite[Definitions~1.11, 1.12, and~1.13 along
with Lemma~1.21]{shadow-geom}.

\begin{defn}\label{d:inner-product}
The \emph{inner product} of tangent vectors $V, W \in \Tmu$ at any
point~$\bmu$ is
$$
  \<V,W\>_\bmu = \|V\|\|W\| \cos\bigl(\angle(V,W)\bigr).
$$
When the point~$\bmu$ is clear from context, the subscript $\bmu$ is
suppressed.
\end{defn}

\begin{remark}\label{r:path-through-apex}
The shortest path in~$\Tmu$ from $V$ to~$W$ passes through the cone
point if $\dd_s(V,W) \geq \pi$ and otherwise behaves as the edge
opposite the angle of size~$\dd_s(V,W)$ between $V$ and~$W$ in the
triangle they span, which is Euclidean by \cite[Lemma~3.6.15]{BBI01}.
That is, the \emph{conical metric} on~$\Tmu$ is given by
$$
  \dd_p(V,W)
  =
  \sqrt{\|V\|^2 + \|W\|^2 - 2\<V,W\>}\
  \text{ for } V,W \in \Tmu.
$$
\end{remark}

\begin{lemma}\label{l:inner-product-is-continuous}
If a basepoint~$\bmu$ in a $\CAT(\kappa)$ space has been fixed, then
the inner product function $\<\,\cdot\,,\,\cdot\,\>_\bmu: \Tmu \times
\Tmu \to \RR$ is continuous.
\end{lemma}

\begin{defn}\label{d:log-map}
Fix $\bmu \in \MM$.  For $v$ in the set $\MM' \subseteq \MM$ of points
in the $\CAT(\kappa)$ space~$(\MM,\dd)$ with unique shortest path
to~$\bmu$, let the unit-speed shortest path from~$\bmu$ to~$v$
be~$\gamma_v$, with tangent vector $V = \gamma_v'(0)$ at~$\bmu$.  The
\emph{log map} is
\begin{align*}
  \log_\bmu : \MM' & \to \Tmu
\\
                 v &\mapsto \dd(\bmu,v) V.
\end{align*}
A~$\CAT(0)$ space~$\XX$ is \emph{conical} with apex~$\bmu$ if the log
map $\XX \to \Tmux$ is an isometry.
\end{defn}

The next concept was introduced in
\cite[Definition~2.1]{tangential-collapse} to guarantee existence of
tangential collapse as in Theorem~\ref{t:collapse}.

\begin{defn}[Localized measure]\label{d:localized}
A measure~$\mu$ on $\CAT(\kappa)$~$\MM$ is \emph{localized}~if
\begin{itemize}
\item%
$\mu$ has unique Fr\'echet mean~$\bmu$,
\item%
$\mu$ has locally convex Fr\'echet function in a neighborhood of~$\bmu$,
and
\item%
the logarithm map $\log_\bmu: \MM \to \Tmu$ is $\mu$-almost surely
uniquely defined.
\end{itemize}
Denote the pushforward of a localized measure~$\mu$ to~$\Tmu$ by
$$
  \hmu = (\log_\bmu)_\sharp\mu.
$$
\end{defn}

\begin{defn}\label{d:directional-derivative}
Fix a localized measure~$\mu$ on a $\CAT(\kappa)$ space~$\MM$.  The
Fr\'echet function~$F$ has \emph{directional derivative} at~$\bmu$
given by
\begin{align*}
  \nablabmu F : \Tmu & \to \RR
\\
                  V &\mapsto \frac{d}{dt} F(\exp_\bmu tV)|_{t=0}
\end{align*}
in which the exponential is a geodesic with constant speed and
tangent~$V$ at~$\bmu$.
\end{defn}

The next three concepts are \cite[Definition~2.12, 2.17,
and~2.18]{tangential-collapse}, where explanations of their
background, geometry, and motivation can be found.

\begin{defn}[Escape cone]\label{d:escape-cone}
The \emph{escape cone} of a localized measure~$\mu$ on a
$\CAT(\kappa)$ space~$\MM$ is the set~$\Emu$ of tangent vectors along
which the directional derivative
of the Fr\'echet function vanishes at~$\bmu$:
\begin{align*}
  \Emu &= \{X \in \Tmu \mid \nablabmu F(X) = 0\}.
\end{align*}
\end{defn}

\begin{defn}[Hull]\label{d:hull}
If~$\XX$ is a conical $\CAT(0)$ space, then the \emph{hull} of any
subset $\cS \subseteq \XX$ is the smallest geodesically convex cone
$\hull\cS \subseteq \XX$ containing~$\cS$.  For a localized
measure~$\mu$ on a $\CAT(\kappa)$ space~$\MM$, set
$$
  \hull\mu
  =
  \hull\supp(\hmu),
$$
the hull of the support in~$\XX = \Tmu$ of the pushforward measure
$\hmu = (\log_\bmu)_\sharp\mu$.
\end{defn}

\begin{defn}[Fluctuating cone]\label{d:fluctuating-cone}
The \emph{fluctuating cone} of a localized measure~$\mu$ on a
$\CAT(\kappa)$ space~$\MM$ is the intersection
\begin{align*}
  \Cmu &= \Emu \cap \hull\mu
  \\   &= \{V \in \hull\mu \mid \nablabmu F(V) = 0\}
\end{align*}
of the escape cone and hull of~$\mu$.  Let $\oCmu$ be the closure
in~$\Tmu$ of the escape cone~$\Cmu$.
\end{defn}

\subsection{Tangential collapse on smoothly stratified spaces}\label{b:collapse}
\mbox{}\medskip

\noindent
Tangential collapse forces the tangent cone of a smoothly stratified
metric space at the Fr\'echet mean of a given measure onto a linear
space in a way that preserves enough of the geometry relevant to
asymptotics of sampling.  The essential results on tangential collapse
are proved in \cite{tangential-collapse}, which also introduces the
relevant class of spaces.  The results in this subsection (except for
the last) are restated or derived from \cite[Definition~3.1,
Definition~4.13, Theorem~4.21, and
Corollary~4.22]{tangential-collapse}.

\begin{defn}[Smoothly stratified metric space]\label{d:stratified-space}
A \emph{smoothly stratified metric space} is a complete, geodesic,
locally compact $\CAT(\kappa)$ space that decomposes as
$$
  \MM = \bigsqcup_{j=0}^d \MM^j
$$
into disjoint locally closed \emph{strata}~$\MM^j$ such that the
stratum~$\MM^j$ for each~$j$ has closure
$$
  \overline{\MM^j} = \bigcup_{k\leq j}\MM^k,
$$
and both of the following hold.
\begin{enumerate}
\item\label{i:manifold-strata}%
(Manifold strata).  The space $(\MM^j,\dd|_{\MM^j})$ is, for each
stratum $\MM^j$, a smooth manifold whose geodesic distance is the
restriction $\dd|_{\MM^j}$ of~$\dd$ to~$\MM^j$.

\item\label{i:exponentiable}%
(Local exponential maps).  The exponential map $\exp_p = \log_p^{-1}$
is locally well defined and a homeomorphism around each point $p \in
\MM$; that is, there exists $\ve >0$ such that the restriction of the
logarithm map $\log_p$ at~$p$ to an open ball $B(p,\ve) \subseteq \MM$
is a homeomorphism onto its image.

\end{enumerate}
\end{defn}

\begin{defn}[Tangential collapse]\label{d:collapse}
A \emph{tangential collapse} of~$\mu$ on a smoothly stratified metric
space~$\MM$ is a map $\LL: \Tmu \to \RR^m$ to an
inner~product~space~such~that
\begin{enumerate}
\item\label{i:pushforward-mean}%
the Fr\'echet mean of the pushforward $\LL_\sharp\hmu$ is
$\LL\bigl(\log_\bmu (\bmu)\bigr) = 0$, where $\hmu =
(\log_\bmu)_\sharp\mu$;
\item\label{i:injective}%
the restriction of~$\LL$ to the closure~$\oCmu$ of the fluctuating
cone~$\Cmu$ is injective;
\item\label{i:partial-isometry}%
for any tangent vector $V \in \Tmu$ and any fluctuating vector $U \in
\Cmu$,
$$
  \<U,V\>_\bmu = \bigl\<\LL(U),\LL(V)\bigr\>_{\LL(\bmu)}.
$$
\item\label{i:homogeneous}%
$\LL$ is \emph{homogeneous}, meaning $\LL(tV) = t\LL(V)$ for all
real~$t \geq 0$ and~$V \in \Tmu$; and
\item\label{i:continuous}%
$\LL$ is continuous.
\end{enumerate}
\end{defn}

\begin{thm}[Tangential collapse]\label{t:collapse}
Any localized probability measure~$\mu$
on a smooth\-ly stratified metric space~$\MM$ admits a tangential
collapse~$\LL: \Tmu \to \RR^m$ such that
\begin{enumerate}
\item\label{i:proper}%
$\LL$ is the composite of a proper map followed by a convex geodesic
projection from a $\CAT(0)$ space onto a subset, and
\item\label{i:stratum}%
$\RR^m$ is the tangent space to a particular smooth stratum of~$\MM$
containing~$\bmu$.
\end{enumerate}
\end{thm}
\begin{proof}
In the terminal collapse $\LL = \bPi \circ \LL_\circ: \Tmu
\to\nolinebreak \Ti$ from \cite[Theorem~4.21]{tangential-collapse},
the composite d\'evissage $\LL_\circ$
\cite[Definition~4.12]{tangential-collapse}
is defined by composing a sequence of limit logarithm maps
\cite[Definition~3.3]{shadow-geom}, each of which is proper by
\cite[Corollary~3.19]{shadow-geom},
and by \cite[Definition~4.15]{tangential-collapse}
the terminal projection $\bPi$ is geodesic projection of the terminal
smoothly stratified metric space onto the tangent space of the
terminal stratum \cite[Definition~4.8]{tangential-collapse}.
\end{proof}

The target of tangential collapse is less important to CLTs than the
convex hull of its image, characterized in
\cite[Lemma~4.25]{tangential-collapse} as follows.

\begin{lemma}\label{l:supp-NLmu}
If $\LL: \Tmu \to \RR^m$ is a tangential collapse of a measure~$\mu$
and $\muL = \LL_\sharp\hmu = (\LL \circ \log_\bmu)_\sharp \mu$, then
$\hull\muL = \hull\LL(\supp\hmu) \cong \RR^\ell$ is a linear
subspace~of\/~$\RR^m$.
\end{lemma} 

\subsection{Random tangent fields and their CLT}\label{b:random-tangent-fields}
\mbox{}\medskip

\noindent
The precursor to the geometric central limit theorems in
Section~\ref{s:CLT} is a ray-by-ray version, phrased here in terms of
random tangent fields.  These results summarize the entirety of
\cite[Section~2]{random-tangent-fields}, followed by \cite[Theorem~4.2
and Corollary~4.3]{random-tangent-fields}.  The exposition in that
paper draws connections to classical central limit theorems.

\begin{defn}[Random tangent field]\label{d:random-tangent-field}
Let $(\Omega,\cF,\bP)$ be a complete probability space and $\MM$ a
smoothly stratified metric space.  A stochastic process~$f$ indexed by
the unit tangent sphere~$\Smu$ at a point $\bmu \in \MM$, meaning a
measurable map
\begin{align*}
  f: \Omega \times \Smu & \to \RR,
\end{align*}
is a \emph{random tangent field} on~$\Smu$.  Typically $\omega$ is
suppressed in the notation.  A random tangent field~$f$ is
\emph{centered} if its expectation vanishes: $\EE f(V) = 0$ for all $V
\in \Smu$.
\end{defn}

\begin{remark}\label{r:random-tangent-field}
A random tangent field~$f$ is often thought of and notated as a
collection of functions $\{f(V): \Omega \to \RR \mid V \in \Smu\}$
indexed by unit tangents at~$\bmu$.  Any \mbox{$\MM$-valued} random
variable $x = x(\omega)$ and function $g\hspace{-.1ex}:\hspace{-.1ex}
\MM \times \Smu \!\to\! \RR$ induce a random tangent~field
$$
  (\omega,V) \mapsto g\bigl(x(\omega),V\bigr),
$$
typically written $g(x,V)$ by abuse of notation.
\end{remark}

\begin{remark}\label{r:extend-to-Tmu}
A random tangent field~$f$ canonically extends to a homogeneous
stochastic process indexed by~$\Tmu$ instead of~$\Smu$, where $f$ is
\emph{homogeneous} if $f(tV) = tf(V)$ for all real~$t \geq 0$ and~$V
\in \Tmu$.  This extension was not needed in
\cite{random-tangent-fields} but arises later here, the first time
being at the start of Section~\ref{b:transform}; see
Definition~\ref{d:average-sample-perturbation}.
\end{remark}

\begin{remark}\label{r:nablamuF(V)=m(mu,V)}
To center a random tangent field, subtract the \emph{tangent
mean~function}
$$
  m(\mu,V)
  =
  \int_\MM \<\log_\bmu y, V\>_\bmu\mu(dy),
$$
or equivalenty add the derivative $\nablabmu F(V)$ by
\cite[Corollary~2.7]{tangential-collapse},
which asserts
$$
  \nablabmu F(V) = -\int_{\Tmu}\<W,V\>_\bmu\,\hmu(dW)
$$
for any localized measure~$\mu$ on any $\CAT(0)$ space~$\MM$.  The
tangent mean function $m(\mu,V) = -\nablabmu F(V)$ is never strictly
positive (as $\bmu$ is a minimizer of~$F$) and vanishes precisely when
$\nablabmu F(V) = 0$.  This vanishing exactly means that $V$ lies in
the escape cone~$\Emu$
(Definition~\ref{d:escape-cone}).
In~summary, $m(\mu,V) \leq 0$, with equality if and only if $V \in
\Emu$.
\end{remark}

\begin{defn}[Gaussian tangent field]\label{d:gaussian-tangent-field}
A \emph{Gaussian tangent field with covariance $\tangCOV$} is a
centered random tangent field $G:\Omega \times \Smu \to \RR$ such
that
\begin{enumerate}
\item%
the \emph{covariance} $(V,W) \mapsto \EE\,G(V)G(W)$ of~$G$
is $\tangCOV(V,W)$ for all $V,W \in \Smu$,~and
\item%
for all $V_1,\ldots, \hspace{-1pt}V_n \!\in\! \Smu$,
$\bigl(G(V_1),\ldots, G(V_n)\bigr)$ is multivariate Gaussian
\mbox{distributed}.
\end{enumerate}
The \emph{Gaussian tangent field induced by~$\mu$} is the centered
Gaussian random tangent field $G: \Omega \times \Smu \to \RR$ whose
covariance is $\EE G(V)G(W) = \tangCOV(\mu,V,W)$ for
$$
  \tangCOV(\mu,V,W) 
  =
  \int_\MM \bigl[\<\log_\bmu y, V\>_\bmu - m(\mu,V)\bigr]\,
           \bigl[\<\log_\bmu y, W\>_\bmu - m(\mu,W)\>_\bmu\bigr]\,\mu(dy).
$$
\end{defn}

\begin{defn}\label{d:empirical-tangent-field}
Let $\MM$ be a $\CAT(\kappa)$ space~$\MM$ with a measure~$\mu$.  Fix a
sequence $\xx = (x_1,x_2,\ldots) \in \MM^\NN$ of mutually independent
random variables $x_i \sim \mu$.  Write the $n^\mathrm{th}$ empirical
measure of~$\mu$ induced by~$\xx$ as $\mu_n = \frac 1n
\sum_{i=1}^n\delta_{x_i}$.  Let
$$
  g_i(V) = \<\log_\bmu x_i, V\>_\bmu - m(\mu,V)
$$
be the \emph{centered random tangent field induced by~$\mu$}.
Then
\begin{align*}
  \Gn(V) &= \frac 1{\sqrt n} \sum_{k=1}^n g_i(V)
\intertext{is the \emph{empirical tangent field} induced by~$\mu$.  In
contrast, the average of $g_1,\dots,g_n$ is}
   \bgn(V) &= \frac 1n \sum_{k=1}^n g_i(V)
           = \frac 1{\sqrt n} \Gn. \notag
\end{align*}
\end{defn}

The \emph{stratified $\CAT(\kappa)$ space} hypothesis in the following
results, which are \cite[Theorem~4.2 and
Corollary~4.3]{random-tangent-fields},
includes all smoothly stratified metric spaces.  More precisely, such
spaces satisfy all axioms for smoothly stratified metric spaces except
for existence of local exponential maps in
Definition~\ref{d:stratified-space}.\ref{i:exponentiable}; see
\cite[Definition~1.8]{random-tangent-fields}.

\begin{thm}[Random field CLT]\label{t:tangent-field-CLT}
Fix a stratified $\CAT(\kappa)$ space~$\MM$ and a localized
measure~$\mu$ on~$\MM$.  Let $G_n = \sqrt n\,\bgn$ be the empirical
tangent fields
for the collection of independent random variables~$x_i$ each
distributed according to~$\mu$, and let $G$ be the Gaussian tangent
field induced by~$\mu$.
Then the $\Gn$ converge to~$G$:
$$
  \Gn \overset{d\;}\to G
$$
in distribution as $n \rightarrow \infty$.
\end{thm}

\begin{cor}\label{c:cont-realization}
In the setting of Theorem\,\ref{t:tangent-field-CLT} there exist
versions of\hspace{.2ex} $\Gn$ and~$G$~so~that
\begin{enumerate}
\item%
$\Gn \to G$ almost surely in the space $\cC(\Smu, \RR)$ of continuous
functions from the unit tangent sphere at~$\bmu$ to~$\RR$, equipped
with the supremum norm, and
\item%
$G$ is almost surely H\"older continuous for any H\"older exponent
less than~$1$.
\end{enumerate}
\end{cor}

\section{Hypotheses on the population measure}\label{s:measures}

\noindent
Given a measure~$\mu$ on a smoothly stratified metric spaces, the
localized hypothesis in Definition~\ref{d:localized} guarantees
existence of tangential collapse in Theorem~\ref{t:collapse} and the
central limit theorem for random tangent fields in
Theorem~\ref{t:tangent-field-CLT}.  The CLTs in Section~\ref{s:CLT}
require additional mild analytic and geometric hypotheses on~$\mu$
(Sections~\ref{b:amenable} and~\ref{b:immured}, respectively) to allow
techniques involving Taylor expansion and convex optimization.

\subsection{Amenable measures}\label{b:amenable}

\begin{remark}
The CLT on manifolds is derived from the Taylor expansion of the
Fr\'echet function at the mean.  It was pointed out in \cite{Tra20}
that the spreading rate of geodesics at~$\bmu$ shapes the asymptotic
behavior of the CLT.  The spreading rate of geodesics at the mean is
captured by the inverse of the Hessian of the Fr\'echet function
\cite[Section~4.1]{Tra20}.  On manifolds, this rate can be quantified
by computing the second derivative of the squared distance function.
It turns out that to derive a CLT for Fr\'echet means, full
twice-differentiability of the squared distance function is not
needed.  Instead, only existence to second order of the Fr\'echet
derivative is needed.
\end{remark}

Recall that $\Smu$ is the unit tangent sphere
(Lemma~\ref{l:angular-metric}) at the Fr\'echet mean~$\bmu$, and that
the \emph{cut locus} of a point is the closure of the set of points
with more than one shortest path to~$\bmu$ (see \cite{le-barden14},
for example, in a Riemannian manifold context that is relevant to the
current purpose).  Thus a point~$w$ lies outside of the cut locus
of~$\bmu$ if $w$ has a neighborhood whose points all have unique
shortest paths to~$\bmu$.

\begin{defn}[Amenable]\label{d:amenable}
A measure~$\mu$ on a $\CAT(\kappa)$ space~$\MM$ is \emph{amenable} if
for all~$w$ not in the cut locus of the Fr\'echet mean, the \emph{half
square-distance function}
$$
  \rho_w = \frac 12 \dd^2(w,\mathord{\,\cdot\,})
$$
has second-order directional derivative
\begin{align*}
  \lambda_w: \Smu & \to \RR
\\*
                     Z &\mapsto \nablabmu^2 \rho_w(Z,Z)
\end{align*}
at~$\bmu$ dominated by a $\mu$-integrable function $\mathfrak{f}: \MM
\to \RR_{\geq 0}$; that is
\begin{enumerate}
\item%
$|\lambda_w(Z)| \leq \mathfrak{f}(w)$ for all $Z \in \Smu$, and
\item%
$\int_\MM \mathfrak{f}(w) \mu(dw) < \infty$.
\end{enumerate}
\end{defn}

The derivative in Definition~\ref{d:amenable} amounts to taking the
second derivative of the restriction of the function~$\rho_w$ to the
geodesic tangent to~$Z$.  The Fr\'echet function of any amenable
measure~$\mu$ has a second-order Taylor expansion at~$\bmu$.  To
formalize this in Proposition~\ref{p:Lambda}, and indeed for many
calculations throughout the paper---especially
Section~\ref{s:escape}---it is cleanest to use polar coordinates.
This is a manifestation of the transition from spatial to radial
variation in \cite{random-tangent-fields}.

\begin{defn}[Polar coordinates]\label{d:polar}
Fix a point~$\bmu$ in a $\CAT(\kappa)$ space~$\MM$.  The \emph{polar
coordinates} of $x \in \MM \setminus \{\bmu\}$ lying outside the cut
locus of~$\bmu$ are the \emph{radius} $r_x = \dd(x,\bmu)$ and the
\emph{direction} $\theta_x = \frac 1{r_x} \log_\bmu x$.  For the case
of~$\bmu$, set $r_\bmu = 0$ and let $\theta_\bmu$ take any value.
\end{defn}

\begin{remark}\label{r:polar}
Even though polar coordinates are defined outside of the cut locus and
the point~$\bmu$ itself, they are not truly coordinates, because
distinct points can have the same polar coordinates.  At issue is the
same phenomenon that prevents exponential maps from existing: a single
geodesic emanating from~$\bmu$ can bifurcate at a singularity.  Viewed
backward, two shortest paths to~$\bmu$ can join and thereafter share
the same terminal segment; this is the phenomenon underlying the
notion of shadow \cite[Definition~3.7]{shadow-geom}; see
\cite[Figure~1]{kale-2015} for a concrete example.  Local exponential
maps on smoothly stratified metric spaces
(Definition~\ref{d:stratified-space}) at least make polar coordinates
into true coordinates locally near~$\bmu$.
\end{remark}

\begin{prop}[Taylor expansion]\label{p:Lambda}
Fix a $\CAT(\kappa)$ space~$\MM$ and an amenable measure~$\mu$
on~$\MM$.  There exists a continuous function
$$
  \Lambda_\bmu: \Smu \to \RR
$$
such that the Taylor expansion of~$F(x)$ can be written in polar
coordinates as
$$
  F(x)
  =
  F(\bmu) + r_x\nablabmu F(\theta_x) + r_x^2\Lambda_\bmu(\theta_x) + o(r_x^2).
$$
\end{prop}
\begin{proof}
By the amenable hypothesis (Definition~\ref{d:amenable}), a Taylor
expansion of the square distance function integrates to get a Taylor
expansion of the Fr\'echet function~\eqref{eq:FrechetF}.
\end{proof}

\begin{lemma}\label{l:uniform-convexity}
The second-order directional derivative $\Lambda_\bmu (\theta_x)$ in
Proposition~\ref{p:Lambda} is positive.  In
addition, there exists a constant $C$ independent of $\mu$ such that
$$
  F(x) - F(\bmu) \geq C\dd^2(x,\bmu).
$$
\end{lemma}
\begin{proof}
Since $\MM$ is a $\CAT(\kappa)$ space and $\bmu$ is the minimizer of
the Fr\'echet function $F$ of~$\mu$, \cite[Corollary~12]{Yo17} (for
the case $\kappa > 0$) together with \cite[Proposition~2.3]{sturm2003}
(for the case $\kappa = 0$) says that $F(x)$ is uniformly convex.
Thus the second-order directional derivative $\Lambda_\bmu(\theta_x)$
in Proposition~\ref{p:Lambda} is positive.

Furthermore, \cite[Corollary~16]{Yo17} (for $\CAT(\kappa)$ with
$\kappa > 0$) together with \cite[Proposition~4.4]{sturm2003} (for
$\CAT(0)$) yields a lower quadratic bound for the rate of increase
of~$F$ at the minimizer~$\bmu$, which implies the desired inequality.
\end{proof}

\begin{remark}\label{r:smeary}
The amenable hypothesis prevents the measure~$\mu$ from being smeary
as in \cite{eltzner-huckemann2019} because the smeary arises when the
quadratic Taylor term vanishes.  Therefore smeary cases are not
covered by the central limit theorems here.
\end{remark}

\begin{example}\label{e:cut-locus}
On a smooth manifold of curvature bounded above, placing additional
mass near the cut locus in one direction can induce fluctuations of
the Fr\'echet mean in unrelated directions.  However, doing so
requires singular Hessian at the mean \cite[Theorem~8 and
Theorem~10]{tran-eltzner-huckemann2021}), which forces the measure to
violate the strong convexity property in
Lemma~\ref{l:uniform-convexity}, so $\mu$ is not amenable.  For a
specific example, let $\mu$ be uniform on a small disk centered at the
south pole in the $2$-sphere.  Add two point masses
$\delta_{\gamma(t)}$ and $\delta_{\gamma(-t)}$ along an arbitrary
direction~$V$.  Choose $t$ so that $t\cot t = -\frac{1-\ve}{\ve}$
(Figure~\ref{f:cut-locus}).  Then the Hessian of the perturbed measure
is singular along a different direction~$W$.  The proof of
\cite[Theorem~10]{tran-eltzner-huckemann2021} yields that the
empirical Fr\'echet mean~$\bmu_n$ is asymptotically concentrated along
the direction~$W$ with a scaling rate of $n^{1/6}$ although the two
perturbation masses are situated along direction~$V$.
\begin{figure*}
\begin{center}
  \includegraphics[scale=0.19]{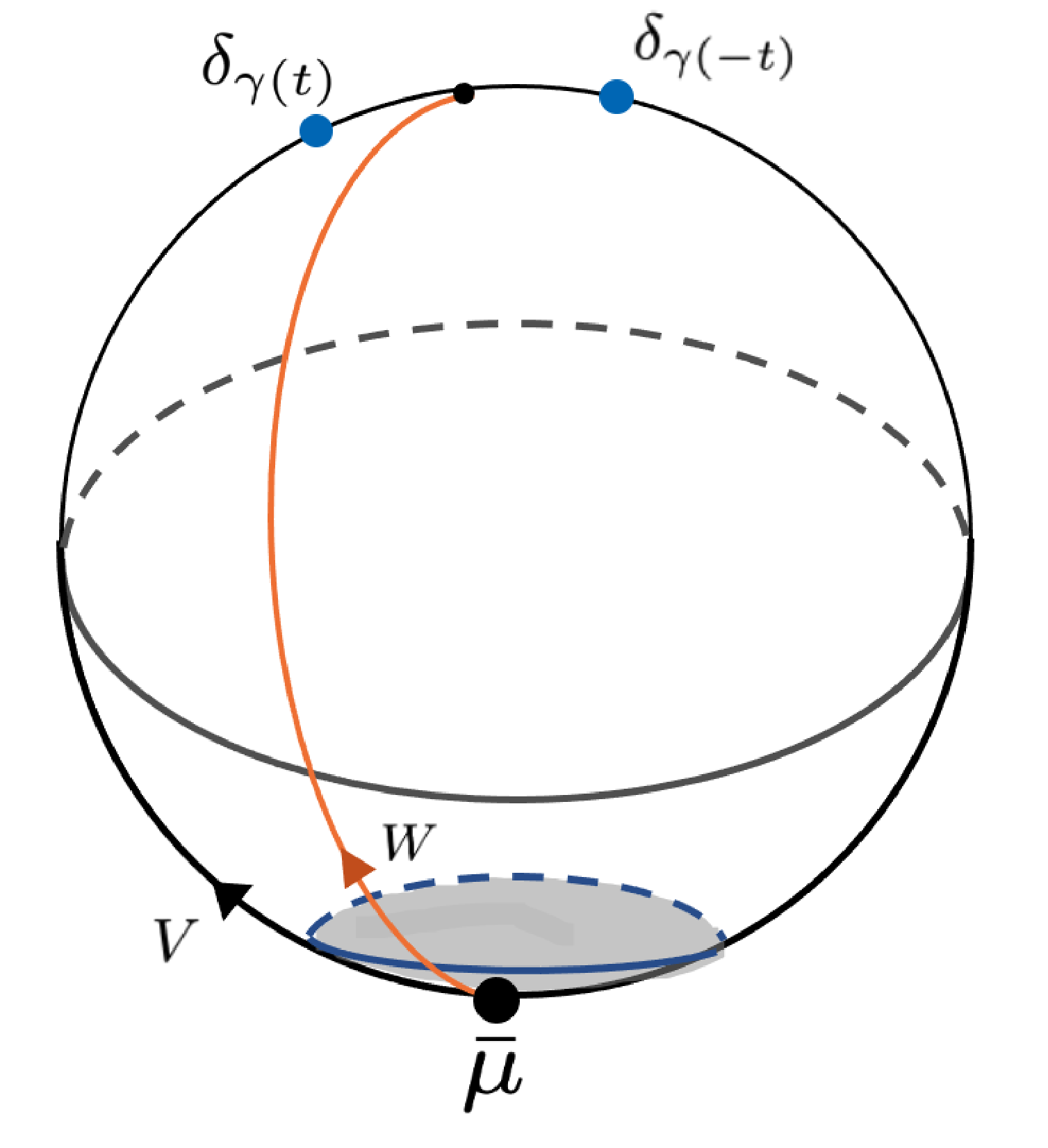}
\end{center}
\caption{Fluctuation of a non-amenable mean in an unexpected
direction}
\label{f:cut-locus}
\end{figure*}
Despite this fluctuation being in an unexpected direction, the
fluctuation does not exit the hull, since the measure~$\mu$ is
supported in a full (albeit small) neighborhood of the~south~pole.
\end{example}

\pagebreak[4]%

\subsection{Immured measures}\label{b:immured}
\mbox{}\medskip

\noindent
A point mass added to a measure~$\mu$ can cause the Fr\'echet mean to
escape from~$\bmu$ in directions that $\mu$ is unable to predict if
the added mass lies outside the support of~$\mu$.  In contrast, the
CLT only concerns perturbations of~$\mu$ by masses sampled from~$\mu$.
It~is reasonable to demand that Fr\'echet means of such samples land,
after taking logarithm at~$\bmu$, in the convex cone generated by the
support.

\begin{defn}[Immured]\label{d:immured}
A measure~$\mu$ on a smoothly stratified metric space~$\MM$
is~\emph{immured} if the Fr\'echet mean~$\bmu$ has an open
neighborhood $U \subseteq \MM$ such that $\log_\bmu \bmu_n
\in\nolinebreak \hull\mu$ whenever $\mu_n$ is a finitely supported
measure on~$\supp\mu$ and $\bmu_n \in U$.
\end{defn}

\begin{example}\label{e:immured-CAT(0)}
Every measure on a $\CAT(0)$ cone is immured, since the log map is an
isometry there and the Fr\'echet mean of any measure on any $\CAT(0)$
space (conical or otherwise) is a limit of points in convex hulls of
finite subsets of the support \cite[Theorem~4.7]{sturm2003}.  In
particular, every measure on an open book \cite{hotz-et-al.2013} or a
phylogenetic tree space as in~\cite{BHV01} is immured.
\end{example}

\begin{example}\label{e:immured-locally-convex}
A measure~$\mu$ on any smoothly stratified metric space~$\MM$ is
immured if the support of~$\mu$ contains a neighborhood of its
Fr\'echet mean~$\bmu$, because then $\hull\mu = \Tmu$.  More
generally, $\mu$~is immured if its support is locally convex
near~$\bmu$ in the sense that $B(\bmu,\ve) \cap \supp\mu$ is convex in
the ball~$B(\bmu,\ve)$ of radius~$\ve$ around~$\bmu$ for some $\ve >
0$.
\end{example}

\begin{example}\label{e:miracle}
Suppose that $\mu$ is a measure supported on the spine of an open book
\cite[Theorem~2.9]{hotz-et-al.2013}.  Perturbing~$\mu$ by adding a
mass on a single page of the open book nudges the Fr\'echet
mean~$\bmu$ onto that page.  Hence the escape cone~$\Emu$ equals the
entire open book.  Limit log along any individual page therefore
collapses~$\Emu$ non-injectively.  This elementary example explains
why it is vital to control potential fluctuations by immuring them
within the convex hull of the support of the measure
(Lemma~\ref{l:immured} and Corollary~\ref{c:confinement-to-Cmu}) in
addition to their automatic confinement
to~$\Emu$~(Theorem~\ref{t:escape-vector-confinement}).
\end{example}

\begin{remark}\label{r:immured}
The immured condition becomes relevant at the very end of
Section~\ref{s:escape}.\ref{b:convergence}, specifically
Lemma~\ref{l:immured} on the way to
Corollary~\ref{c:confinement-to-Cmu}, where it confines escape vectors
to the fluctuating cone~$\Cmu$.  No result before then, including the
rest of Section~\ref{s:escape}.\ref{b:convergence}
through Theorem~\ref{t:escape-vector-confinement}, invokes the immured
hypothesis.  It is possible
the immured condition is not necessary at all, for any of the main
results; see Remark~\ref{r:ideally}.
\end{remark}

\pagebreak[4]

\section{Escape vectors and confinement}\label{s:escape}

Perturbing the measure~$\mu$ by adding a small point mass at~$x \in
\MM$ causes the Fr\'echet mean to escape from the population mean in a
deterministic way, as a function of~$x$
(Definition~\ref{d:escape-vector}).  The central limit theorem can be
expressed as identifying the distribution that results when the
point~$x$ is a random tangent vector generating a Gaussian random
tangent field (Theorem~\ref{t:perturbative-CLT} via
Theorem~\ref{t:Riesz-representation}).  The connection to classical
expressions of the CLT as identifying the limiting distribution of
rescaled Fr\'echet means of samples from~$\mu$ proceeds by directly
comparing those means to perturbations along empirical random tangent
fields (Proposition~\ref{p:transform}), which suffices by the CLT for
random tangent fields (Theorem~\ref{t:tangent-field-CLT}).

The reduction of the CLT to empirical tangent perturbation occupies
Section~\ref{b:transform}, after introducing concepts surrounding
escape vectors in Section~\ref{b:escape}.  The empirical perturbation
itself is naturally expressed as a minimization over~$\MM$ or, by
taking logarithms, over the tangent cone~$\Tmu$ at the Fr\'echet mean.
That minimization is difficult to analyze directly.  However, as
Section~\ref{b:convergence} shows via well chosen approximations of
the problem, it becomes tractable---with explicit formulas and
straightforward uniqueness---when the minimization is taken only over
the escape cone~$\Emu$.  This tack succeeds because perturbation
of~$\mu$ naturally causes the Fr\'echet mean to exit from~$\bmu$ along
an escape vector.  Confinement to~$\Emu$ in this manner is the main
result of Section~\ref{s:escape},
Theorem~\ref{t:escape-vector-confinement}.  It is particularly potent
for central limit theorems when the perturbation of~$\mu$ occurs in
the support of~$\mu$, in which case the escape is further confined to
the fluctuating cone~$\Cmu$ (Corollary~\ref{c:confinement-to-Cmu}).
A~functional version of fluctuating confinement in
Section~\ref{s:confinement} relates convergence in the context of
escape vectors to convergence in the context of random tangent fields
from Section~\ref{b:random-tangent-fields}, which is needed for an
application of the continuous mapping theorem to the random field CLT
from Theorem~\ref{t:tangent-field-CLT}; see
Remark~\ref{r:ev-continuity}.

The setting throughout Section~\ref{s:escape} is that of
Definition~\ref{d:empirical-tangent-field}, except that $\MM$ is not
merely $\CAT(\kappa)$ but a smoothly stratified metric space:
$x_1,x_2,\ldots\!\sim \mu$ are $\MM$-valued independent random
variables with distribution~$\mu$ and empirical measure $\mu_n = \frac
1n \sum_{i=1}^n \delta_{x_i}$, whose Fr\'echet function
$$
  F_n
  =
  F_{\mu_n}
  =
  \frac 1n \sum_{j=1}^n \rho_{x_j}
$$
is expressed compactly using the half square-distance function $\rho$
from Definition~\ref{d:amenable}.

The measure~$\mu$ is repeatedly (but not always) required to be
amenable throughout this section because of
Lemma~\ref{l:converge-of-hessian} and its heavily cited consequence,
Proposition~\ref{p:transform}.  That said, many of the definitions
make sense for measures that need not be amenable.

In this section, the subscript $\bmu$ is often suppressed in the inner
product $\<Y, Z\>_\bmu$ from Definition~\ref{d:inner-product} when
there is no ambiguity about the identity of the base point.


\subsection{Escape vectors}\label{b:escape}
\mbox{}\medskip

\noindent
The main result of this section,
Theorem~\ref{t:escape-vector-confinement}, shows that the output of
the following definition exists and is well defined.

\begin{defn}\label{d:escape-vector}
Let $\mu$ be a measure on a smoothly stratified metric space~$\MM$.
Fix $x \in \MM$ with unit measure~$\delta_x$ supported at~$x$.  For
positive $t \to 0$, the perturbed measures $\mu + t\delta_x$ have
Fr\'echet means $\ol{\mu + t\delta_x}$.
\begin{enumerate}
\item\label{i:x}%
The \emph{escape vector} of~$x$ is
$$
  \ev(x)
  =
  \lim_{t\to 0^+} \frac 1t \log_\bmu(\,\ol{\mu + t\delta_x}\,).
$$

\item\label{i:X}%
If $X \in \Tmu$ is a tangent vector with exponential $x = \exp_\bmu X
\in \MM$, then the \emph{escape vector} of~$X$ is $\ev(X) = \ev(x)$.

\item\label{i:rescale}%
More generally, if $X \in \Tmu$ is any tangent vector, then set
$\ev(X) = \frac 1r\ev(rX)$ for any $r > 0$ such that $rX$ is
exponentiable
(cf.~Definition~\ref{d:stratified-space}.\ref{i:exponentiable}).
\end{enumerate}
\end{defn}

\begin{remark}\label{r:escape-cone}
It is not obvious that the limit in
Definition~\ref{d:escape-vector}.\ref{i:x} exists.  Indeed, existence
is part of the main result of this section,
Theorem~\ref{t:escape-vector-confinement}, whose primary conclusion is
that escape vectors of arbitrary points in~$\MM$ lie in the escape
cone~$\Emu$ from Definition~\ref{d:escape-cone}; that is, escape
vectors of arbitrary points yield vanishing directional derivatives of
the Fr\'echet function.  But arbitrary escape vectors can pull the
Fr\'echet mean out of the fluctuating cone~$\Cmu$ from
Definition~\ref{d:fluctuating-cone}, necessitating
Corollary~\ref{c:confinement-to-Cmu} for escape vectors of samples
from~$\mu$.
\end{remark}

\begin{remark}\label{r:rescale}
It is also not obvious that the rescaling procedure in
Definition~\ref{d:escape-vector}.\ref{i:rescale} yields a well defined
output.  That is also part of
Theorem~\ref{t:escape-vector-confinement}, the crucial point being
homogeneity in Corollary~\ref{c:homogeneity}, which is a consequence
of the explicit formula for the radial polar coordinate of the escape
vector in terms of its direction
Proposition~\ref{p:unique-c-and-limits}.
\end{remark}

\begin{remark}\label{r:discrete-measures}
The (limiting distribution in the) CLT can be thought of as resulting
from taking Fr\'echet means after perturbing~$\mu$ by adding mass at a
single point along a Gaussian random tangent field; indeed, that is
the meaning of Theorem~\ref{t:perturbative-CLT}.  However,
nonlinearity of the tangent cone at~$\bmu$ forces consideration of
perturbations by mass at finitely many points of~$\MM$.  These
finitely many points---or, more precisely, their logarithms
in~$\Tmu$---are pushed to a vector space~$\RR^m$ by tangential
collapse~$\LL$ (Section~\ref{b:collapse}, especially
Definition~\ref{d:collapse} and Theorem~\ref{t:collapse}), where they
are subsequently averaged to get a single vector.  The technical
obstacle is that this average vector in~$\RR^m$ need not have a
preimage in~$\Tmu$,
such as when $\MM$ is the plane~$\RR^2$ with
an open quadrant deleted, in which case $\LL$ is the inclusion $\MM
\into \RR^2$.  Definition~\ref{d:delta-escape} accounts for the fact
that assigning an escape vector to this single average vector
in~$\RR^m$ (Definition~\ref{d:distortion-map}) is a process that
occurs not in~$\RR^m$ but in~$\Tmu$, before the vectors can be
averaged, by working with finitely supported measures instead of with
individual points.  Stating Definition~\ref{d:delta-escape} therefore
first requires extensions of inner products
(Definition~\ref{d:inner-product}) and tangent cone rescaling to the
setting of discrete measures instead of individual vectors.
\end{remark}

\begin{defn}\label{d:sampled-from}
A measure~$\delta$ on a space~$\MM$ is \emph{sampled from~$\MM$},
written $\delta \Subset \MM$, if it has finite support, so for some
points $y^1,\dots,y^j \in \MM$ it has the form
$$
  \delta
  =
  \lambda_1\delta_{y^1} + \dots + \lambda_j\delta_{y^j}.
$$
In terms of half square-distance~$\rho$ in
Definition~\ref{d:amenable}, the measure $\delta$ has Fr\'echet
function
\begin{align*}
  F_\delta
& = \frac 12 \bigl(\lambda_1\dd^2(y^1,\mathord{\,\cdot\,}) + \dots +
    \lambda_j\dd^2(y^j,\mathord{\,\cdot\,})\bigr)
\\
& = \lambda_1\rho_{y^1} + \dots + \lambda_j\rho_{y^j}.
\end{align*}
\end{defn}

\begin{defn}\label{d:pair-with-discrete-measure}
For any measure $\Delta = \lambda_1\delta_{Y^1} + \dots +
\lambda_j\delta_{Y^j}$ sampled from~$\Tmu$,~set
$$
  \<\Delta,\,\cdot\,\>
  =
  \lambda_1\<Y^1,\,\cdot\,\> + \dots + \lambda_j\<Y^j,\,\cdot\,\>
$$
in terms of inner products (Definition~\ref{d:inner-product}).  For
any measure $\delta$ sampled from~$\MM$ with pushforward
$(\log_\bmu)_\sharp\delta =\nolinebreak \Delta$, simplify notation by
writing $\Delta = \log_\bmu\delta$.  If every $Y^i \in \Tmu$ can be
exponentiated
(cf.~Definition~\ref{d:stratified-space}.\ref{i:exponentiable}; note
that this need not be the case even if $Y^i =\nolinebreak \log_\bmu
y^i$ for all~$i$), then also write $\delta = \exp_\bmu\Delta$ and say
that $\Delta$ \emph{can be exponentiated}.
\end{defn}

\begin{remark}\label{r:scaling}
When the measure~$\mu$ is localized, the discrete measure~$\delta$ in
Definition~\ref{d:pair-with-discrete-measure} almost surely has well
defined logarithm by Definition~\ref{d:localized}.  But transmitting
information back to~$\MM$ requires the tangent vectors supporting
$\Delta =\nolinebreak \log_\bmu\delta$ to be sufficiently short,
cf.~Definition~\ref{d:stratified-space}.\ref{i:exponentiable}.  This
issue is rectified, thereby allowing arbitrary measures~$\Delta$
sampled from~$\Tmu$, using homogeneity in the sense of
Definition~\ref{d:collapse}.\ref{i:homogeneous}: vectors in the
support of~$\Delta$ can be shrunk, operated on, and re-stretched (see
Corollary~\ref{c:homogeneity}).  However, the default notation
$t\Delta$ for $t \geq 0$ has (and needs to have, for the purpose of
writing $\hmu + t\Delta$) the standard meaning
$$
  t\Delta = t\lambda_1\delta_{Y^1} + \dots + t\lambda_j\delta_{Y^j}
$$
which necessitates different notation for rescaling the vectors
$Y^1,\dots,Y^j$.  Placing the scalar after~$\Delta$ indicates it
should scale the vectors instead of their masses, as follows.
\end{remark}

\begin{defn}\label{d:scaling}
For $\Delta = \lambda_1\delta_{Y^1} + \dots + \lambda_j\delta_{Y^j}$
sampled from~$\Tmu$ and real $r \geq 0$,~set
$$
  \Delta r = \lambda_1\delta_{rY^1} + \dots + \lambda_j\delta_{rY^j}
$$
\end{defn}

\begin{defn}[Escape vector]\label{d:delta-escape}
Fix a measure $\mu$ on a smoothly stratified space~$\MM$.
\begin{enumerate}
\item\label{i:delta}%
Any measure~$\delta \Subset \MM$ has \emph{escape vector}
$$
  \ev(\delta)
  =
  \lim_{t\to 0^+} \frac 1t \log_\bmu(\,\ol{\mu + t\delta}\,).
$$

\item\label{i:Delta}%
If $\Delta \Subset \Tmu$ with exponential $\delta = \exp_\bmu \Delta
\Subset \MM$, then the \emph{escape vector} of~$\Delta$ is
$\ev(\Delta) = \ev(\delta)$.

\item\label{i:delta-rescale}%
More generally, for any measure $\Delta \Subset \Tmu$, set
$\ev(\Delta) = \frac 1r\ev(\Delta r)$ for any $r > 0$ such that
$\Delta r$ is exponentiable
(cf.~Definitions~\ref{d:stratified-space}.\ref{i:exponentiable}
and~\ref{d:scaling}).
\end{enumerate}
\end{defn}

The goal of Section~\ref{s:escape} is to show in
Theorem~\ref{t:escape-vector-confinement} that escape vectors are well
defined and moreover can be obtained by minimizing over meaningfully
smaller subsets of points or tangent vectors than $\MM$ or~$\Tmu$.
Some further remarks and notations help set the stage for the
constructions and proofs in the remainder of Section~\ref{s:escape}.

\begin{remark}\label{r:rho-delta}
The fact that the measure $\mu + t\delta$ is not a probability
measure, because the mass of $\delta \Subset \MM$ is unrestricted,
makes no material difference to
Definition~\ref{d:escape-vector}.\ref{i:x} or
Definition~\ref{d:delta-escape}.\ref{i:delta} because the Fr\'echet
function in Eq.~\eqref{eq:FrechetF} is defined just as well for
measures of nonunit mass, and the Fr\'echet mean in
Eq.~\eqref{eq:mean} is unaffected by globally scaling the measure.
Thus the measure $\mu + t\delta$ yields the same minimizer as the
probability measure obtained by rescaling $\mu + t\delta$.
Geometrically, it can be helpful to think of the Fr\'echet function
$F_\delta$ here as extending the notion of half square-distance~$\rho$
from Definition~\ref{d:amenable}.
\end{remark}

Passing from minimization over points in~$\MM$ to minimization over
their logarithms only works reversibly in a neighborhood of~$\bmu$,
cf.~Definition~\ref{d:stratified-space}.\ref{i:exponentiable}.  To
maintain precision in the many minimizations to come, it is useful to
have notations for sets of exponentiable tangent vectors at~$\bmu$.

\begin{defn}\label{d:exponentiable}
Given a measure~$\mu$ on a smoothly stratified metric space~$\MM$,
let
$$
  \Tmue \subseteq \Tmu
  \qquad
  \Emue \subseteq \Emu
  \qquad
  \text{and}
  \qquad
  \Cmue \subseteq \Cmu
$$
be the subsets of the tangent cone (Proposition~\ref{p:cat-spaces}),
escape cone (Definition~\ref{d:escape-cone}), and fluctuating cone
(Definition~\ref{d:fluctuating-cone}) consisting of vectors that can
be exponentiated.
\end{defn}

\begin{remark}\label{r:exponentiable}
The $\argmin$ in Remark~\ref{r:rho-delta} equals
$$
  a(t,\delta)
  =
  \argmin_{X\in\Tmue} \bigl(F(\exp_\bmu X) + t\Fd(\exp_\bmu X)\bigr)
$$
as long as the deformation of~$\mu$ by~$t\delta$ is sufficiently mild,
so the Fr\'echet mean doesn't move too much (see
Lemma~\ref{l:magnitude-of-a_n}).  That poses no meaningful obstacle
for the purpose of escape vectors, since the point is to take limits
as $t \to 0$ anyway.
\end{remark}

\begin{remark}\label{r:bn}
Rephrasing the $\argmin$ as taking place over tangent vectors in
Remark~\ref{r:exponentiable} instead of points from~$\MM$ in
Remark~\ref{r:rho-delta} makes way for alternate methods to recover
deformed Fr\'echet means directly in terms of operations on tangent
vectors, notably inner products as in
Definition~\ref{d:pair-with-discrete-measure}.  The methods use
expressions of the form
\begin{align*}
  b(t,\Delta)
& \in
  \argmin_{X\in\Tmue} \bigl(F(\exp_\bmu X) - t\<\Delta, X\>\bigr)
\intertext{or, via polar coordinates (Definition~\ref{d:polar}) and
the second-order coefficient~$\Lambda_\bmu$ in the Fr\'echet function
Taylor expansion (Proposition~\ref{p:Lambda}),}
  c(t,\Delta)
& \in
  \argmin_{X\in\Tmue} \bigl(r_X^2\Lambda_\bmu(\theta_X) - t\<\Delta, X\>\bigr).
\end{align*}
\end{remark}

\subsection{Transforming the CLT to a perturbation problem}\label{b:transform}
\mbox{}\medskip

\noindent
The goal of this subsection is to transform the problem of finding
(the distribution of) the empirical Fr\'echet mean $\bmu_n =
\argmin_{x\in\MM} F_n(x)$ into the problem of finding the minimizer a
function obtained from perturbation of the Fr\'echet function $\Fmu =
F_\mu$ by the empirical random tangent field~$\bgn$ from
Definition~\ref{d:empirical-tangent-field}.  This goal is accomplished
in Proposition~\ref{p:transform}, which shows that a minimizer
$$
  h_n \in \argmin_{x\in\MM} \bigl(F(x) - \bgn(\log_\bmu x)\bigr)
$$
serves as a suitable proxy for the Fr\'echet mean, in the sense that
$\sqrt n\, \dd(h_n, \bmu_n) \xrightarrow{P} 0$.  Note that $\bgn$ is
considered here as a random tangent field on~$\Tmu$, as in
Remark~\ref{r:extend-to-Tmu}.

The first step is a uniform law of large numbers.  This lemma uses
amenability of the measure~$\mu$ (Definition~\ref{d:amenable})
explicitly in the proof and is one of the reasons for requiring that
basic hypothesis.  To understand the statement, observe that the
quantity $\zeta_n$ defined there is a certain proxy, constructed from
the law of cosines, for a $\sqrt n$-zoomed-in version of the half
square-distance from the Fr\'echet mean~$\bmu$ to~$z$.

\begin{lemma}\label{l:converge-of-hessian}
Fix an amenable measure~$\mu$ on a smoothly stratified metric
space~$\MM$.  Using the half square-distance function~$\rho_w$ from
Definition~\ref{d:amenable}, let
\begin{align*}
  \zeta_n: \MM \times \MM
         & \to \RR
\\
  (w,z) &\mapsto \sqrt n\bigl(
	 \rho_w(z/\sqrt n) - \rho_w(\bmu) + \<W, Z/\sqrt n\>\bigr),
\end{align*}
where $Z/\sqrt n = n^{-1/2}\log_\bmu z$ and $z/\sqrt n$ is shorthand
for $\exp_\bmu(Z/\sqrt n)$ and $W = \log_\bmu w$.  Then for any
compact set $U \subseteq \MM$ there is a uniform law of large numbers
$$
  \sup_{z \in U}\Bigl(\frac 1n \sum_{i=1}^n
    \sqrt n\,\zeta_n(x_i, z/\sqrt n)-\EE\bigl[\sqrt n\,\zeta_n(x,z/\sqrt n)\bigr]\Bigr)
  \xrightarrow{P}
  0.
$$
\end{lemma}
\begin{proof}
The result essentially follows from a relatively classical uniform law
of large numbers (cf.~\cite[Lemma~2.4]{newey-mcFadden1994} or
\cite{jennrich1969}) whose assumptions are that there exists an
integrable, positive, measurable function $\mathfrak{f}$ with $|\sqrt
n\,\zeta_n(w, z/\sqrt n)| \leq C \mathfrak{f}(w)$ for all $z \in U$
and that the family of random variables $\sqrt n\,\zeta_n(w, z/\sqrt
n)$ is uniformly continuous in $z$ in a sense made precise below.  The
argument is sketched here assuming existence of such a function
$\mathfrak{f}$ and the needed uniform continuity.

The weak law of large numbers for triangular arrays (see
\cite[Theorem~2.2.6]{Dur19} with $b_n = \sqrt n$) implies that for any
fixed~$z$,
$$
  \frac 1n \sum_{i=1}^n \sqrt n\,\zeta_n(x_i, z/\sqrt n)-\EE\sqrt
  n\,\zeta_n(x,z/\sqrt n) \rightarrow 0
$$
in probability as $n \rightarrow \infty$.

Since $U$ is compact and $\frac{\partial}{\partial z}\bigl(\sqrt
n\,\zeta_n(w, z/\sqrt n)\bigr) = \frac{\partial \zeta_n}{\partial
z}(w, z/\sqrt n)$, the collection of functions $\bigl\{\sqrt
n\,\zeta_n(w, z/\sqrt n) \bigm| n \in \NN\bigr\}$ is uniformly
continuous in $z$. From this and the discussion below, there also
exists an integrable, positive measurable $\mathfrak{g}(w)$ so that
for all $z \in U$ and $n \in \NN$, and for all $\ve > 0$,
\begin{align*}
  \gamma(z,w,\ve)
  :=
  \sup_{z' \in B(z,\delta)}
       \bigl|\sqrt n\zeta_n(w,z'/\sqrt n)-\sqrt n\zeta_n(w, z/\sqrt n)\bigr|
  \leq
  \mathfrak{g}(w)\ve.
\end{align*}
Fixing any $\ve > 0$, let $\{z_k \mid k=1,\dots,K_\ve\} \subset U$ be
such that $U \subset \bigcup_k B(z_k,\ve)$.  Then
\begin{align*}
  \sup_{z \in U}&\Bigl(
    \frac 1n \sum_{i=1}^n \sqrt n\,\zeta_n(x_i, z/\sqrt n)
    - \EE\bigl[\sqrt n\,\zeta_n(x,z/\sqrt n)\bigr]\Bigl)
\\
  & =
  \max_{k \in \{1,\dots,K_\ve\}}\sup_{z \in B(z_k,\ve)}\Bigl(
    \frac 1n \sum_{i=1}^n \sqrt n\,\zeta_n(x_i, z/\sqrt n)
    - \EE\bigl[\sqrt n\,\zeta_n(x,z/\sqrt n)\bigr]\Bigr)
\\
  & \leq
  \sup_{k \in \{1,\dots,K_\ve\}}\Bigl(
    \frac 1n \sum_{i=1}^n \sqrt n\,\zeta_n(x_i, z_k/\sqrt n)
    - \EE\bigl[\sqrt n\,\zeta_n(x,z_k/\sqrt n)\bigr]\Bigr)
  + \frac \ve n \sum_{j=1}^n \mathfrak{g}(x_j).
\end{align*}
The first term (with the supremum over~$k$) goes to zero for
each~$\ve$ by the weak law of large numbers.  Again by the weak law of
large numbers, the second term is bounded in probability by a constant
times~$\ve$ as $n \rightarrow \infty$.  Since $\ve$ is arbitrary the
proof is complete except for showing the existence of~$\mathfrak{f}$
and~$\mathfrak{g}$.  Details are provided for the first, as the second
is similar.

%
Suppose that $w$ and~$z$ are two points in~$\MM$.  Let
$$
  \gamma_n(t) = \exp_\bmu t\frac Z{\|Z\|}\text{ for }t \in \bigl[0,\|Z\|/\sqrt n\,\bigr]
$$
be the unit speed geodesic from~$\bmu$ to~$z/\sqrt n$ and $Y_n(t) =
\log_\bmu \gamma_n(t)$.  Then the Taylor expansion at $\bmu$ of the
half square-distance function $\rho_w$ evaluated at~$z/\sqrt n$ reads
$$
  \rho_w(z/\sqrt n) = \rho_w(\bmu) - \<\log_\bmu w, Z/\sqrt n\>_\bmu
  +
  \frac 12 \|Y_n(t_0)\|^2 \lambda_w\bigl(Y_n(t_0)/t_0\bigr)
$$
for some $t_0 \in \bigl[0,\|Z\|/\sqrt n\,\bigr]$ by the Lagrange form
of the remainder, where the middle term on the right is
because $\nablabmu \rho_w(V) = -\<\log_\bmu w,V\>$ for all $V \in
\Tmu$ (this is \cite[Proposition~2.5]{tangential-collapse})
and the $\lambda_w$ term is as in Definition~\ref{d:amenable}.
Rearranging~yields
\begin{equation}\label{eq:taylor-expand-half-square-distance}
  \rho_w(z/\sqrt n) - \rho_w(\bmu) + \<\log_\bmu w, Z/\sqrt n\>_\bmu
  =
  \frac 12 \|Y_n(t_0)\|^2 \lambda_w\bigl(Y_n(t_0)/t_0\bigr).
\end{equation}
Because $\|Y_n(t)\| = t \leq \|Z\|/\sqrt n$, there is a universal
constant $C > 0$ such that
$$
  \|Y_n(t)\| < \sqrt 2 n^{-1/2}C
  \text{ for all } z \in U
  \text{ and } t \in \bigl[0,\|Z\|/\sqrt n\,\bigr].
$$
By amenability (Definition~\ref{d:amenable})
the second-order directional derivative $\lambda_w\bigl(Z/\|Z\|\bigr)$
of~$\rho_x$ is dominated by a $\mu$-integrable function
$\mathfrak{f}(w)$.  Thus \eqref{eq:taylor-expand-half-square-distance}
implies
\begin{align*}
|\rho_w(z/\sqrt n) - \rho_w(\bmu) + \<\log_\bmu w, Z/\sqrt n\>_\bmu|
   & \leq \frac12 (\sqrt 2 n^{-1/2}C)^2 \mathfrak{f}(w)
\\ & \leq n^{-1} C^2 \mathfrak{f}(w)
\end{align*}
for all $z \in U$.  Therefore, for all $z \in U$,
\begin{align*}
\bigl|\zeta_n(w, z/\sqrt n)\bigr|
  & = \bigl|\sqrt n\,\bigl(\rho_w(z/\sqrt n) - \rho_w(\bmu)
                           + \<\log_\bmu w, Z/\sqrt n\>_\bmu\bigr)\bigr|
\\
  & \leq n^{-1} C^2\mathfrak{f}(w),
\end{align*}
which implies that
$$
  \bigl|\sqrt n\,\zeta_n(w, z/\sqrt n)\bigr|
  \leq  C^2\mathfrak{f}(w),
$$
completing the proof.
\end{proof}

It is also necessary to recall a law of large numbers for the
empirical means~$\bmu_n$.  This result does not require the
measure~$\mu$ to be amenable.

\begin{thm}\label{thm:LLN-of-bmu_n}
For any measure~$\mu$ on a smoothly stratified metric space~$\MM$,
$$
  \lim_{n\to\infty} \dd(\bmu_n, \bmu) \overset{P}= 0.
$$
\end{thm}
\begin{proof}
See \cite{sturm2003} for $\CAT(0)$ spaces and \cite{Yo17} for
$\CAT(\kappa)$ with $\kappa > 0$.
\end{proof}

The following result is used in the proof that $\sqrt n
\log_\bmu\bmu_n$ is tight
(Proposition~\ref{p:tightness-of-sqrtn-mu_n}).

\begin{lemma}[{\cite[Theorem~5.52]{Van00}}]\label{l:thm552:vdV}
Let $C$ be a constant and $\alpha > \beta$ be positive integers.
Assume that for every $n$ and sufficiently small $\delta > 0$, the
Fr\'echet function~$F =\nolinebreak F_\mu$ and the empirical Fr\'echet
function~$F_n = F_{\mu_n}$ satisfy
\begin{equation}\label{eq:thm552:vdV:1a}
  \sup_{\dd(z,\bmu) < \delta} \bigl|F(x) - F(\bmu)\bigr|
  \geq
  C\delta^\alpha
\end{equation}
\vspace{-2ex}
\begin{equation}\label{eq:thm552:vdV:1b}
\makebox[0pt][r]{and\quad\ }
  \EE\Bigl[\sqrt n \sup_{z \in B(\bmu,\delta)} \bigl(F_n(z) - F(z) -
  F_n(\bmu) + F\bigl(\bmu)\bigr)\Bigr]
  \leq
  C\delta^\beta.
\end{equation}
If additionally $\dd(\bmu_n, \bmu) \xrightarrow{P} 0$ then
$n^{1/(2\alpha - 2\beta)} \dd(\bmu_n, \bmu)$ is tight.
\end{lemma}
\begin{proof}
The proof is adapted from \cite[Theorem~5.52]{Van00} to our setting
with some simplifications particular to our setting.
For notational brevity, set the scale \noheight{$r_n =
n^{\frac1{2\alpha -2\beta}}$}.

For each integer $n$ and $\ve > 0$, define
$$
  S_{j,n}^\ve
  =
  \{z \in \MM \mid 2^{j-1} < r_n \dd(z,\bmu) \leq 2^j, \dd(z,\bmu) < \ve\}.
$$
Observe that the event $\{ r_n\dd(\bmu_n,\bmu) > 2^M, \dd(\bmu,\bmu) <
\ve\}$ coincides with $\bigcup_{j\geq M} \{\bmu_n \!\in\nolinebreak
\!S_{j,n}^\ve\}$.  Let $\delta$ be as in \eqref{eq:thm552:vdV:1a} and
observe that fixing any $\ve < \delta$, \eqref{eq:thm552:vdV:1a}
guarantees that
\begin{align*}
  z \in S_{j,n}^\ve
  \implies
  F(\bmu) - F(\bmu_n) \leq -C \frac{2^{(j-1)\alpha}}{r_n^\alpha}\,.
\end{align*}
Combining this implication with the fact that $F_n(\bmu_n) - F_n(\bmu)
\leq 0$ almost surely, since $\bmu_n$ almost surely minimises~$F_n$,
produces
\begin{align*}
  \{\bmu_n \in S_{j,n}^\ve\}
  & = \bigl\{F(\bmu) - F(\bmu_n) \leq -C \tfrac{2^{(j-1)\alpha}}{r_n^\alpha}\bigr\}
      \cap
      \bigl\{\bmu_n \in S_{j,n}^\ve\bigr\}
\\*
  & = \bigl\{F(\bmu) - F(\bmu_n) + F_n(\bmu_n) - F_n(\bmu)
             \leq
             -C \tfrac{2^{(j-1)\alpha}}{r_n^\alpha}\bigr\}
      \cap
      \bigl\{\bmu_n \in S_{j,n}^\ve\bigr\}
\\*
  & \subseteq \Bigl\{\sup_{z \in \MM \,\mid\, \dd(z,\bmu) < 2^j/r_n}
                          n^{\frac12}| F_n(z) - F(z) + F(\bmu) - F_n(\bmu)|
                          \geq
			  C\tfrac{n^{\frac12} 2^{(j-1)\alpha}}{r_n^\alpha}\Bigr\}.
\end{align*}
Applying \eqref{eq:thm552:vdV:1b} and Markov's inequality to the
last term on the right above produces for $n$ sufficiently large
\begin{align*}
  \bP\bigl(\{r_n\dd(\bmu_n,\bmu) > 2^M\}
           \cap
           \{\dd(\bmu_n,\bmu)<\ve\}\bigr)
  & = \sum_{j \geq M} \bP\{\bmu_n \in S_{j,n}^\ve\}
\\
  & \leq \frac{r_n^{\alpha - \beta}}{n^{\frac12}}
         \sum_{j \geq M} \frac{C}{2^{j(\alpha - \beta)}}
    \leq C_{\alpha,\beta,\ve} 2^{-M(\alpha-\beta)}
\end{align*}
for some positive constant $C_{\alpha,\beta,\ve}$ which depends on
$\alpha$, $\beta$, and~$\ve$ but not on~$n$ or~$M$.  Using this
conclusion, observe that
\begin{align*}
  \bP\bigl(\{r_n\dd(\bmu_n,\bmu) > 2^M\}\bigr) =
  & \ \bP\bigl(\{r_n\dd(\bmu_n, \bmu) > 2^M\}
               \cap
               \{\dd(\bmu_n, \bmu) < \ve\}\bigr)
\\
  &\qquad + \bP\bigl(\{r_n\dd(\bmu_n, \bmu) > 2^M\}
                     \cap
                     \{\dd(\bmu_n,\bmu) \geq \ve\}\bigr)
\\
  \leq
  &\ C_{\alpha,\beta,\ve} 2^{-M(\alpha-\beta)}
     + \bP\bigl(\dd(\bmu_n,\bmu) \geq \ve\bigr).
\end{align*}
Since $\dd(\bmu_n,\bmu) \rightarrow 0$ in probability the last term
goes to~$0$ for all $\ve > 0$ as $n \rightarrow \infty$.
\end{proof}

\begin{lemma}\label{l:tightness-of_sqrtn_mu_n:sup_1}
Assume that $\mu$ is a probability measure such that $\int_\MM d^2(x,
\bmu) \mu(dx) \leq K < \infty$.  Then for any sufficiently small
$\delta > 0$ and any $n \in \NN$,
\begin{align*}
  \EE\Bigl[\sqrt n \sup_{z \in B(\bmu,\delta)}
           \bigl|F_n(z) - F(z) + F(\bmu) - F_n(\bmu)\bigr|\Bigr]
  &\leq C\delta,
\end{align*}
where the constant~$C$ does not depend on~$\delta$.
\end{lemma}
\begin{proof}
The proof is essentially the same as the proof of
\cite[Corollary~5.53]{Van00}.  There they employ \cite[Lemma~19.34 and
Corollary~19.35]{Van00} to obtain the result via a chaining argument.
The proof in \cite[Corolary~5.53]{Van00} is written in Euclidian
space, but the results in \cite[Lemma~19.34 and
Corollary~19.35]{Van00} are written for a general measure space.
There are two facts in our setting which are needed to apply the
result.  First, in the notation of \cite[Lemma~19.34, Corollary~19.35,
and Example~19.7]{Van00}, it needs to be shown that $|\dd^2(w_1,z) -
\dd^2(w_2,z)| \leq \dot m (z) \dd(w_1,w_2)$ for all $w_1,w_2 \in
B(\bmu,\delta)$ for $\delta$ sufficiently small and that $\int \dot m
(z)^2 \mu(dz) \leq \infty$.  Observe that for $w_1,w_2 \in
B(\bmu,\delta)$,
\begin{align*}
  |\dd^2(w_1,z) - \dd^2(w_2,z)|
  & = \abs{\dd(w_1,z) + \dd(w_2,z)} \abs{\dd(w_1,z) - \dd(w_2,z)}
\\*
  &\leq \bigl(\dd(w_1,z)+ \dd(w_2,z)\bigr) \dd(w_1,w_2)
\\*
  &\leq 2\bigl(\dd(\bmu,z)+\delta\bigr) \dd(w_1,w_2).
\end{align*}
So for $\delta$ sufficiently small, take $\dot m(z) = 2 \dd(\bmu,z) +
1$.  By assumption, $\int \dot m(z)^2 \mu(dz) \leq\nolinebreak
\infty$.

The last needed ingredient is the entropy, which captures how log of
the $\ve$-covering number scales with~$\ve$, as discussed in
\cite[Section~3.3, especially~Eq.~(3.2)]{random-tangent-fields}.
\cite[Example~19.7]{Van00} explains how to demonstrate that
$\log(\frac1\ve)$ bounds the entropy in our setting when $z$ is in
Euclidian space.  Since locally our stratified space is just a union
of a countable (in fact, finite) number of smooth manifolds, the
result carries over to our setting by adding the log of the covering
number of each of the strata.  The result still scales like
$\log(1/\ve)$.  This reasoning is covered in
\cite[Lemma~3.6]{random-tangent-fields}.
\end{proof}

\begin{remark}\label{r:chaining}
The fact that these chaining proofs can be adapted to our setting is
not surprising, as the proofs of \cite[Lemma~19.34 and
Corolary~19.35]{Van00} are more sophisticated versions of the chaining
proof given in detail in \cite[Theorem~3.1 and
Lemma~5.8]{random-tangent-fields}.
\end{remark}

Now the tightness of $\sqrt n \log_\bmu \bmu_n$ can be addressed.

\begin{prop}\label{p:tightness-of-sqrtn-mu_n}
For any amenable measure~$\mu$ on a smoothly stratified metric space,
the sequence $\sqrt n \dd(\bmu_n, \bmu)$ is asymptotically bounded in
probability as $n \rightarrow \infty$.
\end{prop}
\begin{proof}
Apply Lemma~\ref{l:thm552:vdV} for $\alpha = 2$ and $\beta = 1$.  The
LLN in Theorem~\ref{thm:LLN-of-bmu_n} says that $\dd(\bmu_n, \bmu)
\xrightarrow{P} 0$.  It remains to check the two inequalities in
Lemma~\ref{l:thm552:vdV}.  The second one follows from
Lemma~\ref{l:tightness-of_sqrtn_mu_n:sup_1}.  Next, recall from
Lemma~\ref{l:uniform-convexity} that there is a universal
constant~$C$, independent of the measure~$\mu$, such that for all~$x$
in some neighborhood $B(\bmu,\delta)$ of~$\bmu$ with $\delta <
R_{\kappa/2}$,
$$
  |F(x) - F(\bmu)| \geq C\dd^2(x,\bmu).
$$
Thus the first inequality in Lemma~\ref{l:thm552:vdV} also holds.
\end{proof}

\begin{remark}\label{r:law-from-perturbation}
The following results, culminating in Proposition~\ref{p:transform},
imply that to determine the limit law for the empirical Fr\'echet
mean~$\bmu_n$, it suffices to instead minimize the population
Fr\'echet function after it is perturbed radially by an infinitesimal
tangent field defined from a finite sample.  That is because, in
Proposition~\ref{p:transform}, the convergence of~$h_n$ to~$\bmu_n$
occurs even after stretching by a factor of~$\sqrt n$.
\end{remark}

\begin{defn}\label{d:average-sample-perturbation}%
Fix a measure~$\mu$ on a smoothly stratified space~$\MM$.  The
empirical tangent field $\bgn$ induced by~$\mu$
(Definition~\ref{d:empirical-tangent-field}) yields the \emph{average
\mbox{sample}~\mbox{perturbation}}
$$
  h_n \in \argmin_{x\in\MM} \bigl(F(x) - \bgn(\log_\bmu x)\bigr),
$$
where again $\bgn$ is considered a random tangent field on~$\Tmu$ via
Remark~\ref{r:extend-to-Tmu}.
\end{defn}

\begin{remark}\label{r:perturbation}
The $\argmin$ defining an average sample perturbation is a set rather
than a point, even when considered deterministically.  But the choice
of point in that set is irrelevant, since all choices behave the same
way for convergence purposes, as~follows.
\end{remark}

\pagebreak[2]

\begin{lemma}\label{l:tight-h_n}
Any rescaled average sample perturbation $\sqrt n \log_\bmu h_n$ is
tight.
\end{lemma} 
\begin{proof}
First note that $\lim_{n\to\infty} \dd(\bmu, h_n) = 0$.  To see why,
Lemma~\ref{l:uniform-convexity} says that there is a constant $C > 0$
satisfying $F(x) - F(\bmu) \geq C\dd^2(\bmu,x)$.  On the other hand,
because $\bgn$ is homogeneous as a random tangent field on~$\Tmu$ by
Remark~\ref{r:extend-to-Tmu},
$$
  \bgn(\log_\bmu x) - \bgn(\log_\bmu \bmu)
  =
  \bgn(\log_\bmu x)
  =
  \dd(\bmu, x) \bgn\bigl(\theta_x\bigr),
$$
where $\theta_x$ is the direction of~$x$ in polar coordinates
(Definition~\ref{d:polar}).  Thus
$$
  \bgn(\log_\bmu x) - \bgn(\log_\bmu\bmu)
  \leq
  \dd(\bmu, x) \sup_{\theta\in\Smu} \bgn(\theta)
  =
  \frac{\dd(\bmu, x)}{\sqrt n} \sup_{\theta\in \Smu} \sqrt n\,\bgn(\theta).
$$
Subtracting this equation from the equation in
Lemma~\ref{l:uniform-convexity} gives
$$
  F(x) - \bgn(\log_\bmu x) - F(\bmu) + \bgn(\log_\bmu \bmu)
  \geq
  C\dd^2(\bmu, h_n) - \frac{\dd(\bmu,h_n)}{\sqrt n}
    \sup_{\theta\in \Smu} \sqrt n\,\bgn(\theta).
$$
Replacing $x$ by $h_n$ and using that~$h_n$ is a minimizer yields
$$
  0\geq F(h_n)-\bgn(\log_\bmu h_n)-F(\bmu)+\bgn(\log_\bmu\bmu) \geq
  C\dd^2(\bmu,h_n)-\frac{\dd(\bmu,h_n)}{\sqrt{n}}\sup_{\theta\in
  \Smu}\sqrt n\,\bgn(\theta),
$$
which implies that
\begin{equation}\label{eq:lem:tight-h_n:3}
  \frac{\sup_{\theta \in \Smu} \sqrt n\,\bgn(\theta)}{C}
  \geq
  \sqrt n\,\dd(\bmu,h_n).
\end{equation}
\cite[Lemma~5.8]{random-tangent-fields}
says that $\Gn = \sqrt n\,\bgn$ is tight in $\cC(\Smu,\RR)$.  In
particular, $\sup_{\theta \in \Smu} \sqrt n \,\bgn(\theta)$ is tight.
Thus \eqref{eq:lem:tight-h_n:3} implies that $\sqrt n\,\dd(h_n, \bmu)$
is also tight, completing the proof.
\end{proof}

\begin{prop}\label{p:transform}
Fix an amenable measure~$\mu$ on a smoothly stratified metric space.
Any average sample perturbation~$h_n$
(Definition~\ref{d:average-sample-perturbation}) converges in
probability to the Fr\'echet mean faster than $1/\sqrt n$:
$$
  \sqrt n\, \dd(h_n, \bmu_n) \xrightarrow{P} 0.
$$
\end{prop} 
\begin{proof}
First expand, using the zoomed-in half square-distances $\zeta_n$ from
Lemma~\ref{l:converge-of-hessian}:
\begin{align}\nonumber
n^{1/2}\Bigl(\frac 1n \sum_{i=1}^n
   & \zeta_n(x_i, z/\sqrt n) - \EE\bigl[\zeta_n(x,z/\sqrt n)\bigr]\Bigr)
\\*\label{eq:CLT_zeta_n:1}%
   & = n\biggl(\frac 1n \Bigl(  \sum_{j=1}^n \rho_{x_j}(z/\sqrt n)
                               - \sum_{j=1}^n \rho_{x_j}(\bmu)
                               + \sum_{j=1}^n \<X_j,Z/\sqrt n\>\Bigr)
\\\nonumber
  &\phantom{\mbox{}= n\biggl(\biggr.}\ \
               - \int_\MM\rho_x(z/\sqrt n) \mu(dx)
               + \int_\MM\rho_x(\bmu) \mu(dx)
               - \int_\MM\<X,Z/\sqrt n\> \mu(dx)
        \biggr).
\end{align}
Noting that $F_n(z) = \frac 1n \sum_{j=1}^n \rho_{x_j}(z)$ and $F(z) =
\int_\MM \rho_x(z)\mu(dx)$, Eq.~\eqref{eq:CLT_zeta_n:1} reduces to
\begin{equation*}
\begin{split}
n^{1/2}\Bigl(\frac 1n \sum_{i=1}^n
   & \zeta_n(x_i,z/\sqrt n) - \EE\bigl[\zeta_n(x,z/\sqrt n)\bigr]\Bigr)
\\
   & = n\bigl(F_n(z/\sqrt n)-F_n(\bmu)\bigr)
     + nF(\bmu) - n\bigr(F(z/\sqrt n) - \bgn(Z/\sqrt n)\bigr).
\end{split}
\end{equation*}
Lemma~\ref{l:converge-of-hessian} implies that the left-hand side is a
sequence of stochastic processes converging uniformly to~$0$ for~$z$
in any compact set~$U$ around~$\bmu$.  Hence the equation~implies
\begin{equation}\label{eq:proof_partly_sticky:taylor_pre_CLT_conclusion}
  \sup_{z \in U}
    \Bigl(n \bigl(F_n(z/\sqrt n) - F_n(\bmu)\bigr)
          + n F(\bmu)
          - n \bigl(F(z/\sqrt n) - \bgn(Z/\sqrt n)\bigr)\Bigr)
  =
  c_n,	
\end{equation}
where $c_n$ is a random real valued sequence, independent of~$z$,
converging to~$0$ in probability.  Thus $z$ can be replaced by any
tight random sequence in~$\MM$.  Thanks to Lemma~\ref{l:tight-h_n} and
Proposition~\ref{p:tightness-of-sqrtn-mu_n}, both
\begin{align*}
  \sqrt n\,\bmu_n &= \exp(\sqrt n \log_\bmu \bmu_n)
\\\text{and\quad}
  \sqrt n\,h_n    &= \exp(\sqrt n \log_\bmu h_n)
\end{align*}
are valid choices for~$z$.  Substituting them for $z$ in
\eqref{eq:proof_partly_sticky:taylor_pre_CLT_conclusion} produces
\begin{equation}\label{eq:CLT_zeta_n:2}
  n \bigl(F_n(h_n) - F_n(\bmu)\bigr)
  =
  -n F(\bmu) + n\bigl(F(h_n) - \bgn(\log_\bmu h_n) \bigr) + c_n
\end{equation}
and 
\begin{equation}\label{eq:CLT_zeta_n:2.1}
  n\bigl(F_n(\bmu_n) - F_n(\bmu)\bigr)
  =
  -n F(\bmu) + n\bigl(F(\bmu_n) - \bgn(\log_\bmu \bmu_n)\bigr) + c_n.	
\end{equation}
The definition of average sample perturbation~$h_n$
(Definition~\ref{d:average-sample-perturbation}) yields
$$
  h_n \in \argmin_{x\in \MM}
        \Bigl(-n F(\bmu) + n\bigl(F(x) - \bgn(\log_\bmu x)\bigr)\Bigr).
$$
Therefore
$$
  -n F(\bmu) + n\bigl(F(h_n) - \bgn(\log_\bmu x)\bigr)
  \leq
  -n F(\bmu) + n\bigl(F(\bmu_n) - \bgn(\log_\bmu x)\bigr).
$$
Combining this inequality with \eqref{eq:CLT_zeta_n:2} and
\eqref{eq:CLT_zeta_n:2.1} yields
\begin{align*}%
n\bigl(F_n(h_n) - F_n(\bmu)\bigr) - c_n
   &   =   -n F(\bmu) + n\bigl(F(h_n) - \bgn(\log_\bmu x)\bigr)
\\*& \leq  -n F(\bmu) + n\bigl(F(\bmu_n) - \bgn(\log_\bmu x)\bigr)
\\*&   =    n\bigl(F_n(\bmu_n) - F_n(\bmu)\bigr) - c_n.
\end{align*}
In short,
\begin{equation}\label{eq:CLT_zeta_n:7.1}
  n\bigl(F_n(h_n) - F_n(\bmu)\bigr)
  \leq
  n\bigl(F_n(\bmu_n) - F_n(\bmu)\bigr) - 2c_n.
\end{equation}
On the other hand, $\bmu_n = \argmin_{x\in\MM} F_n(x)$, so
$$
  \bmu_n = \argmin_{x\in\MM}\,n\bigl(F_n(x) - F_n(\bmu)\bigr).
$$
Therefore, for every~$n$,
$$
  n\bigl(F_n(\bmu_n) - F_n(\bmu)\bigr)
  \leq
  n \bigl(F_n(h_n) - F_n(\bmu)\bigr).
$$
{}From this inequality and~\eqref{eq:CLT_zeta_n:7.1} it follows that
\begin{equation}\label{eq:dist_ell_n_and_r_n_1}
  \bigl|n\bigl(F_n(\bmu_n) - F_n(h_n)\bigr)\bigr| \leq 2c_n.
\end{equation}
Lemma~\ref{l:uniform-convexity} produces a universal constant~$C$,
independent of the measure~$\mu$, such that
$$
  \bigl|F_n(\bmu_n) - F_n(h_n)\bigr|
  \geq
  C\dd^2(\bmu_n,h_n)
$$
for some positive constant $C$.  This inequality, in conjunction
with~\eqref{eq:dist_ell_n_and_r_n_1}, yields
$$
  \dd(\bmu_n,h_n)^2 \leq \frac{2c_n}{nC},
$$
or equivalently
$$
  \sqrt n\,\dd(\bmu_n,h_n)
  \leq
  \sqrt{\frac{2c_n}{C}},
$$
which implies the Proposition since $c_n \xrightarrow{P} 0$ as $n \to
\infty$ by~\eqref{eq:proof_partly_sticky:taylor_pre_CLT_conclusion}.
\end{proof}

\subsection{Escape convergence and confinement}\label{b:convergence}
\mbox{}\medskip

\noindent
Escape vector theory rests on convergence of three flavors sequences
that bring uniqueness, homogeneity, duality, an explicit formula, and
confinement to the escape cone.

\begin{notation}\label{n:deterministic}
Fix a measure~$\mu$ on a smoothly stratified metric space~$\MM$~and
\begin{itemize}
\item%
a measure $\Delta = \lambda_1\delta_{Y^1} + \dots +
\lambda_j\delta_{Y^j}$ sampled from~$\Tmu$;
\item%
a sequence of measures $\Delta_n$ sampled from~$\Tmu$ converging
weakly to~$\Delta$
and
\item%
a sequence of positive real numbers $t_n$ converging to~$0$.
\end{itemize}
Using inner products (Definition~\ref{d:inner-product}), polar
coordinates (Definition~\ref{d:polar}), and the second-order Taylor
coefficient~$\Lambda_\bmu$ of the Fr\'echet function
(Proposition~\ref{p:Lambda}), define
\begin{align*}
  b_n &\in\argmin_{x\in\exp_\bmu(\Emue)}\bigl(F(x) - t_n\<\Delta_n,\log_\bmu x\>\bigr)
\\\text{and}\quad
  c_n &\in\argmin_{x\in\exp_\bmu(\Emue)}\bigl(r_x^2\Lambda_\bmu(\theta_x) -
                                              t_n\<\Delta_n,\log_\bmu x\>\bigr)
\intertext{if the measure~$\mu$ has been assumed amenable.  Even
without assuming amenability, if
\begin{itemize}
\item%
$\delta_n \to \delta$ is a convergent sequence of measures sampled
from~$\MM$, then
\item%
set $\Delta = \log_\bmu\delta$ and $\Delta_n = \log_\bmu\delta_n$ for
all~$n$,
\end{itemize}
and use the Fr\'echet function $F_{\delta_n}$ (see
Definition~\ref{d:sampled-from}) to define}
  a_n & = \argmin_{x\in\exp_\bmu(\Emue)}\bigl(F(x) + t_n F_{\delta_n}(x)\bigr).
\end{align*}
\end{notation}

\begin{remark}\label{r:vanishing-linear-term}
For $x \in \exp_\bmu\Emue$, the (directional) Taylor expansion of the
Fr\'echet function in Proposition~\ref{p:Lambda} becomes
$$
  F(x) = F(\bmu) + r_x^2\Lambda_\bmu(\theta_x) + o(r_x^2)
$$
because $\nablabmu F(\theta_x) = 0$.  Therefore the function minimized
for~$b_n$ is the function minimized for~$c_n$ plus a constant and a
hyperquadratic term, namely $F(\bmu) + o(r_x^2)$.
\end{remark}

\begin{remark}\label{r:V-can-be-exponentiated}
The sequences $b_n$ and~$c_n$ are defined regardless of whether
$\Delta_n$ is expressible as the logarithm of a measure sampled
from~$\MM$.  Of course, $\Delta_n$ is always a scalar multiple of such
a logarithm, namely $\Delta_n = \frac 1t\log_\bmu\exp_\bmu t\Delta_n$
for any positive~$t \ll 1$ by
Definition~\ref{d:stratified-space}.\ref{i:exponentiable}.  The fact
that $\Delta$ need not itself be a logarithm becomes irrelevant
anyway: by Corollary~\ref{c:homogeneity}, the limiting quantities
derived from~$a_n$, $b_n$, and $c_n$ are homogeneous---they behave
well with respect to scaling~$\Delta$ in the sense of
Definition~\ref{d:scaling}---so $\Delta$ can be shrunk to any desired
length for the purpose of comparing $b_n$ to~$a_n$, as long
as~$\Delta$ is stretched back to its original length later.
\end{remark}

\begin{remark}\label{r:X-can-be-exponentiated}
That the minimizations in Notation~\ref{n:deterministic} are taken
over a subset of the escape cone and not all of it also causes no
problems because the points~$x$ of interest lie arbitrarily close
to~$\bmu$ for large~$n$.  Indeed, since $t_n \to 0$, the three
sequences all minimize small perturbations of the Fr\'echet function;
see Lemma~\ref{l:magnitude-of-a_n} for an upper bound.
\end{remark}

\begin{remark}\label{r:a_n-b_n-c_n}
The goal is to find the limit of $\log_\bmu b_n/t_n$ as $n \to \infty$
in Theorem~\ref{t:escape-vector-confinement}, ideally---if
the limit exists---without any dependence on the choice of~$b_n$ in
the $\argmin$, which is selected from a set and need not be unique
(Remark~\ref{r:perturbation}).  The idea is to use~$a_n$ since it is
the (unique) minimizer of a convex function on~$\MM$ and is also
automatically confined to the escape cone~$\Emu$
(Proposition~\ref{p:unconfined-analogue}).  If $\dd(a_n,\bmu) \leq K
t_n$ for all~$n$ and a finite constant~$K$, as is indeed proved in
Lemma~\ref{l:magnitude-of-a_n}, then Taylor expansions of~$\Fdn(x)$
and~$F(x)$ can be used to approximate~$a_n$ to order~$o(t_n)$.
Consequently, in Proposition~\ref{p:unique-c-and-limits} the limits
$$
  \lim_{n\to\infty}\frac{\log_\bmu a_n}{t_n}
  =
  \lim_{n\to\infty}\frac{\log_\bmu b_n}{t_n}
  =
  \lim_{n\to\infty}\frac{\log_\bmu c_n}{t_n}
$$
are shown to be equal and hence forced to be independent of potential
$\argmin$ selections for~$b_n$ and~$c_n$.  The advantage of~$c_n$ is
the ability to solve for it explicitly (Lemma~\ref{l:formula-c_n}) in
terms of the second-order Taylor coefficient.  To summarize: $b_n$~is
the one that's needed via Proposition~\ref{p:transform}; $a_n$~is the
one that's unique and confined to~$\Emu$ (as well as interpretable,
cf.~Remark~\ref{r:rho-delta}); and $c_n$ is the one that admits
explicit~solution.
\end{remark}

\begin{lemma}\label{l:abc-bounded}
Fix a measure~$\mu$ on a smoothly stratified metric space.  Using
Notation~\ref{n:deterministic}, $\lim_{n\to\infty} \dd(\bmu, a_n) =
0$.  The same holds for $b_n$ or~$c_n$ if $\mu$ is amenable.
\end{lemma}
\begin{proof}
Similar arguments work for any of~$a_n$, $b_n$, or~$c_n$; only the
$a_n$ case is presented in detail.  For $x$ outside of the closed
ball~$\oB(\bmu,\ve)$ of radius~$\ve$ around~$\bmu$,
Lemma~\ref{l:uniform-convexity} produces a constant~$C$ with
\begin{equation}\label{eq:lem:magnitude-of_a_n:2}
  F(x) - F(\bmu) \geq C \ve^2.
\end{equation}
On the other hand, because $\lim_{n\to\infty} \delta_n = \delta$,
there exists~$N_0$ and some positive, finite~$K$ such that, for $n
\geq N_0$,
$$
  \Fdn(\bmu) \leq 2\Fd(\bmu) + 1 \leq K.
$$
It follows that for any $x\in \MM$ and $n\geq N_0$,
$$
  \Fdn(x) - \Fdn(\bmu)
  \geq
  -\Fdn(\bmu)
  \geq
  -2\Fd(\bmu) - 1
  \geq
  -K.
$$
Multiplying the leftmost and rightmost sides by~$t_n$ and adding
corresponding sides to~\eqref{eq:lem:magnitude-of_a_n:2} gives
$$
  F(x) + t_n\Fdn(x) - F(\bmu) - t_n\rho_\bmu(x)
  \geq
  C\ve^2 - t_n K
$$
for $x \notin \oB(\bmu,\ve)$.  Hence, for $n\geq N_0$,
$$
  F(x) + t_n\Fdn(x)
  >
  F(\bmu) + t_n\rho_\bmu(x)
  \text{ if }
  x \notin \oB(\bmu,\ve),
$$
and thus $\argmin_{x\in\Emu} \bigl(F(x) + t_n \Fdn(x)\bigr)
\subseteq \oB(\bmu,\ve)$, or $a_n \in \oB(\bmu,\ve)$ for $n \geq N_0$.
\end{proof}

\begin{lemma}\label{l:magnitude-of-a_n}
Fix a measure~$\mu$ on a smoothly stratified metric space.
Using~$a_n$ from Notation~\ref{n:deterministic}, there exists a finite
constant $K$ such that for $n$ sufficiently large,
$$
  \dd(a_n,\bmu) \leq K t_n.
$$
\end{lemma}
\begin{proof}
Starting from Lemma~\ref{l:abc-bounded}, the next step is to show that
$t_n \Fdn(x)$ is Lipschitz of modulus~$L t_n$ in $\oB(\bmu,\ve)$ for
some constant~$L$.  Suppose first that $\delta_n$ is supported on a
single point~$v_n$, so $\Fdn = \rho_{v_n}$.  For $x, y \in
\oB(\bmu,\ve)$ and $n$ large enough,
\begin{align*}
t_n\rho_{v_n}(x)-t_n\rho_{v_n}(y)
     &  = \frac{t_n}2 \bigl(\dd(v_n,x)-\dd(v_n,y)\bigr)\bigl(\dd(v_n,x)+\dd(v_n,y)\bigr)
\\*  &\leq\frac{t_n}2 \dd(x,y)\bigl(\dd(v_n,\bmu) + \ve + \dd(v_n,\bmu) + \ve\bigr)
\\*  &\leq L t_n \dd(x,y)
\end{align*}
for $L = \frac 12 \bigl(\dd(v_n,\bmu) + \dd(v_n,\bmu) + 2\ve\bigr)$,
where the middle line of the display is by the triangle inequality in
both factors on the right-hand side.  In the general case, when
$\delta_n$ is supported at~$j$ points, $t_n\Fdn(x) -
t_n\rho_{v_n}(y)$ is a sum of~$j$ terms, each bounded by an
expression of the form $L t_n \dd(x,y)$, for similarly defined~$L$
independent of~$x$ and~$y$.

Now invoke \cite[Proposition~4.32]{bonnans-shapiro2013}
for the functions $f(x) = F(x) + t_n\rho_{v_n}(x)$ and $g(x) = F(x)$
on~$\Emu$
to deduce that
$$
  \dd\Bigl(\argmin_{x \in \exp_\bmu \Emue\!} f(x),
           \argmin_{x \in \exp_\bmu \Emue\!} g(x)\Bigr)
  \leq
  K t_n
$$
for some finite constant~$K$.  In other words, for large $n$,
$\dd(a_n,\bmu) \leq K t_n$ as desired.
\end{proof}

\begin{prop}\label{p:unconfined-analogue}
Using Notation~\ref{n:deterministic}, the unconfined analogue
$$
  \ta_n = \argmin_{x\in\exp_\bmu(\Tmue)}\bigl(F(x) + t_n\Fdn(x)\bigr)
$$
of~$a_n$ is confined to the escape cone:
$
  \log_\bmu \ta_n = \log_\bmu a_n \in \Emu \text{ for all } n \gg 0.
$
\end{prop}
\begin{proof}
Let $V \in \Tmue$ be an exponentiable vector outside of the escape
cone~$\Emu$.  The linear term in the Taylor expansion at~$\bmu$ of~$F$
along~$V$ is strictly positive by Remark~\ref{r:nablamuF(V)=m(mu,V)}.
The linear term in the Taylor expansion at~$\bmu$ of $t_n\Fdn$ is
proportional to~$t_n$ by
Definition~\ref{d:sampled-from}
the fact that $\nablaq \rho_w = -\log_q w$
\cite[Proposition~2.5]{tangential-collapse}.
Therefore the linear coefficients in the Taylor expansions of $F +
t_n\Fdn$ along~$V$ for all~$n$ are bounded below by a strictly
positive number, since $\Fdn \to \Fd$ as~$\delta_n \to \delta$.  As
$\delta_n \to \delta$, the sequence $\frac 1{t_n}\log_\bmu a_n$
eventually remains within a ball of any given radius~$\ve > 0$
around~$\bmu$ for all $n \gg 0$ by Lemma~\ref{l:magnitude-of-a_n}.
Hence $v = \exp_\bmu V$ can't be the minimizer~$\ta_n$ if $n \gg 0$
because~\mbox{$F(v) + t_n\Fdn(v) > F(\bmu) + t_n\Fdn(\bmu)$}.
\end{proof}

An explicit formula for $c_n$ can be derived as follows.

\begin{lemma}\label{l:formula-c_n}
Using Notation~\ref{n:deterministic}, polar coordinates for~$c_n$
satisfy
$$
\theta_{c_n}
  \in\argmax_{x\in\Emue} \frac{\<\Delta_n,\theta_x\>}{\sqrt{\Lambda_\bmu(\theta_x)}}
\quad\text{and}\quad
r_{c_n}
   = \begin{cases}
     0 &\text{if } \<\Delta_n,\theta_{c_n}\> \leq 0
     \\[.5ex]\displaystyle
     t_n\frac{\<\Delta_n,\theta_{c_n}\>}
             {2\Lambda_\bmu(\theta_{c_n})}&\text{otherwise.}
\end{cases}
$$
\end{lemma}
\begin{proof}
The function that defines~$c_n$
in Notation~\ref{n:deterministic} can be written as a quadratic
expression in the radial coordinate~$r_x$:
\begin{align*}
r_x^2 \Lambda_\bmu (\theta_x) - t_n\<\Delta_n, \log_\bmu x\>
   & = r_x^2\Lambda_\bmu(\theta_x) - t_n r_x\<\Delta_n,\theta_x\>
\\ & = \Lambda_\bmu(\theta_x)
       \biggl(r_x - \frac{t_n\<\Delta_n,\theta_x\>}
                        {2\Lambda_\bmu(\theta_x)}\biggr)^2
       - \frac{t_n^2\<\Delta_n,\theta_x\>^2}{4\Lambda_\bmu (\theta_x)}.
\end{align*}
By Lemma~\ref{l:uniform-convexity} $\Lambda_\bmu (\theta_x)$ is always
positive, and of course the radial coordinate~$r_x$ is nonnegative.
If $\<\Delta_n,\theta_{c_n}\> \leq 0$, then the squared expression
containing~$r_x$ is minimized when $r_x = 0$.
But if $\<\Delta_n,\theta_{c_n}\> > 0$, then the squared
$r_x$~expression is minimized---without altering~$\theta_x$---when
that entire expression vanishes, in which case $r_x$ is as claimed,
and minimizing (twice the square root of) the other term yields the
formula for~$\theta_{c_n}$.
\end{proof}

Now having boundedness of~$a_n$ (Lemma~\ref{l:magnitude-of-a_n}),
confinement of its a~priori unconfined analogue
(Proposition~\ref{p:unconfined-analogue}), and an explicit formula
for~$c_n$ (Lemma~\ref{l:formula-c_n}), Remark~\ref{r:a_n-b_n-c_n} can
be fulfilled: the sequence~$a_n$ is approximated by~$b_n$ and by~$c_n$
to order~$o(t_n)$.

\pagebreak[3]

\begin{prop}\label{p:unique-c-and-limits}
Using Notation~\ref{n:deterministic},
$$
  \lim_{n\to\infty}\frac{\log_\bmu a_n}{t_n}
  =
  \lim_{n\to\infty}\frac{\log_\bmu b_n}{t_n}
  =
  \lim_{n\to\infty}\frac{\log_\bmu c_n}{t_n}
  =
  \ev
$$
is the unique vector whose polar
coordinates~(Definition~\ref{d:polar}) satsify
$$
  \theta_\ev
  \in \argmax_{X\in\SEmu} \frac{\<\Delta,X\>}{\sqrt{\Lambda_\bmu (X)}}
  \quad\text{and}\quad
  r_\ev = \frac{\<\Delta,\theta_\ev\>^+}{2\Lambda_\bmu(\theta_\ev)},
$$
where the $\argmax$ is taken over the unit sphere~$\SEmu$ in the
escape cone~$\Emu$ from Def\-inition~\ref{d:escape-cone} and $\alpha^+ =
\max\{\alpha,0\}$, so the ``+'' exponent means to set $\alpha$ to~$0$
if~$\alpha < 0$.
\end{prop}
\begin{proof}
To prove the equality of limits it suffices to show that
$$
  \dd(a_n, b_n) = o(t_n)
  \text{ and }
  \dd(a_n, c_n) = o(t_n),
$$
because $\dd(a_n, b_n) = o(t_n)$ implies equality of the $a_n$
and~$b_n$ limits, and $\dd(a_n, c_n) = o(t_n)$ implies equality of the
$a_n$ and~$c_n$ limits, while equality of the $c_n$ limit with~$\ev$
is Lemma~\ref{l:formula-c_n}.  The proof of $\dd(a_n, b_n) = o(t_n)$
is provided in detail; the corresponding estimation for $c_n$ instead
of~$b_n$ works similarly.

By Lemma~\ref{l:abc-bounded}, $\dd(b_n, \bmu) \to 0$ as $n \to
\infty$.  As $\mu + t_n\delta_n$ is a measure, its Fr\'echet function
$K_n(x) = F(x) + t_n F_{\delta_n}(x)$ satisfies the strong convexity
property in Lemma~\ref{l:uniform-convexity}.  In particular, $K_n$ is
strongly convex on~$\exp_\bmu(\Emue)$.  Consequently, in addition to
$a_n = \argmin_{x \in \exp_\bmu(\Emue)} K_n(x)$,
Lemma~\ref{l:uniform-convexity} yields $K_n(x) - K_n(a_n) \geq
C\dd^2(x,a_n)$ for some positive constant~$C$.  In particular,
\begin{equation}\label{eq:0}
  K_n(b_n) - K_n(a_n) \geq C\dd^2(b_n, a_n).
\end{equation}
The Taylor expansion for $F_{\delta_n}(x)$ at $\bmu$ gives
$$
  K_n(x)
  =
  F(x) - t_n\<\Delta_n, \log_\bmu x\> + t_n\rho_{\delta_n}(\bmu)
       + t_n O\bigl(\dd^2(\bmu,b_n)\bigr).
$$
Applying the above equation to~\eqref{eq:0} yields
\begin{equation}\label{eq:1}
\begin{split}
  K_n(b_n) - K_n(a_n)
& = F(b_n) - t_n \<\Delta_n, \log_\bmu b_n\>
           - F(a_n)
           + t_n \<\Delta_n,\log_\bmu a_n\>
\\
& \phantom{\mbox{} = F(b_n)}
           + t_n O(\dd^2(\bmu,b_n))
           + t_n O(\dd^2(\bmu,a_n))
\\
& \geq C\dd^2(a_n,b_n).
\end{split}
\end{equation}
Since $b_n \in \argmin_{x\in\exp_\bmu(\Emue)} F(x) - t_n \<\Delta_n,
\log_\bmu x\>$ by Notation~\ref{n:deterministic},
$$
  F(b_n) - t_n\<\Delta_n, \log_\bmu b_n\> - F(a_n) + t_n\<\Delta_n,\log_\bmu a_n\>
  \leq
  0.
$$
Thus \eqref{eq:1} implies
\begin{equation}\label{eq:2}
  t_n O\bigl(\dd^2(\bmu,b_n)\bigr) + t_n O\bigl(\dd^2(\bmu,a_n)\bigr)
  \geq
  C\dd^2(a_n,b_n).
\end{equation}

\begin{claim}\label{claim:magnitude-of-b_n}
The result of Lemma~\ref{l:magnitude-of-a_n} holds for~$b_n$ instead
of~$a_n$ if the measure~$\mu$ is amenable: there is a finite constant
$K_b$ such that for $n$ sufficiently~large,
$$
  \dd(b_n,\bmu) \leq K_b t_n.
$$
\end{claim}
\begin{proof}[Proof of Claim.]
Assume conversely that $\{b_n\}$ has a subsequence $\{b_n'\}$ such
that
\begin{equation}\label{eq:3}
  \lim_{n\to\infty} \frac{\dd(b'_n,\bmu)}{t_n}
  =
  \infty.
\end{equation}
Combining Lemma~\ref{l:magnitude-of-a_n} with the triangle inequality
\pagebreak[3]%
for the three points $a_n$, $b_n$, and~$\bmu$ produces, for $n$
sufficiently large,
\begin{equation}\label{eq:4}
  \dd(b'_n,a_n)
  \geq
  \dd(b_n',\bmu) - \dd(\bmu,a_n)
  >
  \frac 12 \dd(b_n',\bmu).
\end{equation}
Combining~\eqref{eq:2} (for $b'_n$ instead of $b_n$) with \eqref{eq:4}
produces
$$
  t_n O\bigl(\dd^2(\bmu,b'_n)\bigr) + t_n O\bigl(\dd^2(\bmu,a_n)\bigr)
  \geq
  C\dd^2(a_n,b'_n)
  >
  \frac C4 \dd^2(\bmu,b'_n)
$$
which is a contradiction as $\dd(b_n',\bmu) > \dd(\bmu,a_n)$ for $n
\gg 0$ and $t_n \to 0$ as $n \to \infty$.
\end{proof}

\noindent
Continuing with the proof of Proposition~\ref{p:unique-c-and-limits},
Eq.~\eqref{eq:2} gives
$$
   t_nO(t_n^2) \geq C\dd^2(a_n,b_n)
$$
when combined with Lemma~\ref{l:magnitude-of-a_n} and
Claim~\ref{claim:magnitude-of-b_n}, so as desired,
$$
  \dd(a_n,b_n) = o(t_n).
$$

All that remains is uniqueness of~$\ev$, which is a simple consequence
of the equality of limits, because each point~$a_n$ minimizes a convex
function and is hence unique.
\end{proof}

The formulas in Lemma~\ref{l:formula-c_n} imply that~$\ev$ depends
homogeneously on~$\Delta$.

\begin{cor}\label{c:homogeneity}
For any $r\geq 0$, using the rescaling notation in
Definition~\ref{d:scaling},
\begin{align*}
  \lim_{n\to\infty}\frac 1t\argmin_{X\in\Emue}\bigl(F(\exp_\bmu X)-t\<\Delta r,X\>\bigr)
& =
  \lim_{n\to\infty}\frac 1t\argmin_{X\in\Emue}\bigl(F(\exp_\bmu X)-t\<r \Delta,X\>\bigr)
\\
& =
  r\lim_{t\to 0}\frac 1t\argmin_{X\in\Emue}\bigl(F(\exp_\bmu X) - t\<\Delta,X\>\bigr).
\end{align*}
\end{cor}
\begin{proof}
The formulas in Lemma~\ref{l:formula-c_n} are homogeneous
in~$\Delta_n$, so the $c_n$ limit in
Proposition~\ref{p:unique-c-and-limits} is homogeneous in~$\Delta_n$.
Therefore the same homogeneity holds for the desired~$b_n$ limit here,
which is equal to the $c_n$ limit by
Proposition~\ref{p:unique-c-and-limits}.
\end{proof}

Here is the result to which Section~\ref{s:escape} has been building:
arbitrary escape vectors $\ev(\Delta)$ from
Definition~\ref{d:delta-escape} are well defined and confined to the
escape cone~$\Emu \subseteq \Tmu$ from Definition~\ref{d:escape-cone},
with an explicit polar coordinate (Definition~\ref{d:polar}) formula
in terms of the second-order Taylor coefficient~$\Lambda_\bmu$
(Proposition~\ref{p:Lambda}).  (The word ``arbitrary'' here
distinguishes from the case where the mass added to~$\mu$ is
restricted to the support of~$\mu$, which begets further confinement
to the fluctuating cone, cf.~Corollary~\ref{c:confinement-to-Cmu}.)

\pagebreak[3]%

\begin{thm}\label{t:escape-vector-confinement}
Fix an amenable measure~$\mu$ on a smoothly stratified metric
space~$\MM$, a~sequence of measures $\Delta_n$ sampled from~$\Tmu$
converging weakly to~$\Delta$,
and a sequence of positive real numbers $t_n \to 0$.  The escape
vector $\ev = \ev(\Delta)$ is well defined and
\begin{align}
\tag{$\ev$}\label{ev}
\ev(\Delta)
& =
  \lim_{t\to 0}\frac 1t
       \argmin_{X\in\Emue}\bigl(F(\exp_\bmu X) - t\<\Delta,X\>\bigr)
\\*\tag{b}\label{b}
& =
 \lim_{n\to\infty}\frac 1{t_n}
      \argmin_{X\in\Emue}\bigl(F(\exp_\bmu X) - t_n\<\Delta_n,X\>\bigr)
\\*\tag{c}\label{c}
& =
 \lim_{n\to\infty}\frac 1{t_n}
      \argmin_{X\in\Emue}\bigl(r_X^2\Lambda_\bmu(\theta_X) - t_n\<\Delta_n,X\>\bigr)
\end{align}
is confined to the escape cone~$\Emu$, with polar
coordinates
$$
  \theta_\ev
  \in \argmax_{X\in\SEmu} \frac{\<\Delta,X\>}{\sqrt{\Lambda_\bmu (X)}}
  \quad\text{and}\quad
  r_\ev = \frac{\<\Delta,\theta_\ev\>^+}{2\Lambda_\bmu(\theta_\ev)}
$$
where the $\argmax$ is over the unit sphere~$\SEmu$ in the escape
cone~$\Emu$ and $\alpha^+ \!\hspace{-1.15pt}=\hspace{-1.15pt}
\max\{\alpha,0\}$.  If, in addition, $\Delta_n = \log_\bmu\delta_n$
for all~$n$, then
\begin{align}
\tag{a}\label{a}
\ev(\Delta)
  =
  \lim_{n\to\infty}\frac 1{t_n}\!
                  \argmin_{\ \log_{\bmu\!}x\,\in\Emu}\bigl(F(x) + t_n\Fdn(x)\bigr).
  \qquad\
\end{align}
\end{thm}
\begin{proof}
The limits in the three right-hand sides of the top displayed equation
exist and have the given polar coordinates by
Proposition~\ref{p:unique-c-and-limits} because the middle of these
two is $\lim_{t\to\infty} \frac 1{t_n}\log_\bmu b_n$, the top one is
the special case where $\Delta_n = \Delta$ for all~$n$, and the bottom
one is $\lim_{t\to\infty} \frac 1{t_n}\log_\bmu c_n$.  To show that
these limits equal the escape vector of~$\Delta$, begin with $\Delta$
exponentiable.  In that case, the assertion follows from equality of
the $a_n$ and~$b_n$ limits in Proposition~\ref{p:unique-c-and-limits}
because the $a_n$ limit yields~$\ev(\Delta)$ by
Definition~\ref{d:delta-escape}.\ref{i:Delta}.  The general case in
Definition~\ref{d:delta-escape}.\ref{i:delta-rescale} reduces to the
exponentiable case by homogeneity in Corollary~\ref{c:homogeneity}
because (i)~some scalar multiple of~$\Delta$ is exponentiable by
Definition~\ref{d:stratified-space}, and (ii)~the definitions of~$b_n$
and~$c_n$ do not require exponentiability.

It remains to prove the final displayed equation~(\ref{a}) despite the
absence of any claim about covergence of the measures~$\delta_n$.
First use Lemma~\ref{l:magnitude-of-a_n} to deduce that restricting
the $\argmin$ to points~$x$ that are exponentials of vectors in~$\Emu$
suffices for all large~$n$ and hence, for the purpose of taking
limits, for all~$n$: the \noheight{$\frac 1{t_n}\argmin$} lands inside
a compact neighborhood of~$\bmu$ for all $n \gg 0$.  Fix any
convergent subsequence of $\frac 1{t_n}\argmin$ expressions indexed by
$n(i)$ for $i = 1,\dots,\infty$.  The corresponding
sequence~$\delta_{n(i)}$ of measures need not converge, but since
$\log_\bmu$ is a proper map (it preserves distance from~$\bmu$ in the
locally compact space~$\MM$), and the logarithmic image $\Delta_{n(i)}
= \log_\bmu\delta_{n(i)}$ converges, the sequence~$\delta_{n(i)}$ has
a convergent subsequence, to which
Proposition~\ref{p:unique-c-and-limits} applies and
produces~$\ev(\Delta)$ as the limit of the original convergent
subsequence of \noheight{$\frac 1{t_n}\argmin$} expressions.
Therefore every convergent subsequence of \noheight{$\frac
1{t_n}\argmin$} expressions converges to the same limit~$\ev(\Delta)$.
By compactness of the relevant neighborhood of~$\bmu$, the whole
sequence of \noheight{$\frac 1{t_n}\argmin$} expressions~converges.
\end{proof}

\begin{remark}\label{r:ev-continuity}
Theorem~\ref{t:escape-vector-confinement} implies continuity of the
escape map~$\ev$ as a function on measures sampled from~$\Tmu$; see
Corollary~\ref{c:escape-continuity}.  More importantly,
Theorem~\ref{t:escape-vector-confinement} asserts a
family-of-functions continuity: if positive real numbers $t_n \to 0$
are given and
$$
  \evn(\Delta)
  =
 \frac 1{t_n} \argmin_{X\in\Cmue}\bigl(F(\exp_\bmu X) - t_n\<\Delta,X\>\bigr)
$$
is the $n^\mathrm{th}$ \emph{escape approximation} of any
measure~$\Delta$ sampled from~$\Tmu$, then
Theorem~\ref{t:escape-vector-confinement} asserts that $\evn(\Delta_n)
\to \ev(\Delta)$ when $\Delta_n \to \Delta$.  A~functional version of
this convergence occupies Section~\ref{s:confinement}, tailored to fit
the hypotheses of a particular form of the continuous mapping theorem
\cite[Theorem~1.11.1]{VW13}, for application in the proof of
Theorem~\ref{t:perturbative-CLT}.
\end{remark}

\begin{cor}\label{c:escape-continuity}
The escape map~$\ev$ is continuous as a function on measures sampled
from~$\Tmu$.  In fact, the $\argmin$ in
Theorem~\ref{t:escape-vector-confinement}(\ref{c}) is equal
to~$t_n\ev(\Delta_n)$.
\end{cor}
\begin{proof}
The $\argmin$ in Theorem~\ref{t:escape-vector-confinement}(\ref{c}) is
$c_n$, so it equals $t_n\ev(\Delta_n)$ thanks to
Lemma~\ref{l:formula-c_n} and the polar coordinates in
Theorem~\ref{t:escape-vector-confinement}.  For continuity, apply
Theorem~\ref{t:escape-vector-confinement}(\ref{c}).
\end{proof}

When the measure $\mu$ is immured (Definition~\ref{d:immured}) and the
measure~$\delta$ is sampled from the support of~$\mu$, the escape
vector is further confined to the fluctuating cone~$\Cmu
\subseteq\nolinebreak \Emu$ from Definition~\ref{d:fluctuating-cone}.
The precise statement needs a lemma and a bit of notation.  The
immured hypothesis was designed specifically for Lemma~\ref{l:immured}
and Corollary~\ref{c:confinement-to-Cmu}.

\begin{lemma}\label{l:immured}
Fix an immured measure~$\mu$ on a smoothly stratified metric
space~$\MM$.  If~$\delta$ is a measure sampled from~$\supp\mu
\subseteq \MM$, then $\log_\bmu(\,\ol{\mu + t\delta}\,) \subseteq
\hull\mu$ for~all~$t \ll 1$.
\end{lemma}
\begin{proof}
The Fr\'echet mean set $\ol{\mu + t\delta}$ consists, by the law of
large numbers \cite{ziezold1977}, limits of Fr\'echet means of
measures sampled from~$\supp\mu = \supp(\mu +\nolinebreak t\delta)$.
By Lemma~\ref{l:magnitude-of-a_n}, for all positive $t \ll 1$ these
Fr\'echet means eventually lie in any neighborhood $U$ of~$\bmu$ as in
Definition~\ref{d:immured}.  As~$\mu$ is immured, the logarithms of
these Fr\'echet means therefore lie in~$\hull\mu$.
\end{proof}

\begin{defn}\label{d:Rmu}
Fix a conical $\CAT(0)$ space~$\XX$ (Definition~\ref{d:log-map}).  The
\emph{cone over} an arbitrary subset $\cS \subseteq \XX$ is the set
$\RR_+\cS$ of scalar multiples of elements of~$\cS$:
$$
  \RR_+\cS
  =
  \{\alpha X \mid X \in \cS \text{ and } \alpha \in \RR_+\}.
$$
When $\cS = \supp\hmu \subseteq \Tmu$, simplify notation by writing
$$
  \RR_+\mu = \RR_+\supp\hmu,
$$
and denote the closure of this set in~$\Tmu$ by~$\Rmu$.
\end{defn}

\begin{cor}\label{c:confinement-to-Cmu}\hspace{-1.2pt}%
Fix an immured amenable measure~$\mu$ on a smoothly stratified
space~$\MM\hspace{-1pt}$.  If a measure $\Delta =
\lambda_1\delta_{Y^1} + \dots + \lambda_j\delta_{Y^j}$ is sampled
from~$\Rmu$, then the escape vector of~$\Delta$ is confined to the
closed fluctuating cone: $\ev(\Delta) \in \oCmu$.
\end{cor}
\begin{proof}
The support of a pushforward measure under a continuous map is the
closure of the image of the support of the original measure.  On the
other hand, the logarithm $\log_\bmu$ is a proper map because it
preserves distance from~$\bmu$ and $\MM$ is locally compact.  Hence
$\supp\hmu = \log_\bmu(\supp\mu)$.  The support hypothesis on~$\Delta$
therefore means that for $i = 1,\dots,j$, each $Y^i \in \Rmu$ is
expressible as a limit $\alpha_i Z^i_n \to Y_i$ with $Z^i_n =
\log_\bmu z^i_n$ and $z^i_n \in \supp\mu$.  Set $\Delta_n =
\lambda_1\delta_{\alpha_1 Z^1_n} + \dots + \lambda_j\delta_{\alpha_j
Z^j_n}$ and use Definition~\ref{d:pair-with-discrete-measure} to
compute
\begin{equation}\label{eq:Deltas}
\begin{split}
\<\Delta_n,\,\cdot\,\>
  & = \lambda_1\<\alpha_1 Z^1_n,\,\cdot\,\>
      + \dots +
      \lambda_j\<\alpha_j Z^j_n,\,\cdot\,\>
\\*
  & = \lambda_1\alpha_1\<Z^1_n,\,\cdot\,\>
      + \dots +
      \lambda_j\alpha_j\<Z^j_n,\,\cdot\,\>
\\*
  & = \<\pDelta_n,\,\cdot\,\>
      \quad\text{for}\quad
      \pDelta_n = \lambda_1\alpha_1 Z^1_n
                  + \dots +
                  \lambda_j\alpha_j Z^j_n.
\end{split}
\end{equation}
Note that $\Delta_n \to \Delta$ and $\pDelta_n = \log_\bmu\pdelta_n$
for $\pdelta_n = \lambda_1\alpha_1 \delta_{z^1_n} + \dots +
\lambda_j\alpha_j \delta_{z^j_n}$.  Therefore
\begin{align*}
\ev(\Delta)
  & =
   \lim_{n\to\infty}\frac 1{t_n}
                    \argmin_{X\in\Emue}\bigl(F(\exp_\bmu X) - t_n\<\Delta_n,X\>\bigr)
\\*
  & =
   \lim_{n\to\infty}\frac 1{t_n}
                    \argmin_{X\in\Emue}\bigl(F(\exp_\bmu X) - t_n\<\pDelta_n,X\>\bigr)
\\*
  & =
  \lim_{n\to\infty}\frac 1{t_n}\!
                  \argmin_{\ \log_{\bmu\!}x\,\in\Emu}\bigl(F(x) + t_n\Fdnp(x)\bigr),
\end{align*}
where the middle equality is by~\eqref{eq:Deltas} and the other two
are by Theorem~\ref{t:escape-vector-confinement}.  The minimizer with
$\pdelta_n$ lies in $\hull\mu$ for $n \gg 0$ by Lemma~\ref{l:immured}
and in~$\Emu$ by construction.  Hence the limit $\ev(\Delta)$ lies in
the closure of $\Emu \cap \hull\mu$, which is $\oCmu$ by
Definition~\ref{d:fluctuating-cone}.
\end{proof}

\begin{remark}\label{r:ideally}
The immured hypothesis would be superfluous if one could prove
directly that the hull in~$\Tmu$ of the set of escape vectors of
points in~$\supp\mu$ maps isometrically to its image under tangential
collapse~$\LL$ as in Definition~\ref{d:collapse}.  It would suffice to
show that this hull satisfies
\cite[Theorem~3.17]{shadow-geom},
which says that the limit log map along a tangent vector~$Z$ induces
an isometry on any geodesically convex subcone containing at most one
ray in the shadow~$\IZ$.  Indeed, then $\Cmu$ could be replaced by
this hull throughout; the proof of Theorem~\ref{t:collapse} in
\cite{tangential-collapse} would remain essentially unchanged.
\end{remark}

\section{Functional confinement}\label{s:confinement}

\subsection{Tangent perturbations}\label{b:tangent-perturbations}
\mbox{}\medskip

\noindent
The main limit theorems in Section~\ref{s:CLT} have at their heart
just one protean convergence, namely Theorem~\ref{t:tangent-field-CLT}
for tangent fields.  The strategy is to transform that convergence
into other settings by the continuous mapping theorem (see
\cite[Theorem~3.27]{kallenberg1997}, for example, or
\cite[Theorem~1.11.1]{VW13} for the form applied here), in particular
in the perturbative CLT in Theorem~\ref{t:perturbative-CLT}.  This
subsection treats the deterministic prerequisites, translating the
convergences and confinements in Section~\ref{b:convergence} to a
functional setting (Theorem~\ref{t:convergence-of-Upsilon}) where the
tangent field CLT in Theorem~\ref{t:tangent-field-CLT} lives.

For the purpose of Theorem~\ref{t:convergence-of-Upsilon}, the input
measure is assumed to be sampled only from the support of~$\mu$
instead of arbitrarily from~$\MM$, to induce confinement to~$\Cmu$ via
Corollary~\ref{c:escape-continuity} instead of the escape cone~$\Emu$
via Theorem~\ref{t:escape-vector-confinement}.  This sampling
from~$\supp\mu$ instead of~$\MM$ is forced by a compactness argument
in Corollary~\ref{c:equiv-limits-of-td} based on isometric embedding
of the fluctuating cone~$\Cmu$ by tangential collapse from
Section~\ref{b:collapse}.  Of course, central limit theorems only care
about measures sampled from $\supp\mu$, so this limitation does not
pose an obstruction to the developments in Section~\ref{s:CLT}.

\begin{remark}\label{r:H}
In terms of average sample perturbations~$h_n$
(Definition~\ref{d:average-sample-perturbation}), our main CLT amounts
to identifying the limiting behavior of~$h_n$ for large~$n$, by
Proposition~\ref{p:transform}:
$$
  \lim_{n\to\infty} \sqrt n \log_\bmu \bmu_n
  \overset{d}=
  \lim_{n\to\infty} \sqrt n \log_\bmu h_n
  =
  \lim_{n\to\infty} \sqrt n \log_\bmu
  \argmin_{x\in\MM} \bigl(F(x) - \bgn(\log_\bmu x)\bigr).
$$
As the $\argmin$ takes place asymptotically in a neighborhood
of~$\bmu$, the random variable $H_n = \log_\bmu h_n$ is more compactly
expressed as an $\argmin$ over the tangent cone at~$\bmu$.
\end{remark}

\begin{defn}\label{d:H}
The \emph{empirical tangent perturbation} is the random variable
\begin{align*}
  H_n
  &\in
  \argmin_{X \in \Tmue}\bigl(F(\exp_\bmu X) - \bgn(X)\bigr)
\\
\intertext{valued in the subset~$\Tmue \subseteq \Tmu$ of vectors that
can be exponentiated, expressed in terms of the Fr\'echet function~$F$
and the average empirical tangent field~$\bgn$
(Definition~\ref{d:empirical-tangent-field}).  The \emph{Gaussian
tangent perturbation} is the $\Tmu$-valued random variable}
  H(t)
  &\in
  \argmin_{X \in \oCmue}\bigl(F(\exp_\bmu X) - tG(X)\bigr),
\end{align*}
for $t > 0$, where $G$ is the Gaussian tangent field induced by~$\mu$
(Definition~\ref{d:gaussian-tangent-field}) and~$\oCmue$ is the set of
exponentiable vectors in the fluctuating cone~$\Cmu$
(cf.~Definition~\ref{d:exponentiable}).
\end{defn}

\begin{remark}\label{r:thus-the-CLT}
Thus the CLT, which a~priori aims to identify the distribution of
$\lim_{n\to\infty} \sqrt n \log_\bmu \bmu_n$, instead need only
identify $\lim_{n\to\infty} \sqrt n H_n$, by
Proposition~\ref{p:transform}.  The idea is that the convergence
$\sqrt n\,\bgn \to G$ from Theorem~\ref{t:tangent-field-CLT} passes
appropriately through the~$\argmin$.  The purpose of this subsection
is to prove the deterministic version of this continuity, with the
random tangent fields $\sqrt n\,\bgn$ and~$G$ replaced by
deterministic continuous functions, in preparation for applying the
continuous mapping theorem in Theorem~\ref{t:perturbative-CLT}.
However, our method of proof to pass from $\lim_{n\to\infty} \sqrt n
H_n$ to $\lim_{t \to 0} \frac 1t H(t)$ cedes control over the
$\argmin$ in~$H(t)$ outside of the fluctuating cone~$\Cmu$---this is
especially focused in Lemma~\ref{l:represent}---which explains why the
$\argmin$ in~$H(t)$ is taken over~$\oCmue$ instead of~$\Tmue$.
\end{remark}

\begin{remark}\label{r:naturally-confined}
It would be nicer to prove that an unconfined version
$$
  \wt H(t) \in \argmin_{X \in \Tmue}\bigl(F(\exp_\bmu X) - tG(X)\bigr)
$$
is automatically confined to~$\Cmu$, so that $\lim_{t \to 0} \frac 1t
\wt H(t) = \lim_{t \to 0} \frac 1t H(t)$.  And indeed, this should be
true when the measure~$\mu$ is immured, because samples from~$\mu$
yield escape vectors that are automatically confined to~$\Cmu$ by
Corollary~\ref{c:confinement-to-Cmu}.  But it is not needed for the
main results, so it is left as an open question.
\end{remark}

\subsection{Representable limits}\label{b:limits}

\begin{defn}\label{d:sup-norm-on-continuous-functions}
Equip the set $\cC(\Tmu,\RR)$ of continuous, positively homogeneous
(commute with nonnegative scaling), real-valued functions on~$\Tmu$
with the sup norm
$$
  \|f\|_\infty = \sup_{Y \in \Smu} \|f(Y)\|.
$$
Functions in $\cC(\Tmu,\RR)$ may be considered as continuous functions
on the unit tangent sphere~$\Smu$ without explicit notation to denote
restriction to~$\Smu$.
\end{defn}

\begin{defn}\label{d:representable}
A function $R \in \cC(\Tmu, \RR)$ is \emph{representable} if there is
a measure $\Delta$ sampled from~$\Rmu$
(Definitions~\ref{d:sampled-from} and~\ref{d:Rmu}) with
$$
  R(X) = \<\Delta,X\>
  \text{ for all }
  X \in \Emu.
$$
A limit $R_n \to R$ in $\cC(\Tmu, \RR)$ is \emph{representable} if
all~$R_n$ are representable.
\end{defn}

\begin{remark}\label{r:representable}
The term ``representable'' in Definition~\ref{d:representable} is
meant to evoke the Riesz representation theorem, wherein functionals
are represented as inner products.  See also
Section~\ref{b:represent}, which takes this idea further to represent
Gaussian tangent fields.
\end{remark}

\begin{example}\label{e:representable}
The average empirical tangent field~$\bgn$ in
Definition~\ref{d:empirical-tangent-field} is representable by
Remark~\ref{r:nablamuF(V)=m(mu,V)}, which explains why $m(\mu,X) = 0$
for $X \in \Emu$.  The Gaussian tangent field $G$ from
Definition~\ref{d:gaussian-tangent-field} in principle might not be
representable, but $G$ is almost surely a limit of representable
functions by Corollary~\ref{c:cont-realization}, which is all the
extended continuous mapping theorem \cite[Theorem~1.11.1]{VW13} needs
for application in Section~\ref{s:CLT}.  Note that even if $G$ is not
representable on~$\Emu$ as required by
Definition~\ref{d:representable}, it is at least representable on the
closed fluctuating cone~$\oCmu$; this major result~is
\mbox{Theorem}~\ref{t:Riesz-representation}.
\end{example}

Some preliminary results concerning restrictions of representable
functions to the fluctuating cone~$\Cmu$ are needed later,
specifically for Corollary~\ref{c:equiv-limits-of-td}.  These items
use the isometry of~$\Cmu$ with a convex cone in~$\RR^m$ from
tangential collapse (Theorem~\ref{t:collapse}).

\begin{defn}\label{d:LL(Delta)}
Fix a tangential collapse $\LL: \Tmu \to \RR^m$ of the measure~$\mu$
on a smoothly stratified metric space~$\MM$.  For any measure $\Delta
= \lambda_1\delta_{W^1} + \dots + \lambda_j\delta_{W^j}$ sampled
from~$\Tmu$,~set
$$
  \LL(\Delta)
  =
  \lambda_1\LL(W^1) + \dots + \lambda_j\LL(W^j) \in \RR^m.
$$
\end{defn}

\begin{remark}\label{r:LL(Delta)}
The image of the measure~$\Delta$ under~$\LL$ in
Definition~\ref{d:LL(Delta)} is a vector in~$\RR^m$, which is to be
distinguished from the pushforward $\LL_\sharp\Delta$, which is a
measure sampled from~$\RR^m$.  This usage of~$\LL(\Delta)$ differs
from the identification of $\log_\bmu\delta$ with
$(\log_\bmu)_\sharp\delta$ in
Definition~\ref{d:pair-with-discrete-measure}, but confusion is
minimized because there is no possibility to produce an element
of~$\Tmu$ from $\log_\bmu\delta$ by adding the points of~$\Tmu$.
\end{remark}

\begin{lemma}\label{l:reduced-escape}
If $\LL: \Tmu \to \RR^m$ is a tangential collapse and $V \!\in
\RR^\ell = \hull\muL \subseteq\nolinebreak \RR^m$ as in
Lemma~\ref{l:supp-NLmu}, then $V = \LL(\Delta)$ for some
measure~$\Delta \Subset \supp\hmu$ with support of size at most~$m$.
That is, the map
$$
  \LL: \Disc_m(\supp\hmu) \onto \RR^\ell
$$
induced by~$\LL$ on the space $\Disc_m(\supp\hmu)$ of measures sampled
from $\supp\hmu$ with support of size at most~$m$ is surjective.
\end{lemma}
\begin{proof}
$V$ is a positive linear combination $\lambda_1 V^1 + \dots +
\lambda_j V^j$ of vectors $V^i = \LL(X^i)$ with $X^i$ in the support
of~$\hmu$ by Lemma~\ref{l:supp-NLmu}.  The number of summands can be
chosen to be $j = m$ by Carath\'eodory's theorem
\cite[Proposition~1.15]{ziegler1995}.
\end{proof}

\begin{remark}\label{r:size-at-most-m}
Despite the proof showing that the number of summands can be chosen
equal to~$m$, the set $\Disc_m(\supp\hmu)$ specifies `support size at
most~$m$'' because some of those $m$ summands might be supported at
the same point.
\end{remark}

\begin{remark}\label{r:gaussian-section}
Lemma~\ref{l:reduced-escape} is one of the key reasons to introduce
hulls in Definition~\ref{d:hull}: it allows construction of discrete
measures as preimages under tangential collapse~$\LL$ whose escape
vectors subsequently lie in the fluctuating cone~$\Cmu$ by
Corollary~\ref{c:confinement-to-Cmu}.  In the presence of a tangential
collapse~$\LL: \Tmu \onto \RR^m$ that is surjective, or even merely
$\LL: \Tmu \onto \RR^\ell \subseteq \RR^m$, it would be possible to
develop the theory in this subsection---and, indeed, in the rest of
Sections~\ref{s:escape} and~\ref{s:CLT}---with individual tangent
vectors $V_n \to V$ instead of $\Delta_n \to \Delta$, as nontrivial
sums in Lemma~\ref{l:reduced-escape} would~not~be~needed.
\end{remark}

The next lemma shows that representable limits can be represented by
convergent sequences of measures and that the limit, while perhaps not
representable on all of~$\Emu$, is at least representable on the
closed fluctuating cone~$\oCmu$; cf.~Remark~\ref{r:Hn=Upsilon(bgn)}.

\begin{lemma}\label{l:represent}
Fix a localized measure~$\mu$ on a smoothly stratified space~$\MM$.
If~$R_n \to\nolinebreak R$ is a representable limit, then some
convergent sequence $\Delta_n \to \Delta$ of measures sampled from
$\Rmu$ satisfies $R_n(X) = \<\Delta_n, X\>$ and $R(X) = \<\Delta, X\>$
for~all~$X \in \oCmu$.
\end{lemma}
\begin{proof}
Continuity of inner products in
Lemma~\ref{l:inner-product-is-continuous} and of the limit
function~$R$ imply that it suffices to prove the claim for~$\Cmu$
instead of~$\oCmu$.

By Definition~\ref{d:representable}, the restriction of~$R_n$
to~$\Cmu$ is an inner product $\<\pDelta_n,\,\cdot\,\>$ for some
$\pDelta_n$ sampled from~$\Tmu$.  Choose a tangential collapse $\LL:
\Tmu \to \RR^m$ by Theorem~\ref{t:collapse}, and identify $\Cmu$ with
$\LL(\Cmu) \subseteq \RR^m$ by
Definition~\ref{d:collapse}.\ref{i:injective}.  For $X \in \Cmu$, the
inner product $\<\pDelta_n, X\>$ is equal to $\<\LL(\pDelta_n), X\>$
by Definition~\ref{d:collapse}.\ref{i:partial-isometry}, and hence
$\<\pDelta_n, X\> = \<W_n, X\>$ for a unique vector $W_n$ in the
linear span of~$\Cmu$.  Since $R_n \to R$, it follows that $W_n \to W$
for a unique vector $W$ in the span of~$\Cmu$ and that $\<\pDelta, X\>
= \<W, X\>$ for $X \in \Cmu$.

Use Lemma~\ref{l:reduced-escape} to pick $Y^i_n \!\in\! \supp\hmu$
such that $\ppDelta_n \!= \sum_{i=1}^m \lambda_i^n Y^i_n$ maps
to~$\LL(\ppDelta_n) =\nolinebreak W_n$.  There is no harm in assuming
$\ppDelta_n$ has exactly $m$ summands because the nonzero mass on one
of the vectors can be split if necessary.  Let $\theta^i_n \in \Smu$
be the direction of~$Y^i_n$.  The sequence $\{\theta^i_n\}_n$ of
directions for fixed~$i$ has a convergent subsequence because $\Smu$
is compact.  Taking subsequences if necessary, assume
$\{\theta^i_n\}_n$ converges for all $i = 1,\dots,m$, so $\theta^i_n
\to \theta^i$.  Let $U^i_n = \LL(\theta^i_n)$ and $U^i =
\LL(\theta^i)$, so $U^i_n \to U^i$ in~$\RR^m$~for~all~$i$.

Some size~$\ell$ subset of the vectors $U^1,\dots,U^m$ form a basis
of~$\RR^\ell$ by Lemma~\ref{l:supp-NLmu}.  Renumbering if necessary,
assume they are $U^1,\dots,U^\ell$.  Then $U^1_n,\dots,U^\ell_n$ form
a basis of~$\RR^\ell$ for all~$i$ and all $n \gg 0$ because linear
independence is an open condition.  Taking subsequences if necessary,
assume $U^1_n,\dots,U^\ell_n$ form a basis of~$\RR^\ell$ for all~$i$
and all~$n$.  The coefficients $\lambda_1^n,\dots,\lambda_\ell^n$
of~$W_n$ in the basis $U^1_n,\dots,U^\ell_n$ converge to the
coefficients $\lambda_1,\dots,\lambda_\ell$ of~$W$ in the basis
$U^1,\dots,U^\ell$ because $W_n \to W$ and $U^i_n \to U^i$ for
all~$i$.  Take $\Delta_n = \sum_{i=1}^\ell \lambda_i^n \theta^i_n$,
which by construction
\begin{itemize}
\item%
is sampled from~$\Rmu$,
\item%
converges to $\Delta = \sum_{i=1}^\ell \lambda_i \theta^i$, and
\item%
satisfies $W_n = \LL(\Delta_n) \to \LL(\Delta) = W$,
\end{itemize}
so $R_n(X) = \<\pDelta_n, X\> = \<W_n, X\> = \<\Delta_n,X\>$ for $X
\in \Cmu$, and similarly for~$R(X)$.
\end{proof}


\subsection{Escape for continuous functions}\label{b:continuous-escape}
\mbox{}\medskip

\noindent
Next comes the analogue for continuous functions~$R$ of escape vectors
for measures~$\Delta$.

\begin{defn}\label{d:Upsilon}
Fix a measure~$\mu$ on a smoothly stratified metric space~$\MM$.  Make
any choice of $\argmin$s over exponentiable sets
(Definition~\ref{d:exponentiable}) as follows.
\begin{align*}
\Ups: \cC(\Tmu, \RR)
            &\to \Tmu
\\*
    \Ups(R) &\in \argmin_{X \in \Emue} \bigl(F(\exp_\bmu X) - R(X)\bigr).
\\\text{and}\quad
\Upt: \cC(\Tmu,\RR)
            & \to \Tmu
\\*
    \Upt(R) &\in \argmin_{X \in \Tmue} \bigl(F(\exp_\bmu X) - R(X)\bigr),
\end{align*}
\end{defn}

\begin{remark}\label{r:Hn=Upsilon(bgn)}
The empirical tangent perturbation~$H_n$ from Definition~\ref{d:H}
results when~$\Upt$ is evaluated on the average empirical tangent
field~$\bgn$ from Definition~\ref{d:empirical-tangent-field}, so
$$
  H_n = \Upt(\bgn)
$$
define the same $\argmin$ expression.  Similarly, the Gaussian tangent
perturbation~$H(t)$ from Definition~\ref{d:H} results when $\Upt$ is
evaluated on the Gaussian tangent field $G$ from
Theorem~\ref{t:tangent-field-CLT}:
$$
  H(t) = \Upt(tG).
$$
\end{remark}

\begin{remark}\label{r:Upsilon}
With $\Upt$ and~$\Ups$ in Definition~\ref{d:Upsilon} and a
representable limit $R_n \to R$, the goal is to show in
Theorem~\ref{t:convergence-of-Upsilon} that the limit does not change
when the $\argmin$ is confined to the escape cone:
\begin{equation}\label{eq:converges-of_Upsilon}
  \lim_{n\to\infty}
  \dd\Bigl(\frac{\Upt(t_n R_n)}{t_n}, \frac{\Ups(t_n R_n)}{t_n}\Bigr)
  = 0.
\end{equation}
\end{remark}

\begin{remark}\label{r:bounded-Upsilon}
By Proposition~\ref{p:transform}, $\argmin_{x\in\MM}\bigl(F(x) -
\bgn(\log_\bmu x)\bigr)$ is within $o(1/\sqrt n)$ of~$\bmu_n$.  Thus
it is safe to assume $\argmin_{X\in\Tmue}\bigl(F(\exp_\bmu X) -
t_nR_n(X)\bigr)$ is within $o(t_n)$ of~$\bmu_n$.  The aim is to
replace the function~$F$ by its order~$2$ Taylor expansion at~$\bmu$
for the purpose of taking $\argmin_{x\in\MM} \bigl(F(x) -
\bgn(\log_\bmu x)\bigr)$.  Applying this idea first needs that
$\Upt(t_nR_n)/t_n$ is bounded; cf.~Lemma~\ref{l:magnitude-of-a_n}.
\end{remark}

\begin{lemma}\label{l:bound-of_Upsilon}
Suppose that $R_n \to R$ is representable and $t_n \to 0$ for a
sequence of positive real numbers~$t_n$.  Then there is a finite
constant~$C(R)$, depending only on the limit~$R = \lim_n R_n$, such
that for some $N_0 < \infty$,
$$
  \Bigl\|\frac{1}{t_n} \Upt(t_n R_n)\Bigr\|
  \leq
  C(R)
  \text{ for all } n \geq N_0.
$$
\end{lemma}
\begin{proof}
The proof of Lemma~\ref{l:tight-h_n} works when $\bgn$ is replaced by
$t_n R_n$ and $1/\sqrt n$ by~$t_n$.  Specifically, altering
Eq.~\eqref{eq:lem:tight-h_n:3} this way yields
$$
  \frac{\sup_{\theta \in \Smu} R_n(\theta)}{C}
  \geq
  \frac{1}{t_n}\,\dd(\bmu,\Upt(t_n R_n)),
$$
whose left-hand side is bounded above by $C(R) = \frac 1C \bigl(1 +
\sup_{\theta\in\Smu} R(\theta)\bigr)$ for large~$n$.\!\!
\end{proof}

\begin{defn}\label{d:more-sequences}
Recall from Definition~\ref{d:localized} that $\hmu =
(\log_\bmu)_\sharp\mu$, and let $\tF = \Fmu \circ\nolinebreak
\exp_\bmu$.  Also recall the (directional) Taylor expansion of the
Fr\'echet function at~$\bmu$ from Proposition~\ref{p:Lambda}.  For $R
\in \cC(\Tmu, \RR)$ define
$$
  \eta(R,x)
  =
  r_x \nablabmu F(\theta_x) + r_x^2\Lambda_\bmu(\theta_x) - r_x R(\theta_x),
$$
in polar coordinates from Definition~\ref{d:polar}.  For
representable~$R_n \to R$, arbitrarily select
\begin{align*}
  \ud_n & \in \argmin \limits_{x \in \exp_\bmu(\Emue)}\eta(t_n R_n, x)
\\\makebox[0pt][r]{and\quad}
  \td_n & \in\ \argmin_{x \in \MM}\ \eta(t_n R_n, x).
\end{align*}
\end{defn}

\begin{remark}\label{r:analogues-of-c_n}
The sequence $\ud_n$ is a sort of functional version of the
sequence~$c_n$ from Notation~\ref{n:deterministic}.  The
sequence~$\td_n$ is the unconfined version of~$\ud_n$, with the argmin
taken over~$\MM$ instead of the (exponential image of the) escape
cone.
\end{remark}

The first order of business is to show in
Proposition~\ref{p:convergence-of-d} that $\frac 1{\;t_n} \log_\bmu
\td_n$ and $\frac 1{\;t_n} \log_\bmu \ud_n$ approach the same limit.
The limit itself is identified in Corollary~\ref{c:equiv-limits-of-td}
as a consequence of the main results of Section~\ref{b:convergence}.
The statement and proof of Proposition~\ref{p:convergence-of-d} then
occurs after Lemmas~\ref{l:theta_an-to-Cmu}, \ref{l:formula-wd_n},
and~\ref{l:convergence-of-rd_n}, which show respectively that the
direction of~$\td_n$ approaches the unit sphere $\SEmu$ in the escape
cone~$\Emu$ as $n$ goes to infinity, that polar coordinates of~$\ud_n$
and~$\td_n$ satisfy formulas derived without knowledge of their
uniqueness or the relevant convergences, and that the radii of~$\td_n$
converge appropriately.

\pagebreak[2]

\begin{cor}\label{c:equiv-limits-of-td}
In the situation of Lemma~\ref{l:represent}, $\ud_n =
t_n\ev(\Delta_n)$~and
$$
  \lim_{n\to\infty} \frac{\log_\bmu \ud_n}{t_n}
  =
  \ev(\Delta).
$$
\end{cor}
\begin{proof}
If $R_n$ is represented on~$\Emu$ by any measure~$\Delta_n \Subset
\Rmu$, whether this measure~$\Delta_n$ resides in a convergent
sequence or not, the $\argmin$ defining $d_n$~equals the expression in
Theorem~\ref{t:escape-vector-confinement}(\ref{c}) because $\nablabmu
F(\theta_x) = 0$ for $\theta_x \in \Emu$ by
Definition~\ref{d:escape-cone}.  Therefore $d_n = t_n\ev(\Delta_n)$ by
Corollary~\ref{c:escape-continuity}.  The simple fact that $\Delta_n
\to \Delta$ is convergent, regardless of any agreement with~$R_n$
or~$R$, yields $\lim_{n\to\infty} \frac 1{t_n} \log_\bmu \ud_n =
\ev(\Delta)$ by Theorem~\ref{t:escape-vector-confinement}(\ref{c}).
\end{proof}

\begin{lemma}\label{l:theta_an-to-Cmu}
For any choice of $\td_n$ in Definition~\ref{d:more-sequences},
$$
  \lim_{n\to\infty} \dd_s(\theta_{\td_n}, \SEmu) = 0.
$$
\end{lemma}
\begin{proof}
Suppose conversely that
\begin{equation}\label{eq:lem:theta_an_to_Cmu:1}
  \lim_{n\to\infty} \dd_s(\theta_{\td_n}, \SEmu) \geq \ve > 0.
\end{equation}
In particular, assume $r(\td_n) > 0$ for all~$n$.  Since $\nablabmu F:
\Tmu \to \RR$ is continuous by
Lemma~\ref{l:inner-product-is-continuous} and
\cite[Corollary~2.7]{tangential-collapse}
(which expresses $\nablabmu F(V)$ as an integral
of inner products), so is its restriction $\nablabmu F|_{\Smu}$ to the
unit tangent sphere.  As $\bmu$ is the minimizer of $F$, its
derivative $\nablabmu F$ is nonnegative.  By
Definition~\ref{d:escape-cone} (of~$\Emu$), $\nablabmu F|_{\Smu}$
vanishes on~$\SEmu$ and is positive everywhere else.
Thus there exists $\alpha > 0$ such that
$$
  \nablabmu F|_{\Smu}(\theta) > \alpha
  \text{ for all } \theta \in \Smu
  \text{ with } \dd_s(\theta,\SEmu) > \ve.
$$
Choose an integer~$N_1$ such that $t_i < \alpha/\sup_{\theta \in
S_\mu\MM} R_n(\theta)$ for all $i,n \geq N_1$.  Then for any $x \neq
\bmu$ and $\dd_s(\theta_x, \Emu) \geq \ve$, the function~$\eta$ from
Definition~\ref{d:more-sequences} evaluates to
\begin{align*}
  \eta(t_n R_n, x)
   & = r_x \nablabmu F(\theta_x) + r_x^2\Lambda_\bmu (\theta_x)
       - t_n r_x R_n(\theta_x)
\\*& > r_x \alpha + r_x^2\Lambda_\bmu (\theta_x) - \alpha r_x
     > 0
\end{align*}
by Lemma~\ref{l:uniform-convexity}.  Thus $\eta(t_n R_n, x) >0$ for
all $n \geq N_1$ if $x \neq \bmu$ satisfies $\dd_s(\theta_x, \Emu)
\geq \ve$.  In particular, $\eta(t_n R_n, \td_n) >0$ due
to~\eqref{eq:lem:theta_an_to_Cmu:1}.  This is a contradiction because
$$
  \eta(t_n R_n, \td_n)
  =
  \min_{x \in \MM} \eta(t_n R_n, x)
  \leq
  \min_{x \in \exp_\bmu\Emue} \eta(t_n R_n, x)
  =
  \eta(t_n R_n, \ud_n)
$$
whereas if $\theta = \theta_{\ud_n}$ and $r = r_{\ud_n}$ then
$$
  \eta(t_n R_n, \ud_n)
  =
  r^2 \Lambda_\bmu(\theta) - r\bigl(2r\Lambda_\bmu(\theta)\bigr)
  =
  - r^2 \Lambda_\bmu(\theta)
  \leq 0
$$
since $2r \Lambda_\bmu(\theta) = t_n R_n(\theta)^+$ by
Lemma~\ref{l:formula-wd_n} and $\Lambda_\bmu(\theta_{\ud_n}) > 0$ by
Lemma~\ref{l:uniform-convexity}.
\end{proof}

\begin{lemma}\label{l:formula-wd_n}
The polar coordinates of~$\ud_n$ satisfy
\begin{align*}
  \theta_{\ud_n}
  \in
  \argmax_{\theta\in \SEmu} \frac{R_n(\theta)}{\sqrt{\Lambda_\bmu(\theta)}}
  &
  \quad\text{and}\quad
  r_{\ud_n}
  =
  \frac{t_n R_n(\theta_{d_n})^+}{2\Lambda_\bmu(\theta_{d_n})},
\intertext{while the polar coordinates of~$\td_n$ satisfy}
  \theta_{\td_n}
  \in
  \argmax_{\theta\in \Smu} \frac{t_n R_n(\theta) - \nablabmu F(\theta)}
                                {\sqrt{\Lambda_\bmu(\theta)}}
  &
  \quad\text{and}\quad
  r_{\td_n}
  =
  \frac{\bigl(t_n R_n(\theta_{\td_n}) - \nablabmu F(\theta_{\td_n})\bigr)^+}
       {2\Lambda_\bmu(\theta_{\td_n})},
\end{align*}
where $\alpha^+ = \max\{\alpha,0\}$ as in
Proposition~\ref{p:unique-c-and-limits}.
\end{lemma}
\begin{proof}
For $d_n$ this is Corollary~\ref{c:equiv-limits-of-td} and
Theorem~\ref{t:escape-vector-confinement}.  For $\td_n$, the proof of
Lemma~\ref{l:formula-c_n} works after replacing all occurrences of
$\<\Delta_n,\theta_x\>$ with $t_n R_n(\theta_x) - \nablabmu
F(\theta_x)$.
\end{proof}

\begin{lemma}\label{l:convergence-of-rd_n}
In terms of the escape vector $\ev = \ev(\Delta)$ in
Corollary~\ref{c:equiv-limits-of-td}, the radial polar coordinate
of~$\td_n$ satisfies
$$
  \lim_{n\to\infty} \frac{r_{\td_n}}{t_n} = r_\ev.
$$
\end{lemma}
\begin{proof}
Let $A = t_n R_n(\theta_x) - \nablabmu F(\theta_x)$, so 
\begin{align*}
\eta(t_n R_n,x)
  &= r_x \nablabmu F(\theta_x) + r_x^2\Lambda_\bmu(\theta_x) - r_x t_n R_n(\theta_x)
\\
  &= \Lambda_\bmu(\theta_x)\biggl(r_x - \frac A{2\Lambda_\bmu(\theta_x)}\biggr)^2
     - \frac{A^2}{4\Lambda_\bmu(\theta_x)}.
\end{align*}
As in the proof of Lemma~\ref{l:formula-wd_n} (see the proof of
Lemma~\ref{l:formula-c_n}), this expression is minimized either when
$r_x = 0$ or the parenthesized expression vanishes.  In either case,
using the minimizer radius formulas in Lemma~\ref{l:formula-wd_n},
\begin{align}\label{eq:eta}
  \eta(t_n R_n,x) = -\Lambda_\bmu(\theta_x) r_x^2
  \text{ when } x \text{ is a minimizer for either } \ud_n \text{ or } \td_n.
\end{align}
Therefore
$$
  -\Lambda_\bmu(\theta_{\td_n}) r^2_{\td_n}
  =
  \eta(t_n R_n, \td_n)
  \leq
  \eta(t_n R_n, \ud_n)
  =
  -\Lambda_\bmu(\theta_{\ud_n}) r^2_{\ud_n},
$$
with the inequality because $\td_n$ is minimized over a larger set.
The inequality implies
$$
  \frac{r_{\td_n} \sqrt{\Lambda_\bmu(\ttdn)}}{t_n}
  \geq
  \frac{r_{\ud_n} \sqrt{\Lambda_\bmu(\tudn)}}{t_n}.
$$
Substituting the formulas for $r_{\td_n}$ and~$r_{\ud_n}$ from
Lemma~\ref{l:formula-wd_n} into this inequality yields
$$
  \frac{R_n(\theta_{\td_n})}{2\sqrt{\Lambda_\bmu(\ttdn)}}
    - \frac{\nablabmu F(\theta_{\td_n})}{t_n\sqrt{\Lambda_\bmu(\ttdn)}}
  \geq
  \frac{R_n(\theta_{\ud_n})}{2\sqrt{\Lambda_\bmu(\tudn)}}\,.
$$
Taking the limit as $n \to \infty$ and using
Corollary~\ref{c:equiv-limits-of-td} yields
\begin{equation}\label{eq:proof:convergence-of-rd_n:1}
\lim_{n\to\infty}\biggl(\frac{R_n(\ttdn)}{2\sqrt{\Lambda_\bmu(\ttdn)}}
                      -\frac{\nablabmu F(\ttdn)}{t_n\sqrt{\Lambda_\bmu(\ttdn)}}\biggr)
  \geq
  \frac{R(\theta_\ev)}{2\sqrt{\Lambda_\bmu(\theta_\ev)}}\,.
\end{equation}

On the other hand, $\nablabmu F\bigl(\theta(x)\bigr) \geq 0$ for all
$x \in \MM$, since $\bmu$ is the minimizer of~$F$ and
$\Lambda_\bmu(\theta_x) > 0$ by Lemma~\ref{l:uniform-convexity}.
Combining this with Lemma~\ref{l:theta_an-to-Cmu} gives
\begin{equation}\label{eq:proof:convergence-of-rd_n:2}
\begin{split}
\lim_{n\to\infty}\biggl(\frac{R_n(\ttdn)}{2\sqrt{\Lambda_\bmu(\ttdn)}}
                      -\frac{\nablabmu F(\ttdn)}{t_n\sqrt{\Lambda_\bmu(\ttdn)}}\biggr)
   & \leq \max_{\theta \in \SEmu}
          \frac{R(\theta)}{2\sqrt{\Lambda_\bmu(\theta)}}
\\
   &   =  \frac{R(\theta_\ev)}{2\sqrt{\Lambda_\bmu(\theta_\ev)}}
\end{split}
\end{equation}
by Corollary~\ref{c:equiv-limits-of-td} (where $\ev = \ev(\Delta)$ is
defined) and the directional polar coordinate in
Theorem~\ref{t:escape-vector-confinement}.  Combining
\eqref{eq:proof:convergence-of-rd_n:1} and
\eqref{eq:proof:convergence-of-rd_n:2} produces
$$
\lim_{n\to\infty}\biggl(\frac{R_n(\ttdn)}{2\sqrt{\Lambda_\bmu(\ttdn)}}
                      -\frac{\nablabmu F(\ttdn)}{t_n\sqrt{\Lambda_\bmu(\ttdn)}}\biggr)
  =
  \frac{R(\theta_\ev)}{2\sqrt{\Lambda_\bmu(\theta_\ev)}}\,.
$$
Substituting back the formula for~$r_{\td_n}$ in
Lemma~\ref{l:formula-wd_n} and the radial polar coordinate~$r_\ev$ in
Corollary~\ref{c:equiv-limits-of-td} gives the desired result:
$\lim_{n\to\infty}\frac 1{t_n}r_{\ud_n} = r_\ev$.
\end{proof}

\begin{prop}\label{p:convergence-of-d}
The limit in Corollary~\ref{c:equiv-limits-of-td} remains valid
with~$\td_n$ in place of~$\ud_n$:
$$
  \lim_{n\to\infty} \frac{\log_\bmu \td_n}{t_n}
  =
  \lim_{n\to\infty} \frac{\log_\bmu \ud_n}{t_n}.
$$
\end{prop}
\begin{proof}
Restricting to any convergent subsequence of directions yields
$\lim_{t\to 0}\theta_{\td_n} = \xi$ for some $\xi \in \Emu$ by
Lemma~\ref{l:theta_an-to-Cmu}.  The $\argmax$ defining~$\xi$ is taken
over the escape cone~$\Emu$ and is hence unique by
Corollary~\ref{c:equiv-limits-of-td}.  Since every convergent
subsequence of directions converges to this same limit in the unit
tangent sphere~$\Smu$, the entire sequence of directions~$\ttdn$
converges.  Since $\nablabmu F$ vanishes on~$\Emu$ by
Definition~\ref{d:escape-cone}, $\nablabmu F(\xi) = 0$, which means
that the $\argmax$ characterizing~$\lim_{n\to\infty} \ttdn$ derived
from Lemma~\ref{l:formula-wd_n} is equal to the similarly derived
formula for~$\lim_{n\to\infty} \tudn$.  On the other hand, the
limiting radii are also equal by Lemma~\ref{l:convergence-of-rd_n}.
\end{proof}

Next comes the final ingredient in the proof
of~\eqref{eq:converges-of_Upsilon}.

\pagebreak[2]

\begin{prop}\label{p:Upsilon_and_a}
Suppose that $R_n \to R$ is representable and $t_n \to 0$ for a
sequence of positive real numbers~$t_n$.  Then for $n$ sufficiently
large,
$$
  \Ups(t_n R_n) \in \log_\bmu \argmin_{x \in \Emu} \eta(t_n R_n, x)
  \quad\text{and}\quad
  \Upt(t_n R_n) \in \log_\bmu \argmin_{x \in \MM} \eta(t_n R_n, x).
$$
\end{prop}
\begin{proof}
The proofs for~$\Ups$ and~$\Upt$ are the same, but simpler for~$\Ups$
because every $\nablabmu F$ term vanishes because the argument lies
in~$\Emu$.  As a guide to how the $\nablabmu F$ terms are carried
through, the proof for~$\Upt$ is written in detail.

The hypotheses are those of Lemma~\ref{l:bound-of_Upsilon}, so
abbreviate $\Upt(t_n R_n) = e_n$ and set $C = C(R)$ from that lemma.
Thus the radial polar coordinate of~$e_n$ satisfies $r_{e_n} \leq C
t_n$.  Hence the sequence $\{e_n/t_n\}_n$ is contained in the compact
ball of radius~$C$ in~$\Tmu$ centered at the cone point.  The Taylor
expansion in Proposition~\ref{p:Lambda} yields
$$
  F(e_n) - t_n R_n (\log_\bmu e_n)
  =
  F(\bmu) + r_{e_n} \nablabmu F(\theta_{e_n})
          + r_{e_n}^2\Lambda_\bmu (\theta_{e_n})
          - t_nR_n(\exp_\bmu e_n)
          + o(r_{e_n}^2).
$$
Thus $r_{e_n}$ is the radius that minimizes
\begin{equation}\label{eq:e_n}
  r \nablabmu F(\theta_{e_n}) + r^2\Lambda_\bmu(\theta_{e_n}) - r t_n R_n(\theta_{e_n})
  =
  \eta\bigl(t_n R_n, r\theta_{e_n}\bigr),
\end{equation}
meaning that
\begin{equation}\label{eq:lem:Upsilon_and_a:0}
r_{e_n}
  =
  \frac{\bigl(t_n R_n(\theta_{e_n}) - \nablabmu F(\theta_{e_n})\bigr)^+}
       {2\Lambda_\bmu (\theta_{e_n})},
\end{equation}
where again $\alpha^+ = \max\{\alpha,0\}$ as in
Lemma~\ref{l:formula-wd_n}.

If $r_{e_n} = 0$ then necessarily $r_{\td_n} = 0$ for any $\td_n\in
\argmin_{x \in \MM} \eta(t_n R_n, x)$.  To see why, note that $r_{e_n}
= 0$ implies $e_n = \bmu$,~so
\begin{equation}\label{eq:lem:Upsilon_and_a:3}
\begin{split}
F(\bmu)
& =
  F(\bmu) - t_n R_n(\log_\bmu \bmu)
\\
& \leq
  F(\td_n) - t_n R_n(\log_\bmu \td_n)
\\
& =
  F(\bmu) + \eta(t_n R_n, \td_n) + o(r_{\td_n}^2)
\\
& =
  F(\bmu) - r_{\td_n}^2\Lambda_\bmu(\theta_{\td_n}) + o(r_{\td_n}^2),
\end{split}
\end{equation}
where the last equality is by~\eqref{eq:eta}.  If $r_{\td_n} > 0$ then
Proposition~\ref{p:convergence-of-d} says that $\theta_{\td_n} \to
\theta_\ev$ for some~$\ev$ of strictly positive length as $n \to
\infty$.  Thus, for large enough~$n$,
$$
  F(\bmu) - r_{\td_n}^2\Lambda_\bmu(\theta_{\td_n}) + o(r_{\td_n}^2)
  <
  F(\bmu),
$$
which contradicts~\eqref{eq:lem:Upsilon_and_a:3}.  Therefore $r_{e_n}
= 0$ implies $r_{\td_n} = 0$.

It remains to consider the case where $r_{e_n} > 0$
in~\eqref{eq:lem:Upsilon_and_a:0}.  As in~\eqref{eq:eta}, completing
the square in \eqref{eq:e_n} gives
$$
  \eta(t_n R_n, e_n)
  =
  -\Lambda_\bmu (\theta_{e_n}) r_{e_n}^2
  =
  -\frac{\bigl(t_n R_n(\theta_{e_n}) - \nablabmu F(\theta_{e_n})\bigr)^2}
        {4\Lambda_\bmu (\theta_{e_n})},
\vspace{-.5ex}
$$
and
\begin{align*}
F(e_n)\!-\!t_n R_n(\log_\bmu e_n)
   &=F(\bmu) + \eta(t_n R_n, e_n) + o(r_{e_n}^2)
\\*&=F(\bmu)
     \!-\!
     \frac{\bigl(t_nR_n(\theta_{e_n})\!-\!\nablabmu F(\theta_{e_n})\bigr)^2\!}
          {4\Lambda_\bmu (\theta_{e_n})}
     \!+\!
     o\biggl(\!\frac{\bigl(t_nR_n(\theta_{e_n})\!-\!\nablabmu F(\theta_{e_n})\bigr)^2}
                    {4\Lambda_\bmu ^2(\theta_{e_n})}\biggr)
\end{align*}
follows.  Because $e_n$ minimizes $F(x) - t_n R_n(\log_\bmu x)$, the
direction $\theta_{e_n}$ solves
$$
  \argmax_{\theta \in \Smu} \frac{t_n R_n(\theta) - \nablabmu
  F(\theta)}{\Lambda_\bmu (\theta)}.
$$
Comparing this and \eqref{eq:lem:Upsilon_and_a:0} with the formula
for~$\td_n$ in Lemma~\ref{l:formula-wd_n} leads to the conclusion that
for $n$ sufficiently large $\log_\bmu e_n \in \log_\bmu \argmin_{x \in
M} \eta(t_n R_n, x)$.
\end{proof}

\subsection{Convergence of continuous escape}\label{b:upsilon-convergence}
\mbox{}\medskip

\noindent
Here, finally, is the verification of~\eqref{eq:converges-of_Upsilon},
enhanced by a number of conclusions reached along the way.

\begin{thm}\label{t:convergence-of-Upsilon}
Fix a localized amenable measure~$\mu$ on a smoothly stratified metric
space and a representable limit $R_n \to R$ as in
Definition~\ref{d:representable}.  There is a measure $\Delta \Subset
\Rmu$ with $R(X) = \<\Delta,X\>$ for all $X \in \Cmu$, and for any
such~$\Delta$, the sequences $\{\ud_n\}$ and~$\{\td_n\}$ from
Definition~\ref{d:more-sequences} relate to the maps $\Ups$
and\/~$\Upt$ in Definition~\ref{d:Upsilon}~as~follows:
$$
  \lim_{n\to\infty} \frac{\Upt(t_n R_n)}{t_n}
  =
  \lim_{n\to\infty} \frac{\log_\bmu \td_n}{t_n}
  =
  \lim_{n\to\infty} \frac{\log_\bmu \ud_n}{t_n}
  =
  \lim_{n\to\infty} \frac{\Ups(t_n R_n)}{t_n}
  =
  \ev(\Delta).
$$
\end{thm}
\begin{proof}
A measure~$\Delta$ representing~$R$ on~$\Cmu$ exists by
Lemma~\ref{l:represent}.  Proposition~\ref{p:Upsilon_and_a} and
Definition~\ref{d:more-sequences} prove the first and third equalities
(between the $\Ups$'s and $d$'s).  The second equality (between the
two types of~$d$'s) is Proposition~\ref{p:convergence-of-d}.
Corollary~\ref{c:equiv-limits-of-td} proves equality
with~$\ev(\Delta)$ for a particular choice of~$\Delta$.  However,
$\Delta$ can be replaced by any measure~$\Delta' \Subset \Rmu$ with
$R(X) = \<\Delta',X\>$ for all $X \in \Cmu$ because $\ev(\Delta')$
only depends on the values of~$\<\Delta',X\>$ for $X \in \oCmu$ by
confinement in Corollary~\ref{c:confinement-to-Cmu}.
\end{proof}

\begin{cor}\label{c:convergence-of-Upsilon}
If the measure~$\mu$ in Theorem~\ref{t:convergence-of-Upsilon} is
immured, then the limits are confined to the closed fluctuating
cone~$\oCmu$, in the sense that the $\argmin$ in each of~$d_n$
and~$\Ups$ can be taken over~$\oCmu$ instead of~$\Emu$, and moreover
$$
  \lim_{n\to\infty} \frac 1{t_n}\Upt(t_n R_n)
  =
  \lim_{t \to 0} \frac 1t \argmin_{X \in \oCmue} \bigl(F(\exp_\bmu X) - t R(X)\bigr).
$$
\end{cor}
\begin{proof}
The immured hypothesis allows Corollary~\ref{c:confinement-to-Cmu} to
confine escape vectors of all measures sampled from~$\Rmu$ to~$\oCmu$.
Therefore the limit on the right-hand side here equals~$\ev(\Delta)$
by Theorem~\ref{t:escape-vector-confinement}(\ref{ev}) and~(\ref{b})
because in Theorem~\ref{t:convergence-of-Upsilon}, $\Delta \Subset
\Rmu$ and $R(X) = \<\Delta,X\>$ for all $X \in \Cmu$, and hence---by
continuity---for all $X \in \oCmu$.
\end{proof}

\section{Geometric central limit theorems}\label{s:CLT}

The central limit theorem that has been the primary target all along
is Theorem~\ref{t:clt}, which characterizes the distribution of
rescaled Fr\'echet means on a smoothly stratified metric space.  Its
statement and the definition of distortion map in
Definition~\ref{d:distortion-map} are isolated in Section~\ref{b:clt}.
The subsequent two subsections~\ref{b:barycenter}
and~\ref{b:minimizer} detail variational interpretations of the CLT as
directional derivatives in spaces of measures or functions.  In
preparation for all of these results, it is necessary to identify the
Gaussian objects on stratified spaces in Section~\ref{b:represent} at
the core of the limiting distribution.  This leads to a preliminary
version of the CLT in Section~\ref{b:perturbation-CLT}, expressed in
terms of tangent perturbations, that collects the various limits of
random tangent objects.

\subsection{Representation of stratified Gaussians}\label{b:represent}
\mbox{}\medskip

\noindent
The right-hand side of any geometric CLT is a stratified-space
generalization of a Gaussian distribution.  The purpose of this
subsection is to construct a random tangent vector whose distribution
is Gaussian in Definition~\ref{d:gaussian-section} and
Theorem~\ref{t:Riesz-representation}.  The crux of the theory needed
for that purpose is representability of the Gaussian tangent field~$G$
(Theorem~\ref{t:Riesz-representation}), in the sense of Riesz
representation, which allows escape vectors to reduce convergence in
the desired CLT to the Gaussian tangent field version in
Theorem~\ref{t:tangent-field-CLT}.  This representability is one of
the main consequences of \mbox{tangential}~\mbox{collapse}
$$
  \LL: \Tmu \to \RR^m
$$
as constructed in \cite{tangential-collapse}.  By
Definition~\ref{d:collapse} and Theorem~\ref{t:collapse}, the
collapse~$\LL$
\begin{itemize}
\item%
preserves the Fr\'echet mean,
\item%
restricts to an isometry from the fluctuating cone~$\Cmu$ to its
image, and
\item%
preserves angles between arbitrary and fluctuating tangent vectors.
\end{itemize}
Leveraging tangential collapse to deduce representability proceeds
by pushing various measures and inner products through~$\LL$, which
begins as follows.

\begin{defn}\label{d:euclidean-gaussian}
Fix a tangential collapse $\LL: \Tmu \to \RR^m$ of the measure~$\mu$
on a smoothly stratified metric space~$\MM$ as in
Theorem~\ref{t:collapse}.  Write $\hmu = (\log_\bmu)_\sharp\mu$ for
the pushforward of~$\mu$ under the log map at~$\bmu$.  The
\emph{collapsed measure} is
$$
  \muL
  =
  \LL_\sharp\hmu
  =
  (\LL \circ \log_\bmu)_\sharp \mu.
$$
As $\RR^m$ is a vector space, average the random variables $\LL(X_i)$
for $X_i = \log_\bmu x_i$ to get
$$
  \oXnL
  =
  \frac 1n \sum_{i=1}^n \LL(\log_\bmu x_i)
  =
  \frac 1n \sum_{i=1}^n \LL(X_i).
$$
\end{defn}

\begin{remark}\label{r:hypotheses}
Definition~\ref{d:euclidean-gaussian} and many other items in this
section do not assume that the measure~$\mu$ is amenable, localized,
or immured.  These hypotheses enforce certain guarantees---the
amenable hypothesis guarantees that escape vectors are well defined by
Theorem~\ref{t:escape-vector-confinement}, for example, and the
immured hypothesis forces escape vector confinement to the fluctuating
cone by Corollary~\ref{c:confinement-to-Cmu}, while the localized
hypothesis guarantees that tangential collapses exist by
Theorem~\ref{t:collapse}---but these hypotheses are omitted from
definitions and results where the constructions make sense without the
hypotheses on~$\mu$ and could in theory be guaranteed by some other
(perhaps weaker) assumptions.  Thus
Definition~\ref{d:euclidean-gaussian} stipulates that a
collapse has been fixed, but it is agnostic about whether that
collapse might have resulted from assuming that $\mu$ is localized.
The same can be said of, for instance, Lemma~\ref{l:supp-NLmu}.
\end{remark}

Vanishing of expected inner products with fluctuating vectors~$V$ in
Remark~\ref{r:nablamuF(V)=m(mu,V)} makes empirical tangent fields
(Definition\,\ref{d:empirical-tangent-field}) act cleanly under
tangential collapse.

\begin{defn}\label{d:gaussian-resolving-vector}
In the setup of Definition~\ref{d:euclidean-gaussian}, a
\emph{Gaussian resolving vector} of~$\mu$ is a Gaussian distributed
vector
$$
  \NLmu \sim \cN(0,\Sigma)
$$
in~$\RR^m$ whose covariance is $\Sigma = \Cov(\muL)$.
\end{defn}

\begin{remark}\label{r:classical-CLT-muL}
The classical CLT for~$\muL$ reads $\sqrt n\,\oXnL \overset{d\;}\to
\NLmu$, or equivalently, in terms of the measure $\muL = (\LL \circ
\log_\bmu)_\sharp \mu$ with covariance $\Sigma = \Cov(\muL)$,
$$
  \lim_{n\to\infty} \sqrt n\,\oXnL \sim \cN\bigl(0,\Sigma\bigr).
$$
The support of this Gaussian is a linear subspace of~$\RR^m$, possibly
a proper subspace, characterized directly in terms of~$\mu$ and~$\LL$
in Lemma~\ref{l:supp-NLmu}.
\end{remark}

\pagebreak[2]

\begin{defn}\label{d:section}
Fix a tangential collapse $\LL: \Tmu \to \RR^m$ of an amenable
measure~$\mu$ on a smoothly stratified metric space~$\MM$, and write
$\RR^\ell = \hull\muL \subseteq \RR^m$.
\begin{enumerate}
\item\label{i:section}%
A \emph{section} of an $\RR^\ell$-valued random variable~$Y$ is a
random variable $\Delta$ valued in the space $\Disc(\Rmu)$ of measures
sampled from~$\Rmu$ (Definition~\ref{d:Rmu}) satisfying $\LL(\Delta) =
Y$ (Definition~\ref{d:LL(Delta)}).  Write $\Delta \in \Lmui(Y)$ to
indicate $\Delta$ is a section of~$Y$.
\item\label{i:dual-vector}%
Any section $\wXnL \in \Lmui(\oXnL)$ is a \emph{dual vector}
of~$\bgn$.
\end{enumerate}
\end{defn}

\begin{remark}\label{r:section}
Sometimes in probabilistic contexts sections are called selections.
Geometrically, and in formal probabilistic terms, $Y: \Omega \to
\RR^\ell$ in Definition~\ref{d:section} is a measurable function, so a
section $\Delta: \Omega \to \Disc(\Rmu)$ of~$Y$ can literally be
induced by a section $\sigma: \RR^\ell \to \Disc(\Rmu)$ in the usual
sense that $\LL \circ \sigma = \mathrm{id}_{\RR^\ell}$.  Such sections
exist for general reasons:
\begin{itemize}
\item%
proper continuous maps between compact Hausdorff spaces admit Borel
measurable sections \cite[Theorem~6.9.6]{bogachev2007}, and
\item%
any convex projection admits a continuous section because the convex
set that is the target is already a subset of the source.
\end{itemize}
Although the map $\Disc(\Rmu) \to \RR^\ell$ induced by a
collapse~$\LL$
from Theorem~\ref{t:collapse} need not itself be proper, for example
because the convex projection can collapse entire strata (compact or
otherwise) to the cone point, thanks to
Theorem~\ref{t:collapse}.\ref{i:proper} the
collapse $\LL$ is a proper map followed by a convex projection, both
of which admit measurable~sections.
\end{remark}

Recall Lemma~\ref{l:supp-NLmu}: any Gaussian resolving vector~$\NLmu$
is supported on~$\RR^\ell = \hull\muL$.

\begin{defn}[Gaussian {}mass]\label{d:gaussian-section}
In the setup of Definitions~\ref{d:euclidean-gaussian}
and~\ref{d:gaussian-resolving-vector}, a \emph{Gaussian {}mass}
induced by~$\mu$ is a section
$$
  \Gmu \in \Lmui(\NLmu)
$$
of a Gaussian resolving vector~$\NLmu$ from
Definition~\ref{d:section}.
\end{defn}

\begin{remark}\label{r:Gmu}
The stratified-space analogue of a normally distributed vector is the
notion of Gaussian {}mass~$\Gmu$ valued in measures sampled
from~$\Tmu$ as in Definition~\ref{d:gaussian-section}.  It implements
a transformative shift in perspective: geometrically in~$\Tmu$, the
supports of the measures~$\Gmu(\omega) \Subset \Tmu$ for
events~$\omega \in \Omega$ might all be constrained to an unexpectedly
small subset of~$\Tmu$, such as when $\mu$ itself has finite support.
This stands in contrast to classical Gaussian distributions, which are
geometrically continuous because sample means vary throughout full
neighborhoods of the population mean.  In singular settings, whatever
continuous variation is lacking in the support of~$\mu$ is reallocated
to the weights in positive combinations of masses at support points.
\end{remark}

Two lemmas are required for the proof of
Theorem~\ref{t:Riesz-representation}.

\begin{lemma}\label{l:equal-LL=>equal-inner-prod}
Fix a tangential collapse $\LL: \Tmu \to \RR^m$ of a measure~$\mu$
on a smoothly stratified space~$\MM$ with measures $\Delta$
and~$\pDelta\!$ sampled from~$\Tmu$ such that $\LL(\Delta)
=\nolinebreak \LL(\pDelta)$.  Then $\<\Delta, X\> = \<\pDelta, X\>$
for all $X$ in the closure~$\oCmu$ of the fluctuating~cone.
\end{lemma}
\begin{proof}
Definitions~\ref{d:pair-with-discrete-measure}, \ref{d:LL(Delta)},
and~\ref{d:collapse}.\ref{i:partial-isometry} imply $\<\Delta, X\> =
\<\pDelta, X\>$ for any fixed $X \in\nolinebreak \Cmu$, given the
equality between images under~$\LL$.  Use continuity of $\LL$ in
Definition~\ref{d:collapse}.\ref{i:continuous} and of inner product
in Lemma~\ref{l:inner-product-is-continuous} to extend the conclusion
to~$\oCmu$.
\end{proof}

\begin{lemma}\label{l:projected-bgn}
In the setup of Definition~\ref{d:euclidean-gaussian}, for
localized~$\mu$ the
tangent field\/~$\bgn$
satisfies
$$
  \bgn(X)
  =
  \bigl\<\oXnL, \LL(X)\bigr\>_{\LL(\bmu)}
$$
for all\/ $X \hspace{-.5ex}\in\hspace{-.3ex} \oCmu$, where the inner
product $\<\,\cdot\,,\,\cdot\,\>_{\LL(\bmu)}$ on~$\RR^m\!$ is natural
from Theorem~\ref{t:collapse}.\ref{i:stratum}.
\end{lemma}
\begin{proof}
The inner product on~$\RR^m$ comes from
Definition~\ref{d:inner-product} via
Theorem~\ref{t:collapse}.\ref{i:proper} because each stratum of~$\MM$
is itself a smoothly stratified metric space.
Remark~\ref{r:nablamuF(V)=m(mu,V)} allows the tangent mean function
terms $m(\mu,X)$ in~$\bgn$ from
Definition~\ref{d:empirical-tangent-field} to be ignored (this is
where the localized hypothesis enters.
Lemma~\ref{l:equal-LL=>equal-inner-prod} implies that the inner
products in~$\bgn$ are unchanged by pushing forward along~$\LL$, and
the sum can be taken inside the inner product because $\RR^m$ is
linear.
\end{proof}

\begin{remark}\label{r:dual-vector}
Lemma~\ref{l:projected-bgn} is equivalent to $\bgn(X) = \<\wXnL,
X\>_\bmu$ for all $X \in \Cmu$.  Indeed, by
Definition~\ref{d:collapse}.\ref{i:partial-isometry} tangential
collapse~$\LL$ does not alter inner products with vectors in~$\Cmu$,
so the assertion holds by
Definition~\ref{d:section}.\ref{i:dual-vector}.  This rephrasing of
Lemma~\ref{l:projected-bgn} says that $\bgn$ is the dual in~$\Tmu$ of
the random vector~$\wXnL$ under the inner product
$\<\,\cdot\,,\,\cdot\,\>_\bmu$ in Definition~\ref{d:inner-product}.
In contrast, Lemma~\ref{l:projected-bgn} says that $\bgn$ is
expressible in terms of genuine duality in~$\RR^m$ under the inner
product $\<\,\cdot\,,\,\cdot\,\>_{\LL(\bmu)}$.
\end{remark}

\pagebreak[3]

\begin{thm}\label{t:Riesz-representation}
Fix a localized measure~$\mu$ on a smoothly stratified metric space.
Gaussian {}masses~$\Gmu$ in Definition~\ref{d:gaussian-section}
exist, and the Gaussian tangent field~$G$ in
Theorem~\ref{t:tangent-field-CLT} is, on the closed fluctuating cone,
the inner \mbox{product}~with~any~$\Gmu$:
$$
  G(X) = \<\Gmu,X\>_\bmu \text{ for all }X \in \oCmu.
$$
\end{thm}
\begin{proof}
Fix a tangential collapse~$\LL$ of the localized measure~$\mu$ by
Theorem~\ref{t:collapse}.  Existence of~$\Gmu$ has been noted in
Remark~\ref{r:section}.
For the rest,
$$
\bigl\<\sqrt n\,\oXnL, \LL(X)\bigr\>_{\LL(\bmu)}
   = \sqrt n \bigl\<\oXnL, \LL(X)\bigr\>_{\LL(\bmu)}
   = \sqrt n\,\bgn(X)
   = G_n(X)
$$
for all $X \in \Cmu$, where the leftmost equality is because inner
products commute with scaling by Definition~\ref{d:inner-product}, the
middle equality is Lemma~\ref{l:projected-bgn}, and the rightmost
equality is simply by Definition~\ref{d:empirical-tangent-field}.  The
classical CLT for~$\muL$ in Remark~\ref{r:classical-CLT-muL} says that
the leftmost inner product converges to $\bigl\<\NLmu,
\LL(X)\bigr\>{}_{\LL(\bmu)}$, which equals $\<\Gmu, X\>_\bmu$ by
invariance of inner product under~$\LL$ in
Lemma~\ref{l:equal-LL=>equal-inner-prod}.  On the other hand, the
functional CLT for random fields in Theorem~\ref{t:tangent-field-CLT}
says that the~$G_n$ on the right converges to~$G$.
\end{proof}

\subsection{CLT by perturbation}\label{b:perturbation-CLT}

\begin{remark}\label{r:aspires}
The goal to which the paper aspires is identification of the limit of
$\sqrt n \log_\bmu \bmu_n$ as $n \to \infty$.  The following central
limit theorem does so in a way that encapsulates all of the
convergences contributing to the main geometric CLT, namely
Theorem~\ref{t:clt}.  The various interpretations and window dressings
in later subsections notwithstanding, the relevant limit is actually
taken only once, and it occurs in the proof of
\ref{t:perturbative-CLT}, which applies the continuous mapping theorem
to the deterministic analogue in
Corollary~\ref{c:convergence-of-Upsilon}, which is stated precisely
for this purpose.
\end{remark}

For the reader flipping to this spot from elsewhere, it may help to
recall explicitly from Definition~\ref{d:H} that
\begin{align*}
  H_n
  &\in
  \argmin_{X \in \Tmue}\bigl(F(\exp_\bmu X) - \bgn(X)\bigr)
\\ \text{and}\quad
  H(t)
  &\in
  \argmin_{X \in \oCmue}\bigl(F(\exp_\bmu X) - tG(X)\bigr),
\end{align*}
where $\bgn$ is from Definition~\ref{d:empirical-tangent-field} and
$G$ is from Definition~\ref{d:gaussian-tangent-field} (see also
Theorem~\ref{t:tangent-field-CLT} for both), and the escape vector
$\ev$ is characterized in Theorem~\ref{t:escape-vector-confinement},
including polar coordinates, although it first appears in
Definition~\ref{d:delta-escape}.  Gaussian {}masses~$\Gmu$ are
introduced in Definition~\ref{d:gaussian-section} and characterized in
Theorem~\ref{t:Riesz-representation}.

\begin{thm}[Perturbative~CLT]\label{t:perturbative-CLT}
Fix a localized immured amenable measure~$\mu$ on a smoothly
stratified metric space~$\MM$.  The empirical Fr\'echet mean~$\bmu_n$,
empirical and Gaussian tangent perturbations $H_n$ and~$H(t)$, and
escape vector~$\ev(\Gmu)$ of any Gaussian {}mass~$\Gmu$~satisfy
$$
  \lim_{n\to\infty} \sqrt n \log_\bmu \bmu_n
  \overset{d}=
  \lim_{n\to\infty} \sqrt n H_n
  \overset{d}=
  \lim_{t\to 0} \frac 1t H(t)
  =
  \ev(\Gmu).
$$
In terms of the Gaussian tangent field~$G$, this limit has polar
coordinates
$$
  \theta_{\ev(\Gmu)}
  = \argmax_{X\in\oSCmu} \frac{G(X)}{\sqrt{\Lambda_\bmu(X)}}
  \quad\text{and}\quad
  r_{\ev(\Gmu)} = \frac{G(\theta_{\ev(\Gmu)})^+}{2\Lambda_\bmu(\theta_{\ev(\Gmu)})}.
$$
\end{thm}
\begin{proof}
The leftmost equality is the logarithm of
Proposition~\ref{p:transform} (which uses amenability), since $H_n =
\log_\bmu h_n$ by Definition~\ref{d:H}.  The rightmost equality is
because
\begin{align*}
\lim_{t\to 0} \frac 1t H(t)
& = \lim_{t\to 0}\frac 1t \argmin_{X \in \oCmue}\bigl(F(\exp_\bmu X) - tG(X)\bigr)
\\*
& = \lim_{t\to 0}\frac 1t \argmin_{X \in \oCmue}\bigl(F(\exp_\bmu X) - t\<\Gmu,X\>\bigr)
\\*
& =
  \ev(\Gmu),
\end{align*}
where the top line is by Definition~\ref{d:H}, the second line is by
Theorem~\ref{t:Riesz-representation} (which uses the localized
hypothesis), and the bottom line is
Theorem~\ref{t:escape-vector-confinement}(\ref{b}).

To prove the middle equality, write $\ev_0: \cC(\Tmu,\RR) \to \Tmu$
for the map
\begin{align*}
  \ev_0: R &\mapsto \lim_{t \to 0}
                   \frac 1t \argmin_{X\in\oCmue}\bigl(F(\exp_\bmu X)-tR(X)\bigr)
\intertext{defined on the space of functions from
Definition~\ref{d:sup-norm-on-continuous-functions}, and denote
by~$\cR$ the set of representable functions from
Definition~\ref{d:representable} to write $\evn: \cR \to \Tmu$ for the
maps}
  \evn:  R &\mapsto \frac 1{t_n} \Upt(t_n R)
\end{align*}
for all $n \geq 1$.  Corollary~\ref{c:convergence-of-Upsilon} (which
needs $\mu$ to be immured) says that whenever $R_n \to R$ is a
representable limit as in Definition~\ref{d:representable}, $\evn(R_n)
\to\nolinebreak \ev_0(R)$, where $\evn$ is defined on~$\DD_n = \cR$
for~$n \geq 1$ and~$\ev_0$ is defined on~$\DD_0 = \cC(\Tmu,\RR)$.  The
extended continuous mapping theorem \cite[Theorem~1.11.1]{VW13} for
the limit $R_n = \sqrt n\,\bar g_n \to G = R$ from
Corollary~\ref{c:cont-realization}, which is representable by
Example~\ref{e:representable}, yields the desired equality.
\end{proof}

\begin{remark}\label{r:confined-in-the-limit}
$H_n$ is a~priori unconfined but is proved, through the work in
Section~\ref{s:confinement}, to be confined to~$\oCmu$ in the limit.
That is what Corollary~\ref{c:convergence-of-Upsilon} says in the
deterministic setting, and the proof of
Theorem~\ref{r:confined-in-the-limit} via the extended continuous
mapping theorem pushes it to the probabilistic setting.
\end{remark}

\subsection{CLT by tangential collapse and distortion}\label{b:clt}

\begin{remark}\label{r:lift-and-escape}
Gaussian {}masses are constructed in
Definition~\ref{d:gaussian-section} by pushing the measure~$\mu$
forward through the logarithm map~$\log_\bmu$ and a tangential
collapse $\LL: \Tmu \to \RR^m$, applying the classical CLT
in~$\RR^m$ to the resulting pushforward measure~$\muL$ to get a
Gaussian random variable~$\NLmu \sim \cN(0,\Sigma)$ on~$\RR^\ell
\subseteq \RR^m$, and taking a section~$\Gmu \in \Lmui(\NLmu)$.
Sections of this sort can be far from unique, both~because
\begin{itemize}
\item%
a single point in~$\RR^\ell$ can be representable by multiple convex
combinations, and
\item%
tangential collapse entails convex projection
(Theorem~\ref{t:collapse}.\ref{i:proper}); see \cite[\S1.4.2,
particularly Fig.~2]{kale-2015} for a quintessential example of
collapsing that occurs.
\end{itemize}
However, inner products of fluctuating and Gaussian {}masses do
not depend on the choice of section in
Definition~\ref{d:gaussian-section}.  This invariance implies that the
escape vector of~$\Gmu$ does not depend on the choice of section, in
the following precise sense, a consequence of
Theorem~\ref{t:escape-vector-confinement} (via
Corollary~\ref{c:confinement-to-Cmu}) and Theorem~\ref{t:collapse}.
\end{remark}

\begin{prop}\label{p:reduced-escape}
Assume the hypotheses of Lemma~\ref{l:equal-LL=>equal-inner-prod} and
that~$\mu$ is immured and amenable.  If $\Delta$ and~$\pDelta$ are
sampled from $\Rmu$ then their escape vectors coincide:
$$
  \ev(\Delta) = \ev(\pDelta) \text{ if } \Delta + \pDelta \Subset \Rmu.
$$
\end{prop}
\begin{proof}
Starting from Lemma~\ref{l:equal-LL=>equal-inner-prod}, confinement in
Corollary~\ref{c:confinement-to-Cmu} (note the immured hypothesis)
allows use of the formula in Theorem~\ref{t:escape-vector-confinement}
with the $\argmin$ over the fluctuating cone~$\oCmu$ instead of the
escape cone to conclude equality of escape~vectors.
\end{proof}

\begin{remark}\label{r:reduced-escape}
Lemma~\ref{l:reduced-escape} ensures that a discrete measure $\Delta =
\Lmui(V)$ in the next definition exists.
Proposition~\ref{p:reduced-escape} guarantees that the choice
of~$\Delta$ does not matter.  Distortion assumes the measure to be
amenable for escape vectors to be defined
(Theorem~\ref{t:escape-vector-confinement}) and immured for
Proposition~\ref{p:reduced-escape}.  The distortion map is one of the
key concepts in central limit theory for Fr\'echet means on stratified
spaces; see Section~\ref{b:distortion} for a discussion of the
geometry of distortion.
\end{remark}

\begin{defn}[Distortion]\label{d:distortion-map}
Fix a tangential collapse $\LL: \Tmu \to \RR^m$ of an immured
amenable measure~$\mu$, and set $\RR^\ell = \hull\muL \subseteq
\RR^m$.  The \emph{distortion map} is
\begin{align*}
  \HH: \RR^\ell & \to \RR^\ell
\\            V & \mapsto \ev \circ \Lmui(V).
\end{align*}
\end{defn}

\begin{example}\label{e:smooth-fluctuating-cone}
If~$\mu$ is a localized measure on a smoothly stratified space~$\MM$
and the Fr\'echet mean~$\bmu$ lies in a maximal stratum of~$\MM$, such
as when $\MM$ itself is a manifold, then the fluctuating cone~$\Cmu$
is a vector space.  Indeed, since $\Chmu =
\Cmu$
\cite[Lemma~2.19]{tangential-collapse},
computing~$\Cmu$ may as well be carried out with the pushforward
measure~$\hmu$ on~$\Tmu$.  Since $\bmu$ lies in a maximal stratum,
\cite[Proposition~3.13]{tangential-collapse}
implies that the escape cone contains~$\Tmu$ and hence, by stratum
maximality, $\Emu = \Tmu$ is a vector space.  But in a vector space,
the Fr\'echet mean of any measure lies relative interior to the convex
hull of the support, so $\hull\hmu = \RR^\ell$.  (This direct argument
can be shortened to a simple citation of Lemma~\ref{l:supp-NLmu}, if
desired: $\log_\bmu$ is a tangential collapse here, so $\hmu =
\muL$.)  Distortion in this case is the inverse of the Hessian of the
Fr\'echet function (see Remark~\ref{r:inverse-hessian}) and therefore
encodes local geometric information~around~$\bmu$.
\end{example}

The next central limit theorem
collapses the singularity to get a Euclidean space~$\RR^m$ and then
corrects the distortion introduced by pushing the measure forward
to~$\RR^m$.

\pagebreak[3]

\begin{thm}[Geometric~CLT]\label{t:clt}
Fix a smoothly stratified metric space~$\MM$.  Any localized immured
amenable probability measure~$\mu$ on~$\MM$ admits a tangential
collapse $\LL: \Tmu \to\nolinebreak \RR^m$.
Using the resulting distortion $\HH: \RR^\ell \to \RR^\ell$ for
$\RR^\ell = \hull\muL \subseteq \RR^m$, the rescaled Fr\'echet sample
means~$\bmu_n$ satisfy
$$
  \lim_{n\to\infty} \sqrt n \log_\bmu \bmu_n
  \sim
  \HH_\sharp\cN(0,\Sigma),
$$
where $\cN(0,\Sigma)$ is the Gaussian measure on~$\RR^\ell$ with the
same covariance $\Sigma = \Cov(\muL)$ as the pushforward $\muL = (\LL
\circ \log_\bmu)_\sharp\mu$ of the measure~$\mu$ to~$\RR^m$.
\end{thm}
\begin{proof}
Tangential collapse is produced by Theorem~\ref{t:collapse}.
Theorem~\ref{t:perturbative-CLT} then says that
\noheight{$\lim_{n\to\infty} \sqrt n \log_\bmu \bmu_n
\hspace{-.2ex}\overset{d}= \ev(\Gmu)$}.  By
Definition~\ref{d:gaussian-section} $\ev(\Gmu) \in \ev(\Lmui \NLmu)$
for any Gaussian resolving vector~$\NLmu$ from
Definition~\ref{d:gaussian-resolving-vector}.  Hence by
Definition~\ref{d:distortion-map} $\ev(\Gmu) = \HH(\NLmu)$, whose law
is $\HH_\sharp\cN(0,\Sigma)$ by definition of~$\NLmu$ and of
pushfoward measure.
\end{proof}

\begin{example}\label{e:manifold}
All of the central limit theorems, namely
Theorems~\ref{t:perturbative-CLT}, \ref{t:clt},
\ref{t:barycenter-CLT}, and~\ref{t:minimizer-CLT}, are unrestricted in
the case where~$\MM$ is a manifold by
Example~\ref{e:smooth-fluctuating-cone}.  In particular,
Theorem~\ref{t:clt} specializes directly to the
form~\eqref{eq:smoothCLT} of the smooth CLT.
\end{example}

\begin{remark}\label{r:non-linearity-of-HH}
The distortion map $\HH$ need not be continuous on~$\RR^\ell$.
Discontinuity of~$\HH$ can occur because tangential collapse~$\LL$
squashes closed sets in~$\Tmu$, so it behaves as a quotient map.  As a
result, even in simple cases when $\Tmu = \RR^m$ is a Euclidean space,
the limiting distribution $\HH_\sharp\cN(0,\Sigma)$ in
Theorem~\ref{t:clt} need not be Gaussian; see
\cite[Example~2.5]{kale-2015} for an example of this behavior.
\end{remark}

\subsection{CLT as derivative in a space of measures}\label{b:barycenter}
\mbox{}\medskip

\begin{remark}\label{r:bb}
Taking a Fr\'echet mean can be viewed as a map to~$\MM$ from the space
$\cP_2(\MM)$ of $L^2$-measures~on~$\MM$,
\begin{align*}
  \bb : \cP_2(\MM) &\to \MM
\\
               \mu &\mapsto \bmu.
\end{align*}
The derivative of such a map at a measure $\mu \in \cP_2(\MM)$ has the
form
$$
  \nablamu\bb: T_\mu\cP_2(\MM) \to \Tmu.
$$
\end{remark}

\begin{remark}\label{r:direction-deriv-barycenter}
An element of~$T_\mu\cP_2(\MM)$ is the initial tangent at $\mu \in
\cP_2(\MM)$ of a curve~$\gamma(t)$ for $t \geq 0$ in $\cP_2(\MM)$ that
emanates from~$\mu$ at $t = 0$.  For example, if $x \in \MM$ is any
point, then the measures $\mu + t\delta_x$ constitute such a curve
in~$\cP_2(\MM)$.  These curves give rise to many copies of~$\Tmu$
inside of~$T_\mu\cP_2(\MM)$: for any fixed positive $s \ll 1$ and any
vector~$X \in \Smu$ in the unit tangent sphere $\Smu \subseteq \Tmu$,
set $x(s) = \exp_\bmu(sX)$,~so
$$
  \Tmus
  =
  \{\mu + \tfrac ts \delta_{x(s)} \in \cP_2\MM
    \mid t \geq 0 \text{ and }
         X \in \Smu\}
$$
is such a copy by
Definition~\ref{d:stratified-space}.\ref{i:exponentiable}.  In
principle, these embedded copies~$\Tmus$ could give rise to different
notions of directional derivative of~$\bb$ at~$\mu$ along various
vectors $X \in \Tmu$, but homogeneity prevents that, as proved in
Lemma~\ref{l:bb}.
\end{remark}

\begin{defn}\label{d:bb-directional-deriv}
Fix an amenable measure~$\mu$ on a smoothly stratified metric
space~$\MM$.  The \emph{directional derivative of\/~$\bb$} at~$\mu$
along~$X$ in the space $\cP_2(\MM)$ of measures is
\begin{align*}
  \nablamu\bb: \Tmu
    & \to \Tmu
\\
  X &\displaystyle
     \mapsto \frac 1s \lim_{t\to 0} \frac 1t \log_\bmu
             \bb\bigl(\mu + t\delta_{\exp_\bmu sX}\bigr).
\end{align*}
for any positive $s \ll 1$, the particular choice of $s$ being
irrelevant by Lemma~\ref{l:bb}.
\end{defn}

\begin{lemma}\label{l:bb}
The limit in Definition~\ref{d:bb-directional-deriv} equals the escape
vector $\ev(X)$ from Definition~\ref{d:escape-vector}, independent of
the choice of~$s$:
$$
  \nablamu\bb(X) = \ev(X) \text{ for all } X \in \Tmu.
$$
\end{lemma}
\begin{proof}
Apply Corollary~\ref{c:homogeneity} to the equivalence of \eqref{a}
and~\eqref{b} in Theorem~\ref{t:escape-vector-confinement}.
\end{proof}

\begin{remark}\label{r:influence}
Lemma~\ref{l:bb} characterizes escape as a singular-space version of
the notion of influence function in robust statistics (see
\cite{mathieu2022} for a recent exposition and discussion of history),

a connection to be explored in future work.  See also
Remark~\ref{r:view-of-Upsilon} for related considerations.
\end{remark}

\begin{remark}\label{r:elegant-directional-deriv}
When the injectivity radius of~$\MM$ is infinite, such as when $\MM$
is $\CAT(0)$, the directional derivative of~$\bb$ can be elegantly
characterized without~$s \ll 1$~as
\begin{align*}
  \nablamu\bb: \Tmu
   & \to \Tmu
\\
  V&\mapsto \lim_{t\to 0} \frac 1t \log_\bmu \bb(\mu + t\delta_{v}).
\end{align*}
\end{remark}

\begin{remark}\label{r:semidifferentiable}
In the language of \cite[Definition~7.20]{rockafellar-wets2009}, the
map $\bb$ is semidifferentiable when the space of directions is taken
to be the set of point masses because
$$
  \lim_{t_n\to 0} \frac 1{t_n} \log_\bmu \bb(\mu + t_n\delta_{x_n})
  =
  \lim_{t\to 0} \frac 1t \log_\bmu \bb(\mu + t\delta_{x})
$$
when $x_n \to x$, which holds by
Theorem~\ref{t:escape-vector-confinement}.
\end{remark}

\begin{remark}\label{r:inverse-hessian}
It is shown in \cite{Tra20} that when $\MM$ is a Riemannian manifold,
$\nablamu\bb(\delta_v)$ coincides with the inverse of the Hessian
at~$\bmu$ of the Fr\'echet function of~$\mu$:
$$
  \nablamu\bb(V) = (\text{Hess}_\bmu F)^{-1}(V).
$$
\end{remark}

This discussion leads to an intrinsic characterization of the limiting
distribution in the stratified CLT in terms of Gaussian
{}masses~$\Gmu$ Definition~\ref{d:gaussian-section}.

\begin{thm}\label{t:barycenter-CLT}
The limiting distribution in the CLT for~$\bmu_n$
(Theorem~\ref{t:clt}) is the directional derivative, in the space
$\cP_2(\MM)$ of measures, of the barycenter map~$\bb$ at~$\mu$ in the
direction of a random discrete measure given by any Gaussian
{}mass~$\Gmu$:
$$
  \lim_{n\to\infty} \sqrt n \log_\bmu \bmu_n
  \overset{d}=
  \nablamu\bb(\Gmu).
$$
\end{thm}
\begin{proof}
Theorem~\ref{t:perturbative-CLT} characterizes the limiting
distribution in Theorem~\ref{t:clt} as the escape vector~$\ev(\Gmu)$,
which equals $\nablamu\bb(\Gmu)$ by Lemma~\ref{l:bb}.
\end{proof}

\begin{remark}\label{c:correction_map_v2}
Recall from Definition~\ref{d:gaussian-section} that tangential
collapse of the Gaussian {}mass~$\Gmu$ is a Euclidean
Gaussian vector $\LL(\Gmu) = \NLmu$ in $\RR^m$ with mean $0$ and
covariance $\Sigma = \Cov(\muL)$ for the pushforward $\muL = (\LL
\circ \log_\bmu)_\sharp \mu$ of the measure~$\mu$ to~$\RR^m$.  The
distortion map in Definition~\ref{d:distortion-map} can alternatively
be characterized as
\begin{align*}
  \HH: \RR^\ell & \to \Tmu
\\
              X &\mapsto \nablamu\bb(\Lmui X),
\end{align*}
which does not depend on the section choice~$\Lmui X$ by
Proposition~\ref{p:reduced-escape} and Lemma~\ref{l:bb}.
\end{remark}

\subsection{CLT as derivative in a space of functions}\label{b:minimizer}

\begin{defn}\label{d:minimizer-map}
Fix a measure~$\mu$ on a smoothly stratified metric space~$\MM$.  The
\emph{minimizer map} on continuous functions from the tangent cone
to~$\RR$,
\begin{align*}
\BB: \cC(\Tmu,\RR) &\to \Tmu
\\
                 f &\mapsto \BB(f) \in \argmin_{X \in \Tmu} f(X),
\end{align*}
takes each function~$f$ to a choice of minimizer.  (``$\BB$'' here is
for ``barycenter''.)
\end{defn}

\begin{remark}\label{r:view-of-Upsilon}
In terms of the function~$\Upt$ from Definition~\ref{d:Upsilon}, the
limit $\lim_{t \to 0} \frac 1t \Upt(tR)$ can be viewed as the
\emph{directional derivative} at \noheight{$\tF = F \circ \exp_\bmu$}
along~$-R$ of the minimizer map.  In other words,
$$
  \lim_{t \to 0} \frac{\Upt(tR)}{t}
  =
  \frac 1t \BB\bigl(\tF(X) - tR(X)\bigr)
  =
  \nablatF\BB(-R).
$$
This is a G\^ateaux derivative when $\MM$ is a manifold, but not when
$\bmu$ is a singular point of~$\MM$, because there the target~$\Tmu$
is not a vector space.  For functions~$-R$ such that $R$ is
representable as in Definition~\ref{d:representable},
Theorem~\ref{t:convergence-of-Upsilon} shows that the directional
derivative of~$\BB$ at~$\tF$ exists in the direction~$-R$: simply take
$R_n = R$ for all~$n$ there.  Alas, the goal is to apply this thinking
to Gaussian tangent fields~$G$ from
Definition~\ref{d:gaussian-tangent-field}, which are not themeselves
representable but are only limits of such; see
Theorem~\ref{t:tangent-field-CLT}, Corollary~\ref{c:cont-realization},
and especially Theorem~\ref{t:Riesz-representation}, the latter saying
that $G$ is representable on the closed fluctuating cone~$\oCmu$
instead of on the escape cone~$\Emu$ as required by representability.
Although it appears likely that
\begin{align}\label{eq:tH}
  \lim_{t \to 0} \frac 1t \argmin_{X \in \Tmue}\bigl(F(\exp_\bmu X) - tG(X)\bigr)
  =
  \lim_{t \to 0} \frac 1t \argmin_{X \in \oCmue}\bigl(F(\exp_\bmu X) - tG(X)\bigr),
\end{align}
the left-hand side being $\lim_{t \to 0} \frac 1t \Upt(tG)$ and the
right-hand side being $\lim_{t \to 0} \frac 1t H(t)$ by
Definition~\ref{d:H}, it does not appear to be immediate from our
methods that this confinement from~$\Tmue$ to~$\oCmue$ holds.  This
forces us, for the current purpose, to define a confined minimizer as
follows.
\end{remark}

\begin{defn}\label{d:minimizer-confined}
Fix a measure~$\mu$ on a smoothly stratified metric space~$\MM$.  The
\emph{confined minimizer map} on continuous functions from the tangent
cone to~$\RR$,
\begin{align*}
\fC: \cC(\Tmu,\RR) & \to \Tmu
\\
                 f &\mapsto \fC(f) \in \argmin_{X \in \oCmue} f(X),
\end{align*}
takes each function~$f$ to a choice of minimizer in the closed
fluctuating cone~$\oCmue$.  (``$\fC$''~here is for ``confined''.)
\end{defn}

The language of Remark~\ref{r:view-of-Upsilon} and
Definition~\ref{d:minimizer-confined} allows the limiting distribution
in the CLT for~$\bmu_n$ to be seen as the directional derivative of
the confined minimizer~map~$\fC$.

\begin{prop}\label{p:directional-derivative-BB}
Fix an amenable measure~$\mu$ on a smoothly stratified metric
space~$\MM$ and select $\fC$ as in Definition~\ref{d:minimizer-map}.
If
$$
  \tF = F\circ \exp_\bmu: \Tmue \to \RR,
$$
whose domain is the subset~$\Tmue \subseteq \Tmu$ of vectors that can
be exponentiated, then
$$
  \nablatF \fC(-G)
  =
  \lim_{t\to 0} \frac{\fC(\tF - tG)}{t}
  =
  \ev(\Gmu)
$$
is well defined on the Gaussian tangent field~$G$ induced by~$\mu$
(Definition~\ref{d:gaussian-tangent-field}) and~$\ev(\Gmu)$ is the
escape vector (Definition~\ref{d:delta-escape}) of any Gaussian
{}mass~\mbox{(Definition~\ref{d:gaussian-section})}.
\end{prop}
\begin{proof}
The expression inside the limit is $\frac 1t H(t)$ by
Definitions~\ref{d:minimizer-confined} and~\ref{d:H}.  The result is
then the rightmost equality in Theorem~\ref{t:perturbative-CLT}.
\end{proof}

\begin{thm}\label{t:minimizer-CLT}
The limiting distribution in the CLT for~$\bmu_n$ in
Theorem~\ref{t:clt} is the directional derivative, in the space
$\cC(\Tmu,\RR)$, of the confined minimizer map~$\fC$ at the Fr\'echet
function $\tF$ along the negative of the Gaussian tangent field~$G$
induced by~$\mu$:
$$
  \lim_{n\to\infty} \sqrt n \log_\bmu \bmu_n
  \overset{d}=
  \nablatF \fC(-G).
$$
\end{thm}
\begin{proof}
Proposition~\ref{p:directional-derivative-BB} says that the right-hand
side is~$\ev(\Gmu)$.  Theorem~\ref{t:perturbative-CLT} implies that
$\ev(\Gmu)$ equals the left-hand side, which is the same as in
Theorem~\ref{t:clt}.
\end{proof}

\begin{conj}\label{conj:minimizer-CLT}
Theorem~\ref{t:minimizer-CLT} holds after the word ``confined'' is
deleted and every symbol~$\fC$ is replaced by~$\BB$.  That is,
confinement in \eqref{eq:tH} from Remark~\ref{r:view-of-Upsilon} is
true.
\end{conj}

\begin{remark}\label{r:estimator}
Conjecture~\ref{conj:minimizer-CLT} could have implications for limit
theory of statistical estimators in singular geometric contexts.  In
smooth contexts, Theorem~\ref{t:minimizer-CLT} already suffices, since
the fluctuating cone is already a vector space by
Example~\ref{e:smooth-fluctuating-cone}.  Thus
Theorem~\ref{t:minimizer-CLT} demonstrates the power of thinking in
ways amenable to singular settings---in this case, with radial
variation along a Gaussian tangent field instead of the usual spatial
variation, as discussed in \cite{random-tangent-fields}---even when
the setting is~smooth.
\end{remark}

\vfill
\eject


\begin{thebibliography}{HMMN15}
\raggedbottom



\bibitem[AM23]{arya-mukherjee-2023}
Shreya Arya and Sayan Mukherjee,
\newblock personal communication on Brownian motion in stratified spaces, 2023.


\bibitem[BBI01]{BBI01}
Dmitri Burago, Yuri Burago, and Sergei Ivanov,
\newblock \emph{A course in metric geometry}, volume~33,
\newblock American Mathematical Soc., 2001.


\bibitem[BH13]{bridson2013metric}
Martin~R Bridson and Andr\'e Haefliger.
\newblock \emph{Metric spaces of nonpositive curvature}, volume 319,
\newblock Springer Science \& Business Media, 2013.

\bibitem[BHV01]{BHV01}
Louis~J Billera, Susan~P Holmes, and Karen Vogtmann,
\newblock \emph{Geometry of the space of phylogenetic trees},
\newblock Advances in Applied Mathematics~\textbf{27} (2001), no.\,4, 733--767.


\bibitem[BK01]{brin-kifer2001}
Michael Brin and Yuri Kifer,
\newblock \emph{Brownian motion, harmonic functions and hyperbolicity
  for Euclidean complexes},
\newblock Mathematische Zeitschrift \textbf{237} (2001), no.\,3, 421--468.



\bibitem[BL18]{barden-le2018}
Dennis Barden and Huiling Le,
\newblock \emph{The logarithm map, its limits and Fr\'echet means in orthant spaces},
\newblock Proceedings of the Londong Mathematical Society (3)
  \textbf{117} (2018), no.\,4, 751--789.

\bibitem[BLO13]{barden-le-owen2013}
Dennis Barden, Huiling Le, and Megan Owen,
\newblock \emph{Central limit theorems for Fr\'echet means in the
  space of phylogenetic trees},
\newblock Electronic J.\ of Probability~\textbf{18} (2013), no.\,25, 25\,pp.

\bibitem[BLO18]{barden-le-owen2018}
Dennis Barden, Huiling Le, and Megan Owen,
\newblock \emph{Limiting behaviour of Fréchet means in the space of phylogenetic trees},
\newblock Annals of the Institute of Statistical Mathematics~\textbf{70}
  (2013), no.\,1, 99--129.

\bibitem[Bog07]{bogachev2007}
Vladimir Bogachev,
\newblock \emph{Measure theory, Vol.~II},
\newblock Springer-Verlag, Berlin, 2007.


\bibitem[Bou05]{bouziane2005}
Taoufik Bouziane,
\newblock \emph{Brownian motion in Riemannian admissible complexes},
\newblock Illinois J. Math. \textbf{49} (2005), no.\,2, 559--580.

\bibitem[BP03]{bhattacharya-patrangenaru2003}
Rabi Bhattacharya and Vic Patrangenaru,
\newblock \emph{Large sample theory of intrinsic and extrinsic sample
  means on manifolds: I},
\newblock Annals of Statistics~\textbf{31} (2003), no.\,1, 1--29.

\bibitem[BP05]{bhattacharya-patrangenaru2005}
Rabi Bhattacharya and Vic Patrangenaru,
\newblock \emph{Large sample theory of intrinsic and extrinsic sample
  means on manifolds: II},
\newblock Annals of Statistics~\textbf{33} (2005), no.\,3, 1225--1259.

\bibitem[BP23]{buet-pennec2023}
Blanche Buet and Xavier Pennec,
\newblock \emph{Flagfolds},
\newblock preprint.  \textsf{arXiv:math.CA/2305.10583}

\bibitem[BS13]{bonnans-shapiro2013}
J~Fr\'ed\'eric Bonnans and Alexander Shapiro,
\newblock \emph{Perturbation analysis of optimization problems},
\newblock Springer Science \& Business Media, 2013.


\bibitem[Dem12]{demailly-2012}
Jean-Pierre Demailly,
\newblock \emph{Analytic methods in algebraic geometry}, Surv. Modern Math., volume~1,
\newblock International Press, Somerville, MAHigher Education Press, Beijing, 2012.

\bibitem[Dur19]{Dur19}
Rick Durrett,
\newblock \emph{Probability: theory and examples}, volume~49,
\newblock Cambridge University Press, 2019.


\bibitem[EH19]{eltzner-huckemann2019}
Benjamin Eltzner and Stephan~F Huckemann,
\newblock \emph{A smeary central limit theorem for manifolds with
  application to high-dimensional spheres},
\newblock Annals of Statistics~\textbf{47} (2019), no.\,6, 3360--3381.

\bibitem[FL$^+$13]{feragen-lo-et-al2013}
Aasa Feragen, Pechin Lo, Marleen de Bruijne, Mads Nielsen, and Fran\,cois Lauze,
\newblock \emph{Toward a theory of statistical tree-shape analysis},
\newblock IEEE Trans. Pattern Anal. Mach. Intell.,  \textbf{35}
  (2013), no.\,8, 2008--2021.
  
\bibitem[Fr\'e48]{frechet1948}
Maurice Fr\'echet,
\newblock \emph{Les \'el\'ements al\'eatoires de nature quelconque
  dans un espace distanci\'e},
\newblock Annales de l'institut Henri Poincar\'e~\textbf{10} (1948), no.\,4, 215--310.

\bibitem[GJS17]{groisser-jung-schwartzman2017}
David Groisser, Sungkyu Jung, and Armin Schwartzman,
\newblock \emph{Geometric foundations for scaling-rotation statistics on symmetric
positive definite matrices: minimal smooth scaling-rotation curves in low dimensions},
\newblock Electron. J. Stat. \textbf{11} (2017), no.\,1, 1092--1159.
 

\bibitem[GM88]{goresky-macpherson1988}
Mark Goresky and Robert MacPherson,
\newblock \emph{Stratified Morse theory}, volume 14,
\newblock Ergebnisse der Mathematik und ihrer Grenzgebiete (3) [Results in
  Mathematics and Related Areas (3)], Springer-Verlag, 1988.

\bibitem[HE20]{huckemann-eltzner2020}
Stephan F. Huckemann and Benjamin Eltzner,
\newblock \emph{Data analysis on nonstandard spaces},
\newblock Wiley Interdiscip. Rev. Comput. Stat. \textbf{13} (2021),
  no.\,3, Paper No. e1526, 19 pp.
\newblock doi:\href{https://doi.org/10.1002/wics.1526}{10.1002/wics.1526}


\bibitem[HHL$^+$13]{hotz-et-al.2013}
Thomas Hotz, Stephan Huckemann, Huiling Le, J.S.\,Marron,
  Jonathan\,C.\,Mattingly, Ezra Miller, James Nolen, Megan Owen, Vic
  Patrangenaru, and Sean Skwerer,
\newblock \emph{Sticky central limit theorems on open books},
\newblock Annals of Applied Probability~\textbf{23} (2013), no.\,6, 2238--2258.

\bibitem[Hir64]{hironaka1964}
Heisuke Hironaka,
\newblock \emph{Resolution of singularities of an algebraic variety
  over a field of characteristic zero. I, II.}
\newblock Ann. of Math. (2) \textbf{79} (1964), 109--203; ibid. (2)
  \textbf{79} (1964), 205--326.


\bibitem[HMMN15]{kale-2015}
Stephan Huckemann, Jonathan Mattingly, Ezra Miller, and James Nolen,
\newblock \emph{Sticky central limit theorems at isolated hyperbolic
  planar singularities},
\newblock Electronic Journal of Probability~\textbf{20} (2015), 1--34.

\bibitem[HMPS20]{harms-michor-pennec-sommer2020}
Philipp Harms, Peter W. Michor, Xavier Pennec, and Stefan Sommer,
\newblock \emph{Geometry of sample spaces},
\newblock preprint, 2020.  \textsf{arXiv:math.ST/2010.08039v3}

\bibitem[Hol03]{holmes2003}
Susan Holmes,
\newblock \emph{Statistics for phylogenetic trees},
\newblock Theoretical Population Biology~\textbf{63} (2003), no.\,1, 17--32.

\bibitem[Hsu88]{hsu-1988}
Pei Hsu,
\newblock \emph{Brownian motion and Riemannian geometry},
\newblock Contemporary Mathematics \textbf{73} (1988), 95--104.

\bibitem[HTDL13]{hartley-trumpf-dai-li2013}
Richard Hartley, Jochen Trumpf, Yuchao Dai, and Hongdong Li,
\newblock \emph{Rotation averaging},
\newblock International journal of computer vision~\textbf{103}
  (2013), no.\,3, 267--305.

\bibitem[Jen69]{jennrich1969}
Robert I.\,Jennrich,
\newblock \emph{Asymptotic properties of non-linear least squares estimators},
\newblock Annals of Mathematical Statistics \textbf{40} (1969), 633--643.

\bibitem[Kal97]{kallenberg1997}
Olav Kallenberg,
\newblock \emph{Foundations of Modern Probability}, volume~2,
\newblock Springer, 1997.

\bibitem[KBCL99]{kendall-barden-carne-le99}
David G.\,Kendall, Dennis Barden, Thomas K.\,Carne, and Huiling Le,
\newblock \emph{Shape and Shape Theory},
\newblock Wiley Series in Probability and Statistics,
  Wiley \& Sons, Ltd., Chichester, 1999.




\bibitem[LB14]{le-barden14}
Huiling Le and Dennis Barden, \emph{On the measure of the cut locus of
  a Fr\'echet mean}, Bulletin London Mathematical Society \textbf{46}
  (2014), no.~4, 698--708.

\bibitem[Le01]{le2001}
Huiling Le,
\newblock \emph{Locating Fr\'echet means with application to shape spaces},
\newblock Advances in Applied Probability~\textbf{33} (2001), no.\,2, 324--338.

\bibitem[LGNH21]{lueg-garba-nye-huckemann2021}
Jonas Lueg, Maryam Garba, Tom Nye, and Stephan Huckemann,
\newblock \emph{Wald space for phylogenetic trees},
\newblock in Geometric Sci.\ of Inform., Lect.\ Notes in
  Comp.\ Sci.\ \textbf{12829} (2021),~710--717.


\bibitem[LSTY17]{lin-sturmfels-tang-yoshida2017}
Bo Lin, Bernd Sturmfels, Xiaoxian Tang, and Ruriko Yoshida,
\newblock \emph{Convexity in tree spaces},
\newblock SIAM J. Discrete Math. \textbf{31} (2017), no.\,3, 2015--2038.

\bibitem[Mat22]{mathieu2022}
Timoth\'ee Mathieu,
\newblock \emph{Concentration study of M-estimators using the inﬂuence function}
\newblock Electronic Journal of Statistics~\textbf{16} (2022), 3695--3750.
\newblock doi: \href{https://doi.org/10.1214/22-EJS2030}{10.1214/22-EJS2030}


\bibitem[MD21]{ooda-2021}
James Stephen Marron and Ian L.\,Dryden,
\newblock \emph{Object oriented data analysis},
\newblock Chapman and Hall/CRC Press, New York, 2021.
\newblock doi: \href{https://doi.org/10.1201/9781351189675}%
                    {10.1201/9781351189675}

\bibitem[MMH11]{mileyko-mukherjee-harer2011}
Yuriy Mileyko, Sayan Mukherjee, and John Harer,
\newblock \emph{Probability measures on the space of persistence diagrams},
\newblock Inverse Problems \textbf{27.12} (2011): 124007.

\bibitem[MMT23a]{shadow-geom}
Jonathan Mattingly, Ezra Miller, and Do Tran,
\newblock \emph{Shadow geometry at singular points of $\CAT(\kappa)$ spaces},
\newblock preprint, 2023.

\bibitem[MMT23b]{tangential-collapse}
Jonathan Mattingly, Ezra Miller, and Do Tran,
\newblock \emph{Geometry of measures on smoothly stratified metric spaces},
\newblock preprint, 2023.

\bibitem[MMT23c]{random-tangent-fields}
Jonathan Mattingly, Ezra Miller, and Do Tran,
\newblock \emph{A central limit theorem for random tangent fields on stratified spaces},
\newblock preprint, 2023.


\bibitem[MOP15]{centroids}
Ezra Miller, Megan Owen, and Scott Provan,
\newblock \emph{Polyhedral computational geometry for averaging metric
  phylogenetic trees},
\newblock Advances in Applied Math.\ \textbf{15} (2015), 51--91.
\newblock doi: \href{http://dx.doi.org/10.1016/j.aam.2015.04.002}%
                    {10.1016/j.aam.2015.04.002}

\bibitem[Nad90]{nadel-1990}
Alan Michael Nadel, 
\newblock \emph{Multiplier ideal sheaves and Kähler-Einstein metrics
  of positive scalar curvature},
\newblock Annals of Mathematics~(2) \textbf{132} (1990), no.\,3, 549--596.



\bibitem[NM94]{newey-mcFadden1994}
Whitney K.\,Newey and Daniel McFadden,
\newblock \emph{Large sample estimation and hypothesis testing},
\newblock in Handbook of Econometrics, vol.\,IV, Handbooks in
  Economics~\textbf{2} (1994), 2111--2245.

\bibitem[NW14]{nye-white2014}
Tom M.\,W.\,Nye and Michael C.\,White,
\newblock \emph{Diffusion on some simple stratified spaces},
\newblock J. Math. Imaging Vision \textbf{50} (2014), no.\,1--2, 115--125.

\bibitem[Nye20]{nye2020}
Tom M.\,W.\,Nye, 
\newblock \emph{Random walks and Brownian motion on cubical complexes}
\newblock Stochastic Process. Appl. \textbf{130} (2020), no.\,4, 2185--2199.




\bibitem[PSD00]{pritchard-stephens-donnelly2000}
Jonathan~K Pritchard, Matthew Stephens, and Peter Donnelly,
\newblock \emph{Inference of population structure using multilocus genotype data},
\newblock Genetics~\textbf{155} (2000), no.\,2, 945--959.

\bibitem[PSF20]{pennec-sommer-fletcher2020}
Xavier Pennec, Stefan Sommer, and Tom Fletcher (eds.),
\newblock \emph{Riemannian geometric statistics in medical image analysis},
\newblock Academic Press, 2020.
\newblock doi: \href{https://doi.org/10.1016/B978-0-12-814725-2.00012-1}%
                    {10.1016/B978-0-12-814725-}
               \href{https://doi.org/10.1016/B978-0-12-814725-2.00012-1}%
                    {2.00012-1}

\bibitem[RW09]{rockafellar-wets2009}
R~Tyrrell Rockafellar and Roger J-B Wets,
\newblock \emph{Variational analysis}, volume 317,
\newblock Springer Science \& Business Media, 2009.

\bibitem[Stu03]{sturm2003}
Karl-Theodor Sturm,
\newblock \emph{Probability measures on metric spaces of nonpositive curvature},
\newblock in Heat kernels and analysis on manifolds, graphs, and metric
  spaces: lecture notes from a quarter program on heat kernels, random
  walks, and analysis on manifolds and graphs, Contemporary
  Mathematics \textbf{338} (2003), 357--390.

\bibitem[Tra20]{Tra20}
Do~Tran,
\newblock \emph{Sampling from stratified spaces},
\newblock PhD thesis, Duke University, 2020.

\bibitem[EHT21]{tran-eltzner-huckemann2021}
Do~Tran, Benjamin Eltzner, and Stephan Huckemann,
\newblock \emph{Smeariness begets finite sample smeariness},
\newblock in: F.\,Nielsen and F.\,Barbaresco (eds), Geometric Science of
  Information (GSI 2021), Lecture Notes in Computer Science \textbf{12829},
  (2021), 29--36.
\newblock doi: \href{https://doi.org/10.1007/978-3-030-80209-7_4}%
                    {10.1007/978-3-030-80209-7\textunderscore{}4}

\bibitem[Van00]{Van00} 
Aad Van der Vaart, 
\newblock \emph{Asymptotic statistics}, volume~3,
\newblock Cambridge university press, 2000.

\bibitem[VW13]{VW13}
Aad van der Vaart and Jon Wellner,
\newblock \emph{Weak convergence and empirical processes: with
  applications to statistics},
\newblock Springer Science \& Business Media, 2013.


\bibitem[Wil19]{willis-2019}
Amy Willis,
\newblock \emph{Confidence sets for phylogenetic trees},
\newblock Journal of the American Statistical Association~\textbf{114}
  (2019), no.\,525, 235--244.

\bibitem[Yok16]{Yo17}
Takumi Yokota,
\newblock \emph{Convex functions and barycenter on CAT(1)-spaces of
  small radii},
\newblock Journal of the Mathematical Society of Japan~\textbf{68}
  (2016), no.\,3, 1297--1323.

\bibitem[Zie77]{ziezold1977}
Herbert Ziezold,
\newblock \emph{On expected figures and a strong law of large numbers
for random elements in quasi-metric spaces},
\newblock in Transactions of the Seventh Prague Conference on
  Information Theory, Statistical Decision Functions, Random Processes
  and of the Eighth European Meeting of Statisticians
  (Tech. Univ. Prague, Prague, 1974), Vol.\,A, pp. 591--602, Reidel,
  Dordrecht, 1977.

\bibitem[Zie95]{ziegler1995}
G\"unter~M.~Ziegler,
\newblock \emph{Lectures on polytopes},
\newblock Graduate Texts in Mathematics vol.~152, Springer--Verlag, New York, 1995.


\end{thebibliography}
\end{document}